\documentclass[11pt,letterpaper]{amsart}
\usepackage[T1]{fontenc}
\usepackage[utf8]{inputenc}
\usepackage{amsmath,amssymb,amsthm,mathrsfs,booktabs}
\usepackage[colorlinks=true,linkcolor=blue,citecolor=blue,urlcolor=blue]{hyperref}
\newcommand{\CC}{\mathbb C}
\newcommand{\C}{\mathbb C}
\newcommand{\PP}{\mathbb P}
\newcommand{\OO}{\mathcal O}
\newcommand{\ZZ}{\mathbb Z}

\newcommand{\0}{{\underline 0}}
\DeclareMathOperator{\Sing}{Sing}
\DeclareMathOperator{\Ind}{Ind}
\DeclareMathOperator{\pol}{pol}
\DeclareMathOperator{\ord}{ord}
\DeclareMathOperator{\Der}{Der}
\DeclareMathOperator{\Res}{Res}
\newtheorem{mainthm}{Theorem}

\newtheorem{letthm}{Theorem}
\newenvironment{lettered}[2][]{\renewcommand{\theletthm}{#2}%
 \if\relax\detokenize{#1}\relax\begin{letthm}\else\begin{letthm}[#1]\fi}{\end{letthm}}
\newtheorem*{thmCprime}{Theorem C$'$}
\newtheorem*{cor*}{Corollary}
\newtheorem{lema}{Lemma}[section]
\newtheorem{teo}[lema]{Theorem}
\newtheorem{prop}[lema]{Proposition}
\newtheorem{cor}[lema]{Corollary}
\newtheorem{conj}[lema]{Conjecture}
\theoremstyle{definition}
\newtheorem{remark}[lema]{Remark}
\newtheorem{exe}[lema]{Example}

\title{Generic vector fields on isolated complex hypersurface germs}
\author{Diogo da Silva Machado and Jos\'e Seade}
\date{}
\subjclass[2020]{Primary 32S65, 14B05; Secondary 37F75, 58K65, 32S25}
\keywords{isolated hypersurface singularities, holomorphic vector fields,
GSV index, Tjurina number, logarithmic vector fields, Thom--Sebastiani,
weighted homogeneous singularities}

\begin{document}

\hyphenation{ana-ly-tic}

\begin{abstract}
At a smooth point, a generic holomorphic vector field has a nondegenerate zero of Poincaré–Hopf index one; at an isolated analytic singularity, every holomorphic field vanishes. The GSV, radial, and homological indices extend to this setting the Poincaré–Hopf index; they depend on both the field and the germ. On an isolated complete intersection singularity, the GSV and homological indices coincide and differ from the radial index by the Milnor number.

By Bonatti and Gómez-Mont, fields minimizing such an index form an open dense set; we call them generic. The minimum index is an analytic invariant distinguishing germs with identical topology, Milnor and Tjurina numbers, and arithmetic and geometric genera. We study this minimum and its realizing fields.

For an isolated hypersurface singularity of dimension $n$, we prove the universal bound $K \geq 1 + (-1)^n\tau$ for the minimum GSV index, where $\tau$ is the Tjurina number, with equality precisely for weighted homogeneous (WH) germs, where the Euler field is generic. Beyond the WH case,  in odd dimensions, the excess is at least two, sharply; in even dimensions, excess two occurs, and we conjecture that excess one is impossible. Double suspension preserves $K$, extending these results and examples to all dimensions. Geometrically, a generic field is one that turns the least, counted with sign, around the radial direction along the curve of the link where it is a complex multiple of the radial field; for curves this recovers a Poincaré–Bendixson formula counting petals. These minima also give global obstructions to holomorphic vector fields on compact singular varieties.

For a fixed germ, however, the minimum index may lie far above the universal bound. We determine these minima for reduced curves and normal surfaces of arbitrary codimension, and obtain a general formula for hypersurface singularities in all dimensions. 
\end{abstract}

\maketitle

\section*{Introduction}

Let $(V,\underline{0})=\{f=0\}\subset(\mathbb C^{n+1},\underline{0})$ be an isolated hypersurface singularity. Among
holomorphic vector fields on $V$ with an isolated zero at the origin, consider the least possible
GSV index, denoted by $K(V)$. By Bonatti and G\'omez-Mont, the minimum exists and is attained on
an open dense set of fields. Thus $K(V)$ is the index of a generic field in a precise sense. Our
purpose is to understand both this analytic invariant and the geometry of the fields that realize
it.

Write $\tau=\dim_{\mathbb C}\mathcal O_{n+1,\underline{0}}/(f,J(f))$ for the Tjurina number. Our first main
result is a sharp universal lower  bound:

\begin{mainthm}\label{thmA}
For every holomorphic vector field $v$ on $V$ with an isolated zero,
\[
 \operatorname{Ind}_{\rm GSV}(v;V,\underline{0})\ge 1+(-1)^n\tau.
\]
Equality holds for every field admitting an ambient extension with invertible linear part; in
particular, for weighted homogeneous germs the Euler field gives
$K(V)=1+(-1)^n\tau=1+(-1)^n\mu$. Moreover,
\[
 K(V)=1+(-1)^n\tau
 \quad\Longleftrightarrow\quad
 (V,\underline{0})\text{ is weighted homogeneous}.
\]
\end{mainthm}
This characterization was conjectured in the first version of this work \cite{MS}, where part of
it was already proved. The remaining implication, from equality to weighted homogeneity for even
$n$, was then proved by Liu and Zhang \cite[Corollary~5.3]{LZeq}, who thereby settled our
conjecture.

\medskip
In the weighted homogeneous  case the Euler field is generic. For a germ that is not weighted homogeneous, let
\[
 e(V)=K(V)-\bigl(1+(-1)^n\tau\bigr)
\]
be the excess, and for a single field put $e(v)=\Ind_{\rm GSV}(v;V,\underline{0})-\bigl(1+(-1)^n\tau\bigr)$. The
\emph{defect} $D(V)=K(V)-\chi(F)$, where $F$ denotes the local Milnor fiber of $V$ at $\underline{0}$, measures instead the distance to the Euler-field value $\chi(F)$;
on surfaces it is the minimum logarithmic index of Theorem~\ref{thmC}, and it can be arbitrarily
negative. We prove that $e(V)\ge2$ in odd dimensions and give sharp examples in every odd
dimension. In even dimensions we prove $e(V)\ge1$, show that $e(V)=2$ occurs, and obtain
restrictions on the linear part of a field with $e=1$. Whether $e=1$ can occur is the remaining
gap conjecture (Theorem~\ref{thmE} and Conjecture~\ref{conjE}).

\medskip We then focus on  the minimizing (or generic)  fields beyond weighted homogeneity. For a reduced
curve germ this is done in full generality, with no restriction on the codimension or the number of branches.
The normalization records the minimum radial index branch by branch. We use the radial index: if $\alpha_i$
is the least order of a lifted derivation on the $i$th branch, then
\[
 K_{\rm rad}(C)=\sum_i\alpha_i-r+1,
\]
where $r$ is the number of branches. For ICIS curves this gives
$K_{\rm GSV}(C)=\sum_i\alpha_i-2\delta(C)$ (Theorem~\ref{thmB}). 

For a normal surface, a resolution (with exceptional divisor $E$)
records the minimum through a divisorial cycle $Z_V$ and a residual length $N_V$:
\[
 K_{\rm rad}(V)=1-\chi(E)-Z_V\cdot(Z_V+\kappa_X)+N_V
\]
(Theorem~\ref{thm:optimal-surface}). This holds for all normal surface singularities, not only complete intersection germs. 
These descriptions identify the minimizing stratum through
leading orders and vanishing data, without choosing a distinguished field. 

From a geometric and dynamic viewpoint, on curves the minimum
has a description  in the spirit of Poincar\'e and Bendixson:  the radial
index is one plus half the number of petals of the real flow, so a generic field is transversal to
the link exactly when the curve is weighted homogeneous (Section~\ref{ss:petals}). On surfaces the
petals reappear transversally to the dicritical exceptional curves. The logarithmic-index
formula in Theorem~\ref{thmC} unifies the curve, surface, and hypersurface viewpoints.

\medskip In every dimension, generic fields admit a simple description on the link $L$. Write a field
on $L$ as $v=c\,R+w$, where $R$ is the radial field, $c=\langle v,z\rangle/|R|^2$, and $w$ lies in
the complex tangent hyperplanes of $L$. On the curve $\Sigma(v)\subset L$ where $w=0$, that is, where
$v$ is a complex multiple of $R$, the function $c$ does not vanish, and
\[
 \Ind_{\rm rad}(v;V,\underline{0})-1=\sum_{\gamma\subset\Sigma(v)}\operatorname{wind}\bigl(c|_\gamma,\underline{0}\bigr)
\]
(Theorem~\ref{thmF}). Thus \emph{a generic field is one that turns the least, counted with sign, around the radial
direction along the curve of the link where it is a complex multiple of the radial field}. The Euler field of a weighted
homogeneous germ does not turn at all; on curves the turns are the petals.

\medskip We notice that the minimum $K$ detects analytic information missed by classical invariants. For instance
consider the plane branches
\[
 C_1:(x,y)=(t^4,t^9+t^{10}),\qquad
 C_2:(x,y)=(t^4,t^9+t^{11}).
\]
They have the same semigroup $\langle4,9\rangle$, hence the same embedded topological type, and
\[
 \mu(C_1)=\mu(C_2)=24,\qquad
 \tau(C_1)=\tau(C_2)=21,\qquad
 \delta(C_1)=\delta(C_2)=12,
\]
yet $K(C_1)=-15$ and $K(C_2)=-14$. The common  lower  bound given by Theorem \ref{thmA} is $1-\tau=-20$. Thus the minimum
contains analytic information that is not determined by these classical invariants. The same
happens for surfaces: the germs
\[
 z^2+x^5+y^6+x^2y^4,\qquad z^2+x^5+y^6+x^3y^3
\]
have the same topology, $\mu=20$, $\tau=18$, and the same geometric and arithmetic genera, but
$K=22$ and $K=23$ (Proposition~\ref{prop:same-invariants}).

\medskip To pass to higher dimensions we prove that the
double suspension preserves $K$, $\mu$, $\tau$, and the excess (Theorem~\ref{thmD}); in
particular, the sharp odd-dimensional examples propagate to every odd dimension. We also show with an example that a single suspension does not work for this purpose.

\medskip

 Finally, we show that the 
local minima give Poincar\'e--Hopf bounds for compact singular varieties and explicit obstructions
for projective hypersurfaces.

\medskip
The paper is organized around these results. Sections~\ref{sec:prelim} and~\ref{sec:generic}
set up the indices and genericity. Sections~\ref{sec:bound} and~\ref{sec:gap} prove the lower
bound and the gap results. Sections~\ref{sec:optimal-curves} and~\ref{ss:surfaces} describe the
minimizing fields on curves and surfaces. Section~\ref{sec:periodic} treats higher dimensions: the
index on a resolution through the logarithmic index, and the double suspension; the
final sections give the open questions and global applications.

\section{Preliminaries}\label{sec:prelim}

Throughout, $R=\OO_{n+1,\0}$ denotes the ring of germs of holomorphic functions at the origin of
$\CC^{n+1}$, and $\mathfrak m$ its maximal ideal. For an $R$-module $M$ of finite length we write
$\ell(M)=\dim_\CC M$. Let $f\in\mathfrak m^2$ be a reduced germ with an isolated critical point at
$\0$, and let
\[
(V,\0)=\{f=0\}\subset(\CC^{n+1},\0)
\]
be the hypersurface germ it defines. Thus $V$ has dimension $n$ and, after shrinking, is smooth
outside the origin. We write $J(f)=(f_{z_1},\ldots,f_{z_{n+1}})$ for the Jacobian ideal and
$\OO_V=R/(f)$ for the local ring of $V$.

\medskip
Choose $\varepsilon>0$ small and $0<|t|\ll\varepsilon$. The
\emph{Milnor fibre} of $f$ is $F=f^{-1}(t)\cap\mathbb B_\varepsilon$. By Milnor's fibration
theorem \cite{mimi}, $F$ is a smooth manifold with boundary, homotopy equivalent to a bouquet of
$\mu$ spheres of dimension $n$, where
\[
\mu=\ell\bigl(R/J(f)\bigr)
\]
is the \emph{Milnor number}. Hence $\chi(F)=1+(-1)^n\mu$. The boundary $\partial F$ is
diffeomorphic to the \emph{link} $L=V\cap\mathbb S_\varepsilon$. The \emph{Tjurina number}
\[
\tau=\ell\bigl(R/(f,J(f))\bigr)
\]
is the dimension of the base of the semiuniversal deformation of $(V,\0)$; it depends only on the
analytic type of the germ \cite{ggel}. Since $J(f)\subset(f,J(f))$, we have $\tau\le\mu$.

\medskip
The germ $(V,\0)$ is \emph{weighted homogeneous} if, in suitable
coordinates, there are positive rational weights $w_1,\ldots,w_{n+1}$ such that $f$ is a
polynomial with $f(\lambda^{w_1}z_1,\ldots,\lambda^{w_{n+1}}z_{n+1})=\lambda f(z)$. Then the
\emph{Euler field}
\[
\xi_w=\sum_iw_iz_i\,\partial_{z_i}
\]
satisfies $\xi_w(f)=f$, so it is tangent to $V$. By Saito's theorem \cite{saito}, the following
are equivalent: $(V,\0)$ is weighted homogeneous; $f\in J(f)$; and $\mu=\tau$. The simplest
example is the cusp $f=x^2+y^3$, with weights $(\tfrac12,\tfrac13)$ and $\mu=\tau=2$.

\medskip
A holomorphic vector field on $(V,\0)$ is a $\CC$-derivation of
$\OO_V$. Lifting the images of the coordinates, every such field $v$ is the restriction of an
ambient field $\widetilde v=\sum_iv_i\partial_{z_i}$ that is tangent to $V$, that is,
\[
\widetilde v(f)=kf\qquad\text{for some }k\in R,
\]
called the \emph{cofactor}. We call $\widetilde v$ an \emph{ambient extension} of $v$, and write
$\Der(-\log f)$ for the $R$-module of ambient fields tangent to $V$. Two extensions of the same
field differ by a field of the form $f\,u$. Since $f\in\mathfrak m^2$, all extensions have the same
linear part $D\widetilde v(\0)$, which we call the linear part of $v$. The origin, being the only
singular point, is preserved by every derivation, so $v$ vanishes at $\0$ and the components $v_i$
lie in $\mathfrak m$.

The cofactors form the ideal $(J(f):f)=\{k\in R:\,kf\in J(f)\}$. Multiplication by $f$ identifies
$R/(J(f):f)$ with $(f,J(f))/J(f)$, so
\[
\ell\bigl(R/(J(f):f)\bigr)=\mu-\tau .
\]
In particular the cofactor ideal is the whole ring exactly in the weighted homogeneous case, where
the Euler field has cofactor $1$. The fields with cofactor $0$ are the combinations of the
\emph{Hamiltonian fields} $f_{z_i}\partial_{z_j}-f_{z_j}\partial_{z_i}$, because the partial
derivatives of $f$ form a regular sequence. Consequently $\Der(-\log f)$ is generated by the
Hamiltonian fields, the fields $f\partial_{z_i}$, and one field for each generator of the cofactor
ideal.

\medskip
 Let $v$ be a vector field on $V$ with an isolated zero at $\0$. On a manifold, the
Poincar\'e--Hopf index of a holomorphic field at an isolated zero is the local degree of its
components, $\ell\bigl(\OO_{n,p}/(v_1,\ldots,v_n)\bigr)$; in particular it is positive. On the
singular germ, the \emph{GSV index} \cite{GSV,Seade0} is defined as follows. The field $v$ does not
vanish on the link $L$. Let $F$ be the Milnor fiber, since $\partial F$ and $L$ are isotopic in the sphere $\mathbb
S_\varepsilon$, it determines, uniquely up to homotopy, a nonvanishing field tangent to $F$ along
$\partial F$. Extend this field to $F$ with isolated zeros; then
$\Ind_{\rm GSV}(v;V,\0)$ is the sum of their Poincar\'e--Hopf indices. It does not depend on the
choices. If $v$ is transversal to the link, for instance the Euler field of a weighted homogeneous
germ, the extension can be taken pointing outward along $\partial F$, and
$\Ind_{\rm GSV}(v;V,\0)=\chi(F)$. For hypersurfaces the GSV index coincides with the homological
index of G\'omez-Mont \cite{a08}. This is the Euler characteristic of the Koszul-type complex
\[
0\to\Omega^n_{V,\0}\xrightarrow{\ \iota_v\ }\Omega^{n-1}_{V,\0}\xrightarrow{\ \iota_v\ }\cdots
\xrightarrow{\ \iota_v\ }\Omega^1_{V,\0}\xrightarrow{\ \iota_v\ }\OO_{V,\0}\to0
\]
of K\"ahler differentials of $\OO_V$, starting in degree $n$, whose differentials are the
contractions by the vector field; its homology is finite-dimensional when $v$ has an isolated
zero.

G\'omez-Mont computed the homological index, and hence the GSV index, from the data of an
ambient extension \cite[Theorem~1]{a08}. Every field on $V$ with an isolated zero has an ambient
extension
$\widetilde v=(v_1,\ldots,v_{n+1})$ with an isolated zero in $\CC^{n+1}$; it suffices to add
$f\,u$ for a generic constant vector $u$. Put $I=(v_1,\ldots,v_{n+1})\subset\mathfrak m$, and let
$k$ be the cofactor. Then
\begin{equation}\label{eq88}
\Ind_{\rm GSV}(v;V,\0)=\ell\bigl(R/(f,I)\bigr)-\tau,\qquad n\text{ odd},
\end{equation}
and {\small
\begin{equation}\label{eq888}
  \Ind_{\rm GSV}(v;V,\0)=\ell(R/I)-\ell\bigl(R/(k,I)\bigr)+\tau
=\ell\bigl(R/(I:k)\bigr)+\tau,\qquad n\text{ even}.
\end{equation}}
The second expression in \eqref{eq888} follows from the isomorphism $R/(I:k)\cong(I,k)/I$ induced
by multiplication by $k$. Formula \eqref{eq88} also holds for reduced plane curves with several
branches. The length $\ell(R/(f,I))$ is the Tjurina number of the field \cite{Cano,Arturo}.

For example, if $f$ is weighted homogeneous, the Euler field has $I=\mathfrak m$ and cofactor $1$.
Then \eqref{eq88} gives $1-\tau=1-\mu=\chi(F)$ for $n$ odd, and \eqref{eq888} gives
$1+\tau=1+\mu=\chi(F)$ for $n$ even, in agreement with the topological computation above. For the
cusp $x^2+y^3$ this value is $-1$.

\medskip
The \emph{radial index} \cite{ASV,EG1} is reminiscent of the indices in \cite{KT, Sch1}. It
 compares $v$ with the outward radial field on $V$: it is
$1+d(v)$, where $d(v)$ is the sum of the Poincar\'e--Hopf indices of a generic field on
$V\cap(\mathbb B_\varepsilon\setminus\mathbb B_\delta)$, $0<\delta<\varepsilon$, that agrees with
$v$ near the outer boundary and with the radial field near the inner one. It is defined on every
reduced germ with an isolated singularity, without any hypothesis on the equations, and for
isolated hypersurfaces and complete intersections
\[
\Ind_{\rm GSV}(v;V,\0)=\Ind_{\rm rad}(v;V,\0)+\chi(F)-1
\]
\cite{BarSeaSuw}. Both indices satisfy conservation of number: under a small perturbation of $v$,
the index at $\0$ equals the sum of the indices of the perturbed field at the zeros that bifurcate
from $\0$. We refer to \cite{BarSeaSuw} for a thorough account on indices.

\section{Genericity and minimizing fields}\label{sec:generic}

\hyphenation{chan-ging}

Bonatti and G\'omez-Mont proved that, on any complex analytic space with an isolated singularity,
the fields of minimum index form an open dense set. Their index is defined through a manifold $W$
bounding the link, and changing $W$ shifts the index of every field by the same integer
\cite[Section~1]{bonatti}. Their theorem therefore applies to the indices we use. We use two of
them. On isolated hypersurface singularities we use the GSV index, which G\'omez-Mont's formulas
make computable. On reduced curves of arbitrary embedding dimension and on normal surface
singularities, where the GSV index is not available in general, we use the radial index.

\medskip
The topology must be made explicit. Fix a representative $V\subset\mathbb B_1$, smooth outside
$\0$, and for $r<1$ let $V_r=V\cap\mathbb B_r$. Let $\Theta_r$ be the space of holomorphic vector
fields on $V_r$ that extend continuously to $\overline{V_r}$, with the supremum norm; it is a
Banach space \cite[Theorem~2.2]{bonatti}. A germ of vector field on $(V,\0)$ is represented in
$\Theta_r$ for $r$ small, and we give the space of germs the inductive limit topology. Thus
$v_j\to v$ whenever, for some $r>0$, all the $v_j$ and $v$ are represented in $\Theta_r$ and
$v_j\to v$ uniformly on $V_r$. Let $\Theta$ denote the set of germs of fields with an isolated zero
at $\0$.

\begin{teo}[Bonatti--G\'omez-Mont \cite{bonatti}]\label{thm:BGM}
Let $(V,\0)$ be a complex analytic germ with an isolated singularity, and let $\Ind$ be the radial
index, the GSV index if $V$ is an isolated complete intersection, or the local Euler obstruction of
vector fields.
\begin{enumerate}
\item[(i)] $\Ind$ is bounded below on $\Theta$ \cite[Corollary~3.2]{bonatti}.
\item[(ii)] The set $\mathcal G\subset\Theta$ of fields at which $\Ind$ attains its minimum is open
and dense; for $r$ small, its representatives form an open dense subset of $\Theta_r$
\cite[Theorems~2.2, 2.6 and 3.1]{bonatti}.
\end{enumerate}
\end{teo}

Bonatti and G\'omez-Mont work with the index $\Ind_W$ defined by a manifold $W$ with $\partial W$
the link, and the radial field has $\Ind_W=\chi(W)$. By their Proposition~1.1, two indices of this
kind differ by a constant independent of the field, so $\Ind_W=\Ind_{\rm rad}+\chi(W)-1$. For an
isolated complete intersection one may take for $W$ the Milnor fibre $F$, and then $\Ind_W$ is the
GSV index. The local Euler obstruction of a vector field \cite{BarSeaSuw} is also of this kind:
it is the obstruction to extending the lift of the field as a section of the Nash bundle, which
coincides with the tangent bundle over the link, so the difference between two fields is again their
relative degree on the link, and
\[
\operatorname{Eu}_V(v,\0)=\Ind_{\rm rad}(v;V,\0)+\operatorname{Eu}_V(\0)-1 .
\]
Thus all these indices have the same minimizing fields, and it suffices to check (i) and (ii) for
one of them.
Bonatti and G\'omez-Mont do it for $\Ind_W$ \cite[Theorem~3.1]{bonatti}, comparing it with the
$V$-multiplicity, the colength of the ideal generated by the equations of $V$ and
the components of an ambient extension of the field: up to an additive constant the two agree on an
open dense set of fields, and the minimal $V$-multiplicity is attained on an open dense set. The
homological index equals the GSV index on complete intersections; we do not use it elsewhere.

\medskip
The constant called $K$ in \cite[Theorem~3.1]{bonatti} is normalized differently: there the index
is compared with the $V$-multiplicity $\ell(R/(f,I))$ of a field, which by \eqref{eq88} is
$\Ind_{\rm GSV}+\tau$ for $n$ odd.

\medskip
We call the fields in $\mathcal G$ \emph{generic}, and write
\[
K_{\rm GSV}(V)=\min_{v\in\Theta}\Ind_{\rm GSV}(v;V,\0),\qquad
K_{\rm rad}(V)=\min_{v\in\Theta}\Ind_{\rm rad}(v;V,\0)
\]
for the minimum GSV and radial indices. For an isolated hypersurface or complete intersection,
\[
K_{\rm GSV}(V)=K_{\rm rad}(V)+\chi(F)-1 .
\]
For reduced curves of arbitrary embedding dimension (Section~\ref{sec:optimal-curves}) and for
normal surface singularities (Section~\ref{ss:surfaces}) we compute $K_{\rm rad}$ and describe the
generic fields. When only hypersurfaces are involved and no confusion is possible we write $K=K_{\rm GSV}$, as
in the introduction.

\medskip
We shall use two further properties of the generic set. The first is stronger than density, and
the second is a finite determinacy statement.

\begin{prop}\label{prop:BGM-complements}
Let $(V,\0)$ be an isolated hypersurface singularity.
\begin{enumerate}
\item[(i)] If $v\in\Theta$ and $w\in\mathcal G$, then $v+tw\in\mathcal G$ for all but finitely
many $t\in\CC$.
\item[(ii)] For $v\in\Theta$ there is $N$ such that every field $v'$ with
$v'-v\in\mathfrak m^N\Der(\OO_V)$ has an isolated zero and the same index as $v$.
\end{enumerate}
\end{prop}

\begin{proof}
We first recall why $\mathcal G$ is open, since the argument is used below; this is conservation of
number. Let $v_j\to v$ uniformly on $V_r$, and fix $s<r$ such that $v$ has no zeros on
$\overline{V_s}\setminus\{\0\}$ close to $\partial V_s$. For $j$ large, $v_j$ has no zeros near
$\partial V_s$, so its zero set in $V_s$ is a compact analytic subset of a Stein space, hence
finite. Thus the zeros of $v_j$ in $V_s$ are $\0$, where every derivation vanishes, and finitely
many points $q$ of the smooth part, and
$\Ind_{\rm GSV}(v;V,\0)=\Ind_{\rm GSV}(v_j;V,\0)+\sum_q\Ind_{\rm PH}(v_j,q)$. The local indices of a
holomorphic field at isolated zeros on a manifold are positive, so
$\Ind_{\rm GSV}(v_j)\le\Ind_{\rm GSV}(v)$. If $v\in\mathcal G$, equality holds.

For (i), choose ambient extensions $\widetilde v,\widetilde w$ with isolated zeros. Then
$\widetilde v+t\widetilde w$ extends $v+tw$, with cofactor $k_v+tk_w$. The lengths in
\eqref{eq88} and \eqref{eq888} are lengths of ideals generated by germs depending linearly on $t$.
By the Nakayama argument in (ii), wherever such a length is at most $M$ it equals the length modulo
$\mathfrak m^{M+1}$, which is given by ranks of matrices polynomial in $t$. Hence each length is
constant, equal to its minimal value, outside a finite subset of $\CC$, and so is the index. For
$|t|$ large, $v+tw=t(w+t^{-1}v)$ lies in the open set $\mathcal G$, and multiplying by a nonzero
constant does not change the index. Therefore the constant value is the minimum.

For (ii), let $v'-v\in\mathfrak m^N\Der(\OO_V)$. Since $\Der(-\log f)\to\Der(\OO_V)$ is surjective
and the cofactor map is $R$-linear, there are extensions with $I'\equiv I$ and $k'\equiv k$ modulo
$\mathfrak m^N$. Let $Q$ be one of the ideals $I$, $(I,f)$, $(I,k)$ entering the formulas, with
$\mathfrak m^M\subset Q$ and $N>M$. Then $\mathfrak m^M\subset Q'+\mathfrak m^N\subset
Q'+\mathfrak m\cdot\mathfrak m^M$, so $\mathfrak m^M\subset Q'$ by Nakayama's lemma, and
$Q'=Q'+\mathfrak m^N=Q+\mathfrak m^N=Q$. All lengths agree.
\end{proof}

Letting $t\to0$ along the good values in (i) shows again that every $v\in\Theta$ is a limit of
generic fields, which recovers the density in Theorem~\ref{thm:BGM}(ii).

\begin{remark}\label{rem:topology}
{\bf (i)} Density fails for the $\mathfrak m$-adic (Krull) topology on germs. For a weighted
homogeneous germ of odd dimension the generic fields are exactly those with invertible linear part
(Lemma~\ref{lem001}(c)), and a coset $v+\mathfrak m^2\Der(\OO_V)$ with $D\widetilde v(\0)$ singular
contains none of them. This is why the topology of uniform convergence on a fixed representative is
the relevant one. Openness, on the other hand, holds for both topologies by
Proposition~\ref{prop:BGM-complements}(ii).

\smallskip\noindent{\bf (ii)} By Proposition~\ref{prop:BGM-complements}(i), $\mathcal G$ meets
every complex line through one of its points in a cofinite set. Consequently, if
$\xi_1,\ldots,\xi_s$ generate the fields on $V$ and $N$ is large, the fields $\sum_jh_j\xi_j$ with
$h_j$ polynomials of degree at most $N$ are generic outside a proper algebraic subset of the
coefficients. For curves (Theorem~\ref{thm:optimal-reduced-curves}) and in all our examples,
constant coefficients suffice.
\end{remark}

\section{The lower bound and its equality case}\label{sec:bound}

\begin{lema}\label{lem:nondeg}
Let $\widetilde v$ be a field tangent to $V$ with invertible linear part at $\0$, and
$\widetilde v(f)=kf$. If $k(\0)\ne0$ then $f\in\mathfrak mJ(f)$, and $(V,\0)$ is weighted homogeneous
\cite{saito}. If $k(\0)=0$, then $n$ is odd and $(V,\0)$ is an ordinary double point.
\end{lema}

\begin{proof}
If $k$ is a unit, $f=k^{-1}\sum_iv_if_{z_i}\in\mathfrak mJ(f)$. Suppose $k(\0)=0$. Since
$I=\mathfrak m$, write $k=\sum A_iv_i$ and put $G_i=f_{z_i}-A_if$; then $\sum_iv_iG_i=0$.
The $G_i$ form a regular sequence: otherwise their zero set contains a nonconstant arc $\gamma$
through $\0$, along which $h=f\circ\gamma$ satisfies $h'=\beta h$ with
$\beta=\sum_i(A_i\circ\gamma)\gamma_i'$ and $h(0)=0$; then $h\equiv0$, so $df\circ\gamma=0$,
contradicting the isolated critical point. Exactness of the Koszul complex of $G$ gives a
skew-symmetric matrix $M$ with $\widetilde v=MG$. Differentiating at $\0$,
$D\widetilde v(\0)=M(\0)\operatorname{Hess}_\0(f)$. Hence $M(\0)$ and the Hessian are invertible;
the first forces $n+1$ to be even, and the second makes $f$ a nondegenerate quadratic form by the
Morse lemma.
\end{proof}

\begin{lema}\label{lem001}
Let $v$ be a field with an isolated zero on $V$.
\begin{enumerate}
\item[(a)] $\Ind_{\rm GSV}(v;V,\0)\ge1+(-1)^n\tau$.
\item[(b)] If $v$ has an ambient extension with invertible linear part, equality holds in (a).
\item[(c)] If $n$ is odd, equality holds in (a) if and only if the ambient extensions of $v$ have
invertible linear part.
\end{enumerate}
\end{lema}

\begin{proof}
For $n$ odd, $\ell(R/(f,I))\ge1$ in \eqref{eq88}. For $n$ even we prove $k\notin I$. Suppose
$k=\sum A_iv_i$, and set $G_i=f_{z_i}-A_if$, so that $\sum v_iG_i=0$. As in the proof of
Lemma~\ref{lem:nondeg}, $G$ is a regular sequence, and exactness of the Koszul complexes of $G$ and
of $\widetilde v$ gives skew-symmetric matrices $M,M'$ with $\widetilde v=MG$ and $G=M'\widetilde v$.
Hence $(G)=I$, both lists are minimal generating systems of this complete intersection ideal, and
$M(\0)$ is invertible. Since $n+1$ is odd, this is impossible for a skew-symmetric matrix. Thus
$\ell(R/I)-\ell(R/(k,I))\ge1$, and \eqref{eq888} gives (a).

If the extension is nondegenerate, $I=\mathfrak m$. For $n$ odd, $(f,I)=\mathfrak m$ and the index
is $1-\tau$. For $n$ even, $\ell(R/(k,I))\in\{0,1\}$ and the value $1$ is excluded by $k\notin I$,
so the index is $1+\tau$. This proves (b). If $n$ is odd and the index is $1-\tau$, then
$(f,I)=\mathfrak m$; since $f\in\mathfrak m^2$, the $v_i$ span $\mathfrak m/\mathfrak m^2$, and
$I=\mathfrak m$ by Nakayama's lemma. This proves (c).
\end{proof}

\begin{remark}\label{rem001}
The length $\ell(R/(f,I))$ is the Tjurina number of the vector field \cite{Cano,Arturo}; in odd
dimension the GSV index is its difference with the Tjurina number of the hypersurface. For even
$n$, the weaker bound $\Ind_{\rm GSV}\ge\tau$ was obtained in \cite{DSM2}.
\end{remark}

\begin{proof}[Proof of Theorem~\ref{thmA}]
The inequality is Lemma~\ref{lem001}(a), and equality for nondegenerate fields is (b). In
coordinates in which $f$ is weighted homogeneous, the Euler field is tangent to $V$ with invertible
linear part, so $K=1+(-1)^n\tau$. If $n$ is odd and $K=1-\tau$, a generic field has a nondegenerate
extension by (c), and $(V,\0)$ is weighted homogeneous by Lemma~\ref{lem:nondeg}.
For even $n$, the converse follows from \cite[Corollary~5.3]{LZeq}: if $K=1+\tau$, the minimum is
attained by a field, and Liu--Zhang prove that attainment of this lower bound forces weighted
homogeneity.
\end{proof}

\begin{cor}\label{cor:K-odd}
If $n$ is odd, then $K=1-\tau$ if and only if $(V,\0)$ is weighted homogeneous. Equivalently,
$e(V)\ge1$ and $D(V)\ge\mu-\tau+1$ for every germ that is not weighted homogeneous.
Theorem~\ref{thm:odd-gap} improves these bounds to $e(V)\ge2$ and $D(V)\ge\mu-\tau+2$.
\end{cor}

\begin{remark}\label{rem:even-converse}
{\bf (i)} For $n$ even the converse of Lemma~\ref{lem001}(b) fails for individual fields. On
$f=xy+z^2$, with $\tau=1$, the field $x(1+z)\partial_x+y(z-1)\partial_y+z^2\partial_z$ satisfies
$\widetilde v(f)=2zf$ and has index $2-1+1=2=1+\tau$ by \eqref{eq888}, although its linear part
$\mathrm{diag}(1,-1,0)$ is singular.

\smallskip\noindent{\bf (ii)} For $n$ even, \eqref{eq888} shows that $K=1+\tau$ precisely when a
generic field has $\ell(R/(I:k))=1$, equivalently when $k$ spans the socle of the Artinian
complete intersection $A=R/I$. By \cite[Corollary~5.3]{LZeq}, this forces weighted homogeneity.
Consequently, on a non-weighted-homogeneous germ every field has
$\ell(R/(I:k))\ge2$ (otherwise its index would force $K=1+\tau$). Conjecture~\ref{conjE}
asserts the stronger bound $\ell(R/(I:k))\ge3$, equivalently that $e=1$ never occurs.
\end{remark}


\begin{remark}[Relation with \cite{MS2}]\label{rem:MS2}
In \cite{MS2} we prove that a reduced complex analytic germ, of arbitrary codimension and possibly
with non-isolated singularities, is weighted homogeneous if and only if, in suitable local
coordinates, it carries a holomorphic vector field vanishing at the origin and pointing strictly
outward along a small link. This is equivalent to the existence of a field whose linear part has
spectrum in an open half-plane bounded by a line through $0$. That theorem characterizes weighted
homogeneity by a \emph{single} field. Theorem~\ref{thmA} characterizes it by the \emph{minimal}
index, and the two statements interact differently in the two parities.

\smallskip\noindent{\bf (i)} For $n$ odd, Theorem~\ref{thmA} implies the transversality criterion
for isolated hypersurfaces. A field transversal to the link has GSV index  equal to Euler-Poincaré characteristic of the Milnor fiber $\chi(F)=1-\mu$, the
index of the radial field; for $n\ge3$ the link is connected, and for plane curves we take the
field pointing outward. The inequality of Theorem~\ref{thmA} then gives $1-\mu\ge1-\tau$, hence
$\mu=\tau$, and the germ is weighted homogeneous by \cite{saito}.

\smallskip\noindent{\bf (ii)} For $n$ even, the same argument gives only $1+\mu\ge1+\tau$, which
carries no information. In this parity the transversality criterion requires the methods of
\cite{MS2}, or for isolated hypersurfaces those of \cite{LZ}. Conversely, the criterion of
\cite{MS2} does not yield the even-dimensional case of Theorem~\ref{thmA}. A field realizing
$K=1+\tau$ need not have spectrum in a half-plane: the field of Remark~\ref{rem:even-converse}(i)
has index $1+\tau$ and linear part $\mathrm{diag}(1,-1,0)$. Nor does attaining the bound imply
transversality: when $\tau<\mu$, the value $1+\tau$ differs from the index $\chi(F)=1+\mu$ of a
field transversal to the link. The implication
$K=1+\tau\Rightarrow$ weighted homogeneous therefore rests on \cite[Corollary~5.3]{LZeq}.

\smallskip\noindent{\bf (iii)} For curves of arbitrary embedding dimension, the spectral criterion
of \cite{MS2} combines with the normalization formula of Theorem~\ref{thmB}; see
Corollary~\ref{cor:curve-wh}.

\end{remark}


\section{The gap above the lower bound}\label{sec:gap}

In this section, we state Theorem E, which yields a gap above the lower bound in the non-weighted homogeneous case. We begin by proving the necessary preliminary results. Throughout, $(V,\0)$ is not weighted
homogeneous. By Lemma~\ref{lem:nondeg}, the cofactor of every logarithmic field then satisfies
$k(\0)=0$; otherwise $k$ is a unit and $f\in\mathfrak mJ(f)$. For an ambient logarithmic field
$v=(v_1,\ldots,v_{n+1})$ with $v(f)=kf$ write $J_v=Dv$, $B=J_v(\0)$, $I=(v_1,\ldots,v_{n+1})$, and
\begin{equation}\label{eq:Pn}
 P_n(J_v,k)=[t^n]\,\frac{\det(1+tJ_v)}{1+tk}=\sum_{j=0}^n(-k)^{n-j}\sigma_j(J_v),
\end{equation}
where $\sigma_j$ is the $j$th elementary symmetric function of the eigenvalues and
$\sigma_0=1$. We use the field-level excess $e(v)$ defined in the introduction, so that $e(V)$ is
the minimum of $e(v)$.

\begin{lema}\label{lem:two-zero}
If $(V,\0)$ is not weighted homogeneous, the linear part $B$ of every logarithmic field has the
eigenvalue $0$ with algebraic multiplicity at least two. In particular $\sigma_n(B)=0$. If
$\ord f\ge3$, then $B$ is nilpotent. Moreover $\tau\ge3$.
\end{lema}

\begin{proof}
By the splitting lemma choose coordinates $(x,y)$ with
$f=g(x_1,\ldots,x_d)+y_1^2+\cdots+y_r^2$, $g\in\mathfrak m_x^3$, $d+r=n+1$. Since $V$ is not
weighted homogeneous, neither is $g$, so $d\ge2$: a Morse germ and a germ in one variable are
weighted homogeneous. Write $v=(a,b)$. Since $k(\0)=0$, the degree-two part of $v(f)=kf$ says that
the quadratic form $2\sum_jy_j(Bz)_{y_j}$ vanishes. The coefficients of the monomials $x_iy_j$ and
$y_iy_j$ give
\[
 B=\begin{pmatrix}B_x&C\\0&D\end{pmatrix},\qquad D+D^T=0 .
\]
Restricting $v(f)=kf$ to $y=0$ gives $a(x,0)g=k(x,0)g$, a logarithmic field of $g$ with linear
part $B_x$. Since $g$ has an isolated singularity of order at least three and is not weighted
homogeneous, $B_x$ is nilpotent by Schulze's theorem \cite[Proposition~4]{Schulze}. Hence the
characteristic polynomial of $B$ is divisible by $t^d$, with $d\ge2$. If $\ord f\ge3$, then
$r=0$ and $B=B_x$. Finally $(g,J(g))\subset\mathfrak m_x^2$, so $\tau(f)=\tau(g)\ge d+1\ge3$.
\end{proof}

\begin{lema}\label{lem:residue}
If $I$ is $\mathfrak m$-primary, then
\[
 \Ind_{\rm GSV}(v;V,\0)=\Res_\0\left[\frac{kP_n(J_v,k)\,dz_1\wedge\cdots\wedge dz_{n+1}}
 {v_1,\ldots,v_{n+1}}\right].
\]
In particular, $kP_n\in I$ implies $\Ind_{\rm GSV}(v;V,\0)=0$.
\end{lema}

\begin{proof}
After a linear change of coordinates, $(f,v_2,\ldots,v_{n+1})$ is a regular sequence. Suwa's
residue formula \cite[Theorem~1]{Suwa96} gives
$\Ind_{\rm GSV}(v;V,\0)=\Res_\0[f_{z_1}P_n\,dz/(f,v_2,\ldots,v_{n+1})]$. By the transformation law,
this equals $\Res_\0[v_1f_{z_1}P_n\,dz/(fv_1,v_2,\ldots,v_{n+1})]$. Since
$v_1f_{z_1}=kf-\sum_{i\ge2}v_if_{z_i}$, it also equals
$\Res_\0[fkP_n\,dz/(fv_1,v_2,\ldots,v_{n+1})]=\Res_\0[kP_n\,dz/(v_1,\ldots,v_{n+1})]$; compare
\cite[Proposition~4.3]{LZeq}.
\end{proof}

\begin{prop}\label{prop:mult-two}
If $(V,\0)$ is not weighted homogeneous, no ambient logarithmic field satisfies
$\ell(R/I)=2$.
\end{prop}

\begin{proof}
Suppose $\ell(R/I)=2$. A local algebra of length two has square-zero maximal ideal, so
$\mathfrak m^2\subset I$; in particular $f\in I$ and $k\,\mathfrak m\subset I$. Hence
$kP_n\equiv P_n(\0)k=\sigma_n(B)k\pmod I$, because $k(\0)=0$ in \eqref{eq:Pn}. By
Lemma~\ref{lem:two-zero}, $\sigma_n(B)=0$, so $kP_n\in I$ and the index vanishes by
Lemma~\ref{lem:residue}. On the other hand, if $n$ is odd, \eqref{eq88} gives
$\Ind_{\rm GSV}=\ell(R/(I,f))-\tau=2-\tau<0$; if $n$ is even, \eqref{eq888} gives
$\Ind_{\rm GSV}\ge\tau\ge3$. Both contradict the vanishing of the index.
\end{proof}

\begin{lema}\label{lem:representative}
Let $v$ be an ambient logarithmic field and suppose that $(I,f)$ has colength two. Then there are
constants $c_1,\ldots,c_{n+1}$ such that $(v_1+c_1f,\ldots,v_{n+1}+c_{n+1}f)=(I,f)$.
\end{lema}

\begin{proof}
Since $R/(I,f)$ has length two, there are coordinates with $(I,f)=(z_1,\ldots,z_n,z_{n+1}^2)$; in
particular $W=(I,f)/\mathfrak m(I,f)$ has dimension $n+1$. The classes of $v_1,\ldots,v_{n+1},f$
span $W$. Let $T:\CC^{n+1}\to W$ send the standard basis to the classes of the $v_i$, and let
$b\in W$ be the class of $f$. If $T$ is invertible, take $c=0$. Otherwise $\operatorname{rank}T=n$
and $b\notin\operatorname{im}T$; choose a linear form $c$ that does not vanish on $\ker T$. Then
$T+b\,c$ is injective: projecting to $W/\operatorname{im}T$ shows that a vector in its kernel lies in
$\ker c$, hence in $\ker T$. Nakayama's lemma gives the assertion.
\end{proof}

The field $v+f\sum_ic_i\partial_{z_i}$ is again logarithmic, with cofactor
$k+\sum_ic_if_{z_i}$, and restricts to the same field on $V$.

\begin{teo}[Odd-dimensional gap]\label{thm:odd-gap}
If $n$ is even and $(V,\0)$ is non-weighted-homogeneous, then every holomorphic vector field on $V$
with an isolated zero satisfies
\[
 \Ind_{\rm GSV}(v;V,\0)\ge3-\tau .
\]
Hence $K\ge3-\tau$ and $e(V)\ge2$. This includes reduced plane curves, with no assumption on
the number of branches or the multiplicity.
\end{teo}

\begin{proof}
Choose an ambient extension with $I$ $\mathfrak m$-primary and put $Q=\ell(R/(I,f))\ge1$. If $Q=1$,
then $(I,f)=\mathfrak m$, and since $f\in\mathfrak m^2$ the linear parts of the $v_i$ span
$\mathfrak m/\mathfrak m^2$. Thus $B$ is invertible, contradicting Lemma~\ref{lem:two-zero}. If
$Q=2$, Lemma~\ref{lem:representative} lets us replace the extension by
$v+f\sum_i c_i\partial_{z_i}$, without changing the field on $V$, so that its component ideal is
$(I,f)$. This extension has $\ell(R/I)=2$, contradicting Proposition~\ref{prop:mult-two}.
Hence $Q\ge3$, and \eqref{eq88} gives the bound. For a reduced plane curve, \eqref{eq88} and Lemma~\ref{lem:residue} hold as well, and
the same argument applies.
\end{proof}

The bound is sharp in every odd dimension: the branch $x^3+y^7+xy^5$ has $K=-8=3-\tau$
(Section~\ref{ss:E12curve}), and Theorem~\ref{thmD} carries this equality to its double
suspensions.

\begin{lema}\label{lem:order-Pn}
Let $n$ be even, $(V,\0)$ not weighted homogeneous, and $M=\ell(R/(I:k))$. If $P_n\in\mathfrak m^q$
for some $q\ge1$, then $M\ge q+1$, that is, $e(v)\ge q$.
\end{lema}

\begin{proof}
By \eqref{eq888} and Lemma~\ref{lem:two-zero}, $\Ind_{\rm GSV}=\tau+M>0$, so $kP_n\notin I$ by
Lemma~\ref{lem:residue}, that is, $P_n\notin L=(I:k)$. If $M\le q$, then
$\mathfrak m^q\subset\mathfrak m^M\subset L$, because $R/L$ is a local algebra of length $M$. This
would give $P_n\in L$.
\end{proof}

\begin{prop}\label{prop:even-rank}
Let $n$ be even and $(V,\0)$ not weighted homogeneous. Every field with an isolated zero satisfies
\[
 e(v)\ge n-\operatorname{rank}B .
\]
If moreover $\ord f\ge3$ and $e(v)=1$, then $B^{n-1}\ne0$; thus $B$ is nilpotent with Jordan type
$J_{n+1}(0)$ or $J_n(0)\oplus J_1(0)$.
\end{prop}

\begin{proof}
Let $r=\operatorname{rank}B$. In a $j\times j$ minor of $J_v=B+O(\mathfrak m)$, a nonzero
term can use at most $r$ constant columns; hence at least $j-r$ columns come from the
$O(\mathfrak m)$ part. Thus $\sigma_j(J_v)\in\mathfrak m^{\max(0,j-r)}$. Since $k\in\mathfrak m$, every term of \eqref{eq:Pn}
lies in $\mathfrak m^{n-j+\max(0,j-r)}\subset\mathfrak m^{\max(0,n-r)}$. If $r<n$, Lemma~\ref{lem:order-Pn}
gives $e(v)\ge n-r$. If $r\ge n$, then $n-r\le0$. In fact,
$e(v)\ge1$ for every field on a non-weighted-homogeneous germ: if some field had $e(v)=0$,
then $K$ would attain the lower bound, contrary to Theorem~\ref{thmA}.

Suppose $e(v)=1$, that is, $M=2$, and put $L=(I:k)$. Then $\mathfrak m^2\subset L$ and $L/\mathfrak
m^2$ is a hyperplane in $\mathfrak m/\mathfrak m^2$, so there are a nonzero vector $\eta\in\CC^{n+1}$
and an element $\varepsilon\in R/L$ with
$[h]_{R/L}=h(\0)+dh(\0)(\eta)\,\varepsilon$ for all $h\in R$. Since $I\subset L$, $B\eta=0$.
Multiplication by $k$ maps $R/L$ isomorphically onto $kA$, $A=R/I$. As
$P_n(\0)=\sigma_n(B)=0$, we get $[kP_n]_A=dP_n(\0)(\eta)\,[k\varepsilon]_A$, and
Lemma~\ref{lem:residue} forces $dP_n(\0)(\eta)\ne0$. Differentiating \eqref{eq:Pn} at $\0$, where
$k(\0)=0$, only the terms $j=n$ and $j=n-1$ survive:
\[
 dP_n(\0)(\eta)=d\sigma_n|_B(C_\eta)-dk(\0)(\eta)\,\sigma_{n-1}(B),\qquad
 C_\eta=D^2v(\0)[\eta,\cdot\,] .
\]
If $\ord f\ge3$, $B$ is nilpotent by Lemma~\ref{lem:two-zero}, so $\sigma_{n-1}(B)=0$, and
differentiating $\det(1+t(B+sC))$ at $s=0$ gives
$d\sigma_n|_B(C)=(-1)^{n-1}\operatorname{tr}(B^{n-1}C)$. Hence
$\operatorname{tr}(B^{n-1}C_\eta)\ne0$ and $B^{n-1}\ne0$. A nilpotent matrix of size $n+1$ with
$B^{n-1}\ne0$ has a Jordan block of size at least $n$.
\end{proof}

\begin{cor}\label{cor:zero-linear}
Let $n$ be even and $(V,\0)$ not weighted homogeneous. If a field of minimal index has zero linear
part, then $e(V)\ge n$; in particular $e(V)\ge2$. If $\ord f\ge3$ and the linear part of a field of
minimal index satisfies $B^{n-1}=0$, then $e(V)\ge2$.
\end{cor}

\begin{proof}
Apply Proposition~\ref{prop:even-rank} to a field of minimal index.
\end{proof}

For surfaces, Corollary~\ref{cor:zero-linear} contains Proposition~\ref{prop:gap-first-jet} and
removes its assumption on the cofactor. It also shows that, for a surface of multiplicity at least
three with $e=1$, the linear part of a generic field is a nonzero nilpotent matrix.

\begin{lettered}[Gap]{E}\label{thmE}
Let $(V,0)$ be a non-weighted-homogeneous isolated hypersurface singularity.
\begin{enumerate}
\item If $n$ is odd, every holomorphic vector field $v$ on $V$ with an isolated zero satisfies
\[
 \Ind_{\rm GSV}(v;V,0)\geq3-\tau.
\]
Consequently $e(V)\geq2$, and this bound is sharp in every odd dimension. This includes reduced
plane curves.
\item If $n$ is even, let $B$ be the linear part of an ambient extension of $v$. Every such field
satisfies
\[
 e(v)\geq n-\operatorname{rank}B.
\]
If moreover $\operatorname{ord}f\geq3$ and $e(v)=1$, then $B$ is nilpotent with a Jordan block of
size at least $n$.
\end{enumerate}
\end{lettered}

\begin{proof}[Proof of Theorem~\ref{thmE}]
Part~(1) is Theorem~\ref{thm:odd-gap} together with the sharpness remark following it, and
part~(2) is Proposition~\ref{prop:even-rank}.
\end{proof}

\hyphenation{in-ter-me-diate}

In even dimension the situation is already largely determined. By Theorem~\ref{thmA}, the bound
$1+\tau$ is attained only by weighted homogeneous germs, so every other germ satisfies
$K\ge2+\tau$, that is, $e(V)\ge1$. On the other hand, there are germs with $K=3+\tau$, that is,
with $e(V)=2$, in every even dimension (Corollary~\ref{cor:unbounded}). The only open question is
whether the intermediate value $K=2+\tau$ can occur. Part~(2) shows that a field realizing it
would have a linear part of rank at least $n-1$, nilpotent with a Jordan block of size at least
$n$ when $\operatorname{ord}f\ge3$.

\begin{conj}[Even-dimensional gap]\label{conjE}
If $n$ is even and $(V,\underline{0})$ is not weighted homogeneous, then $K(V)\ge3+\tau(V)$, that is,
$e(V)\ge2$.
\end{conj}

\section{Curves and their normalization}\label{sec:optimal-curves}

The genericity statement is intrinsic to reduced curves. The radial index
is the appropriate invariant in this generality; for complete-intersection
curves it translates into the GSV index. We subsequently specialize to
plane curves, where K\"ahler differentials provide further formulas and
explicit computations. By Theorem~\ref{thmD}, the plane-curve results
also apply to their double suspensions in odd dimensions.

\subsection{Generic fields and the index on the normalization}
Let $(C,\underline{0})$ be a reduced complex analytic curve germ, with an isolated
singular point, and write
\[
 R=\mathcal O_{C,\underline{0}},\qquad
 \widetilde R=\bigoplus_{i=1}^r\mathbb C\{t_i\},\qquad
 \delta(C)=\dim_{\mathbb C}(\widetilde R/R).
\]
Here $r$ is the number of branches. Every derivation $\xi$ of $R$
lifts uniquely to the normalization, where
\[
 \widetilde\xi=
 \left(a_i(t_i)\frac{\partial}{\partial t_i}\right)_{i=1}^r.
\]
Indeed, a derivation is induced by an ambient holomorphic vector field
tangent to $C$. Its local flow preserves $C$ and lifts to the
normalization. Since the flow fixes the singular point, it fixes each
point above it; in particular $a_i(\underline{0})=0$. A field has an isolated zero
on $C$ precisely when none of its $a_i$ is identically zero.

Define
\begin{equation}\label{eq:branchwise-alpha}
 \alpha_i(C)=\min\left\{\operatorname{ord}_{t_i}a_i:
 \xi\in\operatorname{Der}_{\mathbb C}(R),\ a_i\not\equiv0\right\}.
\end{equation}
The minima are finite: for example, multiplying a nonzero derivation
of the function field of the $i$th branch by a sufficiently high power
of the conductor yields a derivation of $R$ supported on that branch.
These are invariants of the entire curve, not of its branches taken
separately. They do not depend on the normalization parameters.

\begin{lettered}{B}\label{thm:optimal-reduced-curves}\label{thmB}
Let $(C,\underline{0})$ be a reduced complex analytic curve germ with an isolated
singularity and $r$ branches. For any holomorphic vector field $\xi$
with an isolated zero at the origin, whose lift has coefficients $a_i$,
\begin{equation}\label{eq:normalization-radial}
 \operatorname{Ind}_{\rm rad}(\xi;C,\underline{0})
 =\sum_{i=1}^r\operatorname{ord}_{t_i}a_i-r+1.
\end{equation}
Consequently its minimum radial index is
\begin{equation}\label{eq:optimal-radial-curves}
 K_{\rm rad}(C)=\sum_{i=1}^r\alpha_i(C)-r+1.
\end{equation}
It is attained by a generic constant linear combination of any finite
set of generators of $\operatorname{Der}_{\mathbb C}(R)$.

If, in addition, $C$ is an isolated complete-intersection curve
singularity, then
\begin{equation}\label{eq:normalization-gsv}
 \operatorname{Ind}_{\rm GSV}(\xi;C,\underline{0})
 =\sum_{i=1}^r\operatorname{ord}_{t_i}a_i-2\delta(C),
\end{equation}
and
\begin{equation}\label{eq:optimal-reduced-curves}
 K_{\rm GSV}(C)=\sum_{i=1}^r\alpha_i(C)-2\delta(C)
 =\sum_{i=1}^r\alpha_i(C)-\mu(C)-r+1.
\end{equation}
The same generic fields realize both minima.
\end{lettered}

\begin{proof}
The radial-index normalization formula
\eqref{eq:normalization-radial} follows by excising a small
neighborhood of the singular point and comparing the radial index
with the ordinary indices on the normalization. The punctured curve
is the disjoint union of $r$ punctured discs. Each lifted field
extends over its missing point with ordinary complex index
$\operatorname{ord}_{t_i}a_i$. The radial reference field has ordinary
index one on each disc, whereas its radial index at the singular
point is one. Thus the correction is $1-r$, giving
\eqref{eq:normalization-radial}.

Choose generators $\xi^{(1)},\ldots,\xi^{(s)}$ of
$\operatorname{Der}_{\mathbb C}(R)$ and denote their lifted coefficients
by $a_i^{(j)}$. On each branch the smallest nonzero order among
$a_i^{(j)}$ equals $\alpha_i(C)$: multiplication by a function of
$R$ cannot decrease an order, and sums cannot decrease the minimum.
The coefficient of $t_i^{\alpha_i(C)}$ in
$\sum_j\lambda_j a_i^{(j)}$ is a nonzero linear form in the constants
$\lambda_j$. Outside the union of these finitely many proper
hyperplanes, all orders equal $\alpha_i(C)$ simultaneously.
This proves \eqref{eq:optimal-radial-curves} and the genericity claim.

For an ICIS curve, the standard radial--GSV comparison gives
\[
 \operatorname{Ind}_{\rm GSV}(\xi;C,\underline{0})
 =\operatorname{Ind}_{\rm rad}(\xi;C,\underline{0})-\mu(C).
\]
The ICIS curve identity $\mu(C)=2\delta(C)-r+1$ now gives
\eqref{eq:normalization-gsv} and \eqref{eq:optimal-reduced-curves}.
\end{proof}

The theorem gives a coordinate-free description of the minimizing stratum. Let
$p_1,\ldots,p_r$ be the points of the normalization over $\underline{0}$, and set
\[
 L_i=\left(\mathfrak m_{p_i}^{\alpha_i(C)}/\mathfrak m_{p_i}^{\alpha_i(C)+1}\right)
       \otimes_{\mathbb C}T_{p_i}\widetilde C.
\]
Taking the leading class of the lifted derivation defines a linear map
\[
 \ell_C:\operatorname{Der}_{\mathbb C}(R)\longrightarrow\bigoplus_{i=1}^rL_i,
 \qquad \xi\longmapsto(\ell_i(\xi))_{i=1}^r.
\]
Write $L_C=\operatorname{im}(\ell_C)$ and
\[
 \Omega_C=L_C\setminus\bigcup_{i=1}^r\ker(\operatorname{pr}_i|_{L_C}).
\]
Each projection is nonzero by the definition of $\alpha_i(C)$, so $\Omega_C$ is a nonempty
Zariski-open subset of the finite-dimensional space $L_C$. The minimizing fields are exactly
$\ell_C^{-1}(\Omega_C)$: their lifts have order $\alpha_i(C)$ on every branch. For an irreducible
curve, $L_C$ is a line, so the projective leading class of a minimizing field is unique. This
canonical stratum, rather than a selected field, is the curve model for genericity. For an ICIS,
the same stratum minimizes the GSV index.

\begin{remark}
The first part of the theorem does not require an embedding as a
complete intersection. Its essential data are the normalization and
the module of derivations. The complete-intersection assumption is
used only to express the resulting minimum in terms of the GSV index
and the Milnor number.
\end{remark}

\begin{remark}
For plane curves, formula \eqref{eq:normalization-gsv} can also be
proved directly by a differential-length calculation. With
$\omega_C$ the dualizing module and $\overline\Omega_C^1$ the image
of K\"ahler differentials in meromorphic differentials, finite duality
and the residue presentation give
\[
 \ell(\omega_C/\nu_*\Omega_{\widetilde C}^1)=\delta(C),
 \qquad
 \ell(\omega_C/\overline\Omega_C^1)=\tau(C).
\]
Hence
\begin{equation}\label{eq:differential-colength}
\ell(\nu_*\Omega_{\widetilde C}^1/\overline\Omega_C^1)=\tau-\delta.
\end{equation} If $(x_i(t_i),y_i(t_i))$
parametrize the branches and
$W=R(x_i')_i+R(y_i')_i\subset\widetilde R$, then
$\ell(\widetilde R/W)=\tau-\delta$. Contraction with the lifted
coefficient tuple $a=(a_i)$ identifies the coefficient ideal
$I_\xi=(\xi(x),\xi(y))$ with $aW$, whence
\[
 \ell(R/I_\xi)=\sum_i\operatorname{ord}_{t_i}a_i+
 \tau-2\delta.
\]
The plane-curve index formula \eqref{eq88} recovers
\eqref{eq:normalization-gsv}. This argument also explains the
appearance of differential values in the following subsections.
\end{remark}

\medskip
Theorem~\ref{thm:optimal-reduced-curves} also characterizes weighted homogeneity for curves of
arbitrary embedding dimension. The input is the following spectral criterion \cite[Lemma~3]{MS2},
valid for every complex analytic germ: if a holomorphic vector field vanishing at the origin has
linear spectrum in an open half-plane bounded by a line through $0$, the germ is weighted
homogeneous.

\begin{cor}\label{cor:curve-wh}
Let $(C,\underline{0})$ be a singular reduced curve germ with $r$ branches. Then $K_{\rm rad}(C)\ge1$, and if
$C$ is an ICIS, $K_{\rm GSV}(C)\ge1-\mu(C)$; in both cases equality holds if and only if
$\alpha_i(C)=1$ for every $i$. Moreover:
\begin{enumerate}
\item $(C,\underline{0})$ is weighted homogeneous if and only if some derivation of $\OO_{C,\underline{0}}$ has lifted
coefficients $a_i=c_it_i+O(t_i^2)$, where $c_1,\ldots,c_r$ are nonzero and lie in an open
half-plane bounded by a line through $0$. In particular, a weighted homogeneous curve satisfies
$K_{\rm rad}(C)=1$, and $K_{\rm GSV}(C)=1-\mu(C)$ if it is an ICIS.
\item If $C$ is irreducible, then $(C,\underline{0})$ is weighted homogeneous if and only if $\alpha_1(C)=1$,
that is, $K_{\rm rad}(C)=1$. For an irreducible ICIS curve this is equivalent to
$K_{\rm GSV}(C)=1-\mu(C)$.
\end{enumerate}
\end{cor}

\begin{proof}
Since $\alpha_i(C)\ge1$, the inequalities and the characterization of equality follow from
\eqref{eq:optimal-radial-curves} and \eqref{eq:optimal-reduced-curves}.

Suppose that $(C,\underline{0})$ is weighted homogeneous, with positive weights $w_j$ in coordinates
$z_1,\ldots,z_N$. The Euler field $\sum_jw_jz_j\partial_{z_j}$ is tangent to $C$. On the $i$th
branch choose a coordinate $z_j$ whose restriction $h=bt_i^k+\cdots$, with $b\ne0$ and $k\ge1$,
is not identically zero. The lifted coefficient satisfies $a_ih'=w_jh$, hence
$a_i=(w_j/k)t_i+O(t_i^2)$. All the $c_i$ are positive, and all the $\alpha_i$ equal one.

Conversely, let $\xi$ be a derivation as in (1), and let $R=\OO_{C,\underline{0}}$ with maximal ideal
$\mathfrak m$. For $k\ge1$ let $F_k\subset\mathfrak m$ consist of the functions whose restrictions
have order at least $k$ on every branch. Since the lift of $\xi(h)$ is $(a_ih_i')_i$, the derivation
preserves each $F_k$. Sending $h\in F_k$ to the vector of its coefficients of $t_i^k$ embeds
$F_k/F_{k+1}$ in $\CC^r$, and $\xi$ acts there as $\operatorname{diag}(kc_1,\ldots,kc_r)$. Let
$\mathfrak c$ be the conductor, an ideal of $\widetilde R$ contained in $\mathfrak m$. Then
$\mathfrak m\mathfrak c=(\mathfrak m\widetilde R)\mathfrak c\subset\mathfrak m^2$, and since
$\widetilde R/\mathfrak m\widetilde R$ has finite length, $\mathfrak m\mathfrak c$ contains $F_k$
for $k$ large. Hence the images of the $F_k$ form a finite $\xi$-invariant filtration of
$\mathfrak m/\mathfrak m^2$ whose graded pieces are quotients of the $F_k/F_{k+1}$. Every eigenvalue
of the linear part of $\xi$ is therefore of the form $kc_i$ with $k\ge1$, and lies in the given
half-plane. The spectral criterion shows that $(C,\underline{0})$ is weighted homogeneous.

If $C$ is irreducible, the condition in (1) only asks that $c_1\ne0$, that is, that some derivation
lift with order one; this means $\alpha_1(C)=1$.
\end{proof}

For several branches the half-plane condition is what the argument uses; we do not know whether
$\alpha_i(C)=1$ for all $i$ suffices. For plane curves, Corollary~\ref{cor:curve-wh} is consistent
with Theorem~\ref{thmA}: there $1-\tau\ge1-\mu$, and $K=1-\mu$ forces $\mu=\tau$.

\subsection{The minimum on a plane branch}

Suppose now that $C$ is irreducible. Let $\Gamma$ be its value
semigroup, with conductor $c=\mu(C)=2\delta(C)$, and put
$$
 \Lambda=\{\operatorname{ord}_t w+1:0\ne w\in W\},
 \qquad W=Rx'(t)+Ry'(t).
$$
Thus $\Lambda$ is the set of values of the nonzero differentials on
$C$. It contains $\Gamma\setminus\{0\}$ and is cofinite in the
positive integers. Since the branch is singular, $1\notin\Lambda$.

\begin{cor}\label{cor:largest-differential-gap}
For a singular plane branch,
$$
 K(C)=-\max\bigl(\mathbb Z_{>0}\setminus\Lambda\bigr).
$$
Equivalently, if $g_\Lambda$ denotes this largest gap, then
$\alpha_1(C)=\mu(C)-g_\Lambda$. In particular, for a weighted
homogeneous branch, $K(C)=1-\mu(C)=1-\tau(C)$.
\end{cor}

\begin{proof}
In the fraction field of $R$, the coefficients of derivations form
exactly the fractional ideal
$$
 D=(R:W)=\{a:aW\subset R\}.
$$
Indeed, the condition states that $a\partial_t$ sends the coordinate
functions, and hence every convergent function on $C$, into $R$.
Conversely, every derivation has this form. Its lift is holomorphic
by the normalization argument above.

The symmetry theorem for fractional ideals of a Gorenstein branch
\cite[Proposition~2.20]{Pol} gives
$$
 z\in\operatorname{val}(D)
 \quad\Longleftrightarrow\quad
 c-1-z\notin\operatorname{val}(W)
 \quad\Longleftrightarrow\quad
 c-z\notin\Lambda.
$$
Since all integers at least $c$ belong to $\Lambda$, and the largest
integer outside $\Lambda$ is the positive integer $g_\Lambda$, the
least value of $D$ is $c-g_\Lambda$. Apply
Theorem~\ref{thm:optimal-reduced-curves} with $r=1$.
For a weighted homogeneous branch, after an analytic coordinate
change one has a monomial parametrization. Then
$\Lambda=\Gamma\setminus\{0\}$, whose largest positive gap is
$c-1$, and the last assertion follows.
\end{proof}

\subsection{Exact formulas for plane branches}

Put $d=\mu-\tau$. By \eqref{eq:differential-colength},
\begin{equation}\label{eq:extra-differential-count}
 \#(\Lambda\setminus\Gamma)=d.
\end{equation}
Indeed, the positive integers missing from $\Lambda$ are counted by
$\ell(\mathbb C\{t\}/W)=\tau-\delta$, and there are $\delta$
positive integers missing from $\Gamma$.

List the elements of $\Gamma$ in increasing order as
$0=\gamma_0<\gamma_1<\gamma_2<\cdots$, with $\gamma_1=m$ the
multiplicity. Let $c=\mu$ be the conductor.

\begin{cor}\label{thm:curve-exact-minima}
Let $C$ be a singular plane branch of multiplicity $m$.
\begin{enumerate}
\item If $m=3$, then $K(C)=1-\mu+3(\mu-\tau)=1+2\mu-3\tau$.
\item If $\mu-\tau=1$, then $K(C)=1-\mu+m$.
\end{enumerate}
\end{cor}

\begin{proof}
For $m=3$ there is a single Puiseux pair $(3,q)$ with $q>3$
and $3\nmid q$. Choose coordinates with $x=t^3$ and
$\operatorname{ord}_t y=q$, so $\Gamma=\langle3,q\rangle$ and
$c=2q-2$. Every differential value below $q$ comes from the
leading term of $a\,dx$ for some $a\in R$, since a nonzero
$b\,dy$ with $b\in R$ has value at least $q$. Such a value lies
in $3+\Gamma\subset\Gamma$. Also $q\in\Gamma$. Thus every element
of $\Lambda\setminus\Gamma$ is a gap of $\Gamma$ above $q$.

The gaps above $q$ all lie in the residue class of $2q$ modulo
$3$, since the least semigroup elements in the three classes are
$0,q,2q$. These gaps end at $2q-3=c-1$. Multiplication of a
differential by $x$ increases its value by $3$. Consequently
$\Lambda\setminus\Gamma$ consists of the last $d$ gaps in that
arithmetic progression:
\[
 \Lambda\setminus\Gamma
 =\{c-1-3j:0\le j<d\}.
\]
If $d>0$, its smallest element is $2q-3d>q$, so $3d<q$.
The inequality also holds for $d=0$.

By symmetry of the semigroup, reflection $g\mapsto c-1-g$
sends the gaps of $\Gamma$ to semigroup elements below $c$.
Under this reflection the filled gaps correspond to
$0,3,\ldots,3(d-1)$. Because $3d<q$, the smallest semigroup
element outside this list is $3d$. Thus the largest gap still
missing from $\Lambda$ is $c-1-3d$, which gives the first
formula by Corollary~\ref{cor:largest-differential-gap}.

For the second formula let $g$ be the unique element of
$\Lambda\setminus\Gamma$. By symmetry $c-1-g\in\Gamma$, and
$\Lambda+\Gamma\subset\Lambda$ gives $c-1\in\Lambda$.
Hence $g=c-1$. The next largest gap of $\Gamma$ is $c-1-m$,
the reflection of the smallest positive semigroup element. A plane branch of multiplicity two is
weighted homogeneous, so $d=1$ forces $m\ge3$; hence $c-1-m>0$.
The same corollary gives $K(C)=1-c+m$.
\end{proof}

\medskip
\noindent The following finite procedure makes the differential-value
criterion effective for arbitrary plane branches.

\begin{teo}\label{thm:branch-finite-minimum}
For every singular plane branch, let $j$ be the first index
in $0\le j\le d$ for which $c-1-\gamma_j\notin\Lambda$. Then
\[
 K(C)=1-\mu+\gamma_j.
\]
Membership for these at most $d+1$ candidates is computed from the
parametrization modulo $t^c$. In coordinates $x=t^m$, form the
coefficient matrix of
\[
 t\,x^i y^k z'(t)\pmod{t^c},
 \qquad z\in\{x,y\},\quad i\ge0,\quad 0\le k<m,
\]
keeping the rows whose order is less than $c$. Its pivot exponents,
with columns in increasing order, are the elements of $\Lambda$
below $c$.

If $d>0$, one has $1-\mu+m\le K\le 1-\mu+\gamma_d$.
For $d=2$ the only possibilities are $1-\mu+m$ and
$1-\mu+\gamma_2$.
\end{teo}

\begin{proof}
By symmetry of $\Gamma$, the map $u\mapsto c-1-u$ sends its
gaps to its elements below $c$. Consider the reflected filled gaps
\[
 S=\{c-1-\lambda:\lambda\in\Lambda\setminus\Gamma\}.
\]
The set $S$ has $d$ elements and is closed under subtraction of
semigroup elements whenever the difference is in $\Gamma$.
Indeed, if $s\in S$ and $s-\gamma\in\Gamma$ with $\gamma\in\Gamma$,
then $c-1-(s-\gamma)\in\Lambda+\Gamma\subset\Lambda$.
By symmetry this number is a gap, so $s-\gamma\in S$.
In particular $0\in S$ if $d>0$.

Because $1$ is still missing from $\Lambda$, there are at least
$d+1$ semigroup elements below $c$. The least element of
$\Gamma\setminus S$ is among the first $d+1$ and reflects to the
largest gap of $\Lambda$. Corollary~\ref{cor:largest-differential-gap}
gives the formula and the bounds. If $d=2$, then
$S=\{0,s\}$ for a minimal generator $s$ of $\Gamma$.
The least element outside $S$ is $m$ if $s\ne m$ and
$\gamma_2$ if $s=m$.

For the finite calculation, the Weierstrass equation in $y$ has
degree $m$, so $1,y,\ldots,y^{m-1}$ generate $R$ over
$\mathbb C\{x\}$. The displayed rows therefore span the
image of $tW$ modulo $t^c$. Only finitely many survive the truncation
because $x=t^m$. The order of a nonzero vector in this image is
its first nonzero column exponent. The set of such orders is
exactly the set of leading exponents in an echelon basis, giving
the pivot criterion. Truncating the parametrization modulo $t^c$
suffices since $t\,d/dt$ preserves $(t^c)$.
\end{proof}

In terms of the excess, $e=2(\mu-\tau)$ in multiplicity three, and $e=m-1\ge2$ when
$\mu-\tau=1$, since branches of multiplicity two are weighted homogeneous. In both cases
$e\ge2$ unless the branch is weighted homogeneous, as Theorem~\ref{thm:odd-gap} guarantees for
every reduced plane curve.

\begin{exe}
For the node $xy=0$, derivations on the normalization have the form
$(t_1b_1(t_1)\partial_{t_1},t_2b_2(t_2)\partial_{t_2})$.
Hence $\alpha_1=\alpha_2=1$, $\delta=1$, and $K=0=1-\tau$.
For the monomial branch $(t^p,t^q)$ with coprime $p,q\ge2$, the Euler
field gives $\alpha_1=1$, and
$K=1-(p-1)(q-1)$.

Plane branches with Milnor number at most $11$ are weighted homogeneous: their semigroups are
$\langle2,q\rangle$, $\langle3,4\rangle$ or $\langle3,5\rangle$, and each of these carries a single
analytic type. The first branches that are not weighted homogeneous have $\mu=12$. The simplest is
the branch of type $E_{12}$,
$$
 g=x^3+y^7+xy^5,\qquad \Gamma=\langle3,7\rangle,\quad \mu=12,\quad \tau=11 .
$$
The field
$$
 \delta_2=(147xy+25xy^2+5y^4)\,\partial_x+(-3x+63y^2+10y^3)\,\partial_y,
$$
which satisfies $\delta_2(g)=(441y+75y^2)\,g$, is generic. Its lift to the normalization $y=t^3$, $x=-t^7+\cdots$ has order $4=\alpha_1(C)$,
and by \eqref{eq88} its index is $\ell\bigl(R/(g,\delta_2(x),\delta_2(y))\bigr)-\tau=3-11=-8=K$.
Thus $K=3-\tau$, which is the value allowed by Theorem~\ref{thmE}(1) for a germ that is not
weighted homogeneous: the excess is exactly two. As for every field on a germ that is not weighted
homogeneous (Lemma~\ref{lem:nondeg}), its linear part is degenerate: here it is the nilpotent field
$-3x\,\partial_y$. The details, including a basis of the
fields on $C$, are in Section~\ref{ss:E12curve}.
\end{exe}

\subsection{Examples}\label{ss:E12curve}


\begin{cor}\label{cor:cubic-branch-family}
Let $q,a$ be positive integers with $3\nmid q$ and
$2q/3<a\le q-2$. The germ $C_{q,a}$ defined by
\[
 f=x^3+y^q+xy^a
\]
is a plane branch with
\[
 \mu=2(q-1),\qquad \mu-\tau=q-a-1,\qquad K(C_{q,a})=q-3a.
\]
In particular, all these branches are non-weighted homogeneous.
\end{cor}

\begin{proof}
With $y=t^3$ and $x=t^q u$, the equation becomes
$u^3+1+t^{3a-2q}u=0$. Its three solutions near the roots of
$u^3=-1$ are permuted transitively by the changes of parameter
$t\mapsto\zeta t$ with $\zeta^3=1$, because $3\nmid q$.
Thus the germ is a branch with semigroup $\langle3,q\rangle$
and Milnor number $2(q-1)$.

Work in the Milnor algebra $A_f=\OO_{\CC^2,0}/(f_x,f_y)$.
Set $b=q/a$. In $A_f$ the Jacobian relations read
\[
 x^2=-\tfrac13 y^a,\qquad xy^{a-1}=-b y^{q-1}.
\]
Multiplying the second relation by $x$ and substituting gives
$-y^{2a-1}/3=b^2 y^{2q-a-1}$. Since $3a>2q$, this forces
$y^{2q-a-1}=0$. The weighted Euler identity with weights $1/3$
and $1/q$ gives
$[f]=-b(2/3-a/q)[y^q]$, with $2/3-a/q\ne0$. The ideal $([y^q])$ is spanned by the classes
$[y^q],\ldots,[y^{2q-a-2}]$, since $xy^q=-b y^{2q-a}=0$.

These classes are independent. By the second Jacobian relation,
they are nonzero scalar multiples of $[xy^a],\ldots,[xy^{q-2}]$.
For weights $(q,3)$ the initial forms of $f_x,f_y$ are $3x^2$
and $q y^{q-1}$. They are a regular sequence, so the associated
graded Milnor algebra is $\mathbb C[x,y]/(x^2,y^{q-1})$.
In particular the indicated monomials are linearly independent in
$A_f$. It follows that $\mu-\tau=\dim_{\CC}([f])=q-a-1$.
Corollary~\ref{thm:curve-exact-minima} now gives $K=q-3a$.
\end{proof}

For example, $C_{10,7}$ and $C_{10,8}$ both have semigroup
$\langle3,10\rangle$. Their minima are $-11$ and $-14$,
respectively. The member $C_{7,5}$, with $e=2$, is the $E_{12}$ branch
below, where we also exhibit a field attaining the minimum.

\subsubsection*{A branch of type \texorpdfstring{$E_{12}$}{E12}}
Let $C=\{g=0\}$ with
$$
g=x^3+y^7+xy^5 .
$$
This branch is not weighted homogeneous: its semigroup is $\Gamma=\langle3,7\rangle$, $\mu=12$ and $\tau=11$. Consider
$$
\delta_1=(5x^2+7xy^2)\,\partial_x+(2xy+3y^3)\,\partial_y,
$$
$$
\delta_2=(147xy+25xy^2+5y^4)\,\partial_x+(-3x+63y^2+10y^3)\,\partial_y .
$$
A direct computation gives
$$
\delta_1(g)=(15x+21y^2)\,g,\qquad \delta_2(g)=(441y+75y^2)\,g,\qquad \det(\delta_1,\delta_2)=-15\,g ,
$$
so $\delta_1,\delta_2$ is a basis of $\mathrm{Der}(-\log C)$ by Saito's criterion. With a parametrization $y=t^3$, $x=-t^7+\cdots$, the lifts of $\delta_1$ and $\delta_2$ to the normalization have orders $7$ and $4$. Hence $\alpha_1(C)=4$ and, by Theorem~\ref{thm:optimal-reduced-curves},
$$
K=4-12=-8,\qquad\text{while}\qquad 1-\tau=-10,\quad 1-\mu=-11 .
$$
For the numerical invariants, the weighted initial form $x^3+y^7$
has Milnor number $12$. In the Milnor algebra of $g$ the relations are
$$x^2=-\tfrac13y^5,\qquad xy^4=-\tfrac75y^6,\qquad [g]=\tfrac1{15}[y^7].$$
They imply $y^8=0$ and $xy^7=0$, whereas $y^7\ne0$; for example,
the standard monomials are $1,y,\ldots,y^7,x,xy,xy^2,xy^3$.
Thus the ideal generated by $[g]$ has dimension one and $\tau=11$.

The generic vector fields are $a\delta_1+b\delta_2$ with $b(0)\neq0$; for instance $\delta_2$ itself, for which $\langle g,\delta_2(x),\delta_2(y)\rangle=\langle x-21y^2-\tfrac{10}{3}y^3,\,y^3\rangle$ has colength $3$ and formula \eqref{eq88} gives $3-11=-8$. Indeed, after eliminating $x=21y^2+\frac{10}{3}y^3$, the first
component of $\delta_2$ becomes
$$\frac{y^3}{3}(250y^2+3060y+9261),$$
which is $y^3$ times a unit; the equation $g$ belongs to the resulting ideal.
The field $\delta_1$ has index $-5$ and the Hamiltonian field has index $0$. By Corollary~\ref{cor:largest-differential-gap}, the largest gap in the differential values is $8$; indeed $\Lambda=(\Gamma\setminus\{0\})\cup\{11\}$.

\subsubsection*{Branches with the same Milnor and Tjurina numbers}
Consider the branches
$$
C_1:\ (t^4,\ t^9+t^{10}),\qquad C_2:\ (t^4,\ t^9+t^{11}).
$$
They have the same semigroup $\langle4,9\rangle$, hence conductor
and Milnor number $24$. The differential $9y\,dx-4x\,dy$ has value
$14$ on $C_1$ and $15$ on $C_2$. For completeness, the full calculation
below the conductor is the following finite coefficient computation.
Reduce modulo $t^{23}$ the polynomials
$$x^ay^b x'(t)\quad(4a+9b\leq19),\qquad
x^ay^b y'(t)\quad(4a+9b\leq14),\qquad a,b\geq0.$$
These span the image of $W$ in $\CC\{t\}/(t^{23})$: omitted monomials
have order at least $23$ and cannot affect any lower coefficient.
Row reduction, with columns in increasing powers of $t$, gives the
following differential values (pivot exponents plus one):
$$
\begin{array}{c|l}
C_1&4,8,9,12,13,14,16,17,18,19,20,21,22,23\\
C_2&4,8,9,12,13,15,16,17,18,19,20,21,22,23.
\end{array}
$$
All values at least $24$ already belong to $\Gamma\setminus\{0\}$.
Thus $\Lambda\setminus\Gamma$ is respectively $\{14,19,23\}$ and
$\{15,19,23\}$, and $\mu-\tau=3$ in both cases, so $\tau=21$.
The gaps of $\Gamma$ are $1,2,3,5,6,7,10,11,14,15,19,23$, so the largest gap of $\Lambda$ is $15$ for $C_1$ and $14$ for $C_2$. By Corollary~\ref{cor:largest-differential-gap},
$$
K(C_1)=-15,\qquad K(C_2)=-14 .
$$

The two possibilities for $d=2$ already occur with a common semigroup.
Take $(x,y)=(t^5,t^6+t^j)$ for $j=8,9$. Both branches have
$\Gamma=\langle5,6\rangle$ and $\mu=20$. The matrix of
Theorem~\ref{thm:branch-finite-minimum} has pivot exponents below
$20$ given by
\[
\begin{array}{c|l|c}
j&\Lambda\cap[1,19]&K\\\hline
8&5,6,10,11,12,13,15,16,17,18,19&-14\\
9&5,6,10,11,12,14,15,16,17,18,19&-13 .
\end{array}
\]
In both rows, the monomials multiplying $t x'$ or $t y'$ are
$1,x,y,x^2,xy,y^2$. Thus $d=2$ and $\tau=18$ for both branches.

\subsection{Petals and tangencies with the link}\label{ss:petals}
The minimizing property of generic fields on curves can be seen in the real phase portrait of the
field. Let $d(v)=\Ind_{\rm rad}(v;C,\0)-1$; it is the degree of $v$ on the link relative to the
radial field. The link is a union of $r$ circles, one for each branch, and the boundary field has a
relative winding number $d_i(v)$ on each of them. In a parameter
$t_i$ of the $i$th branch the lift is $a_i(t_i)\partial_{t_i}$, and if $\operatorname{ord}a_i=k$,
its winding relative to the radial field $t_i\partial_{t_i}$ is $k-1$. Thus
\[
 d(v)=\sum_{i=1}^rd_i(v),\qquad d_i(v)=\operatorname{ord}_{t_i}a_i-1\ge\alpha_i(C)-1 .
\]
The real trajectories of $t^k\partial_t$ form the Leau--Fatou flower: $2(k-1)$ elliptic sectors,
or petals, and no hyperbolic sector. This gives an index formula in the spirit of Poincar\'e and
Bendixson, who express the index of a planar singular point as $1+(e-h)/2$, with $e$ elliptic and
$h$ hyperbolic sectors.

\begin{prop}[Bendixson formula for curves]\label{prop:bendixson}
Let $(C,\0)$ be a reduced curve germ and $v$ a holomorphic field on $C$ with an isolated zero. The
real trajectories of $v$ on $C\setminus\{\0\}$ near $\0$ form $e(v)=2\sum_i(\operatorname{ord}a_i-1)$
elliptic sectors and no hyperbolic sector, and
\[
 \Ind_{\rm rad}(v;C,\0)=1+\frac{e(v)}2 .
\]
If $C$ is an ICIS, then $\Ind_{\rm GSV}(v;C,\0)=1-\mu+e(v)/2$. For $\varepsilon$ small, the
petals are also visible on the link: $v$ is tangent to $C\cap\mathbb S_\varepsilon$ at exactly
$e(v)$ points, and at each of them the trajectory stays inside the ball.
\end{prop}

\begin{proof}
The normalization maps a disjoint union of discs homeomorphically onto $C$ away from $\0$, so the
phase portrait of $v$ near $\0$ is the union of those of its lifts. A germ $a(t)\partial_t$ with
$\operatorname{ord}a=k\ge1$ is analytically conjugate to $t^k(1+ct^{k-1})^{-1}\partial_t$, and its
real phase portrait is that of $t^k\partial_t$: $2(k-1)$ elliptic sectors and no hyperbolic one.
Summing over the branches, $e(v)=2\sum_id_i(v)=2d(v)$, and $\Ind_{\rm rad}=1+d(v)$. For an ICIS,
$\Ind_{\rm GSV}=\Ind_{\rm rad}+\chi(F)-1=\Ind_{\rm rad}-\mu$.

For the link, write $|z|^2=|t|^{2m}u(t)$ on a branch of multiplicity $m$, with $u(\0)>0$, and put
$s=tu^{1/2m}$. Then $C\cap\mathbb S_\varepsilon$ is the circle $|s|=\varepsilon^{1/m}$, and in the
coordinate $s$ the field is $b(s)\partial_s$ with $\operatorname{ord}b=k$. It is tangent to the
circle where $\operatorname{Re}\bigl(\bar s\,b(s)\bigr)=0$. For $|s|$ small these are $2(k-1)$
points near the directions where $s^{k-1}$ is purely imaginary, and they are the points where the
petals touch the circle from inside.
\end{proof}

For a planar field transverse to a small circle except at finitely many tangencies, Poincar\'e's
formula, generalized by Pugh \cite{Pugh}, reads $\Ind=1+(i-x)/2$, with $i$ internal and $x$
external tangencies. Here $x=0$ and $i=e(v)$: the petals account for the whole index. A weighted homogeneous germ is the case without petals: its Euler field has a star node on
every branch and is transversal to the link. By Theorem~\ref{thm:optimal-reduced-curves} a generic
field minimizes every winding $d_i$ simultaneously, so it minimizes the number of petals on every
branch, not only their total; this is the geometric content of Theorem~\ref{thmB}.

\begin{cor}\label{cor:petals}
Let $(C,\0)$ be a reduced curve germ and $v$ a holomorphic field on $C$ with an isolated zero. For
$\varepsilon$ small, $v$ has no external tangencies with the link $C\cap\mathbb S_\varepsilon$,
and the number $i(v)$ of internal tangencies satisfies $i(v)\ge2\sum_i(\alpha_i(C)-1)$, with
equality if and only if $v$ is generic. In particular:
\begin{enumerate}
\item if $C$ is weighted homogeneous, then $v$ is generic if and only if it is transversal to the
link;
\item if $C$ is a plane curve, then $v$ is generic if and only if $i(v)=2(K+\mu-1)$;
\item if $C$ is a branch, a generic field has $2(\alpha_1(C)-1)$ tangencies, which is zero if and
only if $C$ is weighted homogeneous, and at least four otherwise. For a plane branch this number is
$2(\mu-g_\Lambda-1)$, with $g_\Lambda$ the largest gap of
Corollary~\ref{cor:largest-differential-gap}.
\end{enumerate}
\end{cor}

\begin{proof}
By Proposition~\ref{prop:bendixson}, all tangencies are internal and $i(v)=e(v)=2d(v)$. By
Theorem~\ref{thm:optimal-reduced-curves}, $d(v)\ge\sum_i(\alpha_i-1)$, with equality exactly for
generic fields. For a weighted homogeneous curve $\alpha_i=1$ for every $i$
(Corollary~\ref{cor:curve-wh}), which gives (1). For a plane curve
$K_{\rm rad}=K+\mu$, which gives (2). In (3), $\alpha_1=1$ exactly in the weighted homogeneous case
by Corollary~\ref{cor:curve-wh}(2). If $\alpha_1=2$, Lemma~\ref{order2} shows that $C$ is weighted
homogeneous, so $\alpha_1=1$, a contradiction; hence $\alpha_1\ge3$ otherwise. The last statement is
$\alpha_1=\mu-g_\Lambda$.
\end{proof}

For the $E_{12}$ branch $x^3+y^7+xy^5$ of Section~\ref{ss:E12curve}, $\alpha_1=4$, so a generic
field has six petals, all of them touching the link, while on the weighted homogeneous branch $x^3+y^7$ the Euler field has none.

\section{Surfaces and the logarithmic index}\label{ss:surfaces}

\hyphenation{ca-no-ni-cal}

In this section $(V,\0)$ is a \emph{normal surface singularity}, not necessarily a hypersurface or
a complete intersection, and $\kappa_X=c_1(\omega_X)$ denotes the canonical class of a resolution.
The natural index is then the radial index, and we write $K_{\rm rad}(V)$ for its minimum. When
$(V,\0)$ is an isolated hypersurface singularity, $\Ind_{\rm GSV}=\Ind_{\rm rad}+\mu$, so all
results translate to the GSV index, with $K=K_{\rm GSV}=K_{\rm rad}+\mu$ and $D=K_{\rm rad}-1$;
we state these translations explicitly. Only the algebraic formula for the residual length
(Section~\ref{ss:residual}), the comparison with the Tjurina number, and the explicit computations
of Sections~\ref{ss:E12surface} and~\ref{ss:excess} use the hypersurface structure.

\subsection{The index on a resolution}

The proof of Theorem~\ref{thm:optimal-reduced-curves} has two steps: the index is read on the
normalization, and the minimum is then attained by a generic field. In dimension two the
normalization is replaced by a resolution to which vector fields lift, and the orders $\alpha_i$ by
the vanishing orders of the lifted fields along the exceptional curves.

Let $\pi:X\to V$ be the minimal resolution, with $E=\pi^{-1}(\0)$. By \cite{BW} every holomorphic
vector field $v$ on $V$ lifts to a holomorphic vector field $\widetilde v$ on $X$. The same holds for
the minimal good resolution, which is obtained from the minimal one by blowing up singular or
non-normal-crossing points of the exceptional divisor. These form a finite invariant set at each
stage, so every local flow fixes the required centers. In what follows $\pi$ is one of these two
resolutions. Every holomorphic field on $V$ vanishes at $\0$: its local flow consists of
automorphisms of the germ, which fix its only singular point. If $v$ has an isolated zero at $\0$,
the zero set of $\widetilde v$ is contained in $E$, and we denote by $\operatorname{Ind}(\widetilde
v;X)$ its total Poincar\'e--Hopf index: the sum of the indices of any smooth vector field on
$X_\varepsilon=\pi^{-1}(V\cap\mathbb B_\varepsilon)$ with isolated zeros which coincides with
$\widetilde v$ near the boundary.

\begin{teo}\label{thm:gsv-resolution}
For every holomorphic vector field $v$ on $V$ with an isolated zero at $\0$,
\begin{equation}\label{eq:rad-resolution}
 \operatorname{Ind}_{\rm rad}(v;V,\0)=1+\operatorname{Ind}(\widetilde v;X)-\chi(E).
\end{equation}
If $(V,\0)$ is an isolated hypersurface singularity, then
\begin{equation}\label{eq:gsv-resolution}
 \operatorname{Ind}_{\rm GSV}(v;V,\0) =\operatorname{Ind}(\widetilde v;X)-\chi(E) +1+\mu .
\end{equation}
\end{teo}

\begin{proof}
Let $M$ be a compact manifold with boundary and $u$ a vector field on $M$ without zeros on
$\partial M$. By the Poincar\'e--Hopf theorem for manifolds with boundary, $\operatorname{Ind}(u;M)-
\chi(M)$ depends only on the homotopy class of $u$ on a collar of $\partial M$, and for a field
pointing outward it vanishes. Apply this to $X_\varepsilon$, whose boundary is the link and which
retracts onto $E$. A collar of $\partial X_\varepsilon$ is identified with a collar of the link in
$V$, where the lift is $v$. Comparing $v$ with the outward radial field shows that
$\operatorname{Ind}(\widetilde v;X)-\chi(E)$ is the degree of $v$ on the link relative to the radial
field, which is $\Ind_{\rm rad}(v)-1$. This proves \eqref{eq:rad-resolution}. For a hypersurface,
$\Ind_{\rm GSV}=\Ind_{\rm rad}+\chi(F)-1$ \cite{BarSeaSuw} and $\chi(F)=1+\mu$, which gives
\eqref{eq:gsv-resolution}.
\end{proof}

The same boundary comparison gives
$\Ind_{\rm rad}(v;V,\0)=1+\operatorname{Ind}(\widetilde v;X)-\chi(E)$ for every isolated singularity
of any dimension $n$ and every resolution to which $v$ lifts, and hence
$\Ind_{\rm GSV}(v;V,\0)=\operatorname{Ind}(\widetilde v;X)-\chi(E)+\chi(F)$ for isolated
hypersurfaces and complete intersections. For $n=1$ and the normalization, $E$ consists of $r$
points and $\operatorname{Ind}(\widetilde v;X)=\sum_i \operatorname{ord}_{t_i}a_i$, so this is
\eqref{eq:normalization-gsv}.

To compute the index on $X$, let $Z$ be the divisorial part of the zero scheme of $\widetilde v$. It
is an effective divisor supported on $E$, and $\widetilde v=\sigma_Z\,w$, where $\sigma_Z$ is the
canonical section of $\mathcal O_X(Z)$ and $w$ is a section of $\Theta_{X}(-Z)$ with isolated zeros.
Let $N(w)$ be the sum of the local lengths of the residual zero scheme of $w$, supported on $E$.

\begin{prop}\label{prop:index-on-resolution}
With this notation,
\begin{equation}\label{eq:excess-surface}
 \operatorname{Ind}(\widetilde v;X) =N(w)-Z\cdot(Z+\kappa_X) .
\end{equation}
\end{prop}

\begin{proof}
The total index is the second Chern class of $\Theta_{X}$ localized at the zero set of $\widetilde
v$. Since $\Theta_{X}=\Theta_{X}(-Z)\otimes\mathcal O(Z)$, the residual intersection formula
\cite[Sections~9.2 and 14.1]{Fulton} gives
$$
\operatorname{Ind}(\widetilde v;X)=N(w)+c_1(\Theta_{X}(-Z))\cdot Z+Z^2,
$$
where the intersection numbers are computed on the compact support of $Z$, and
$c_1(\Theta_X(-Z))=-\kappa_X-2Z$.
\end{proof}

On a normal surface singularity, Theorem~\ref{thm:BGM} gives the existence of $K_{\rm rad}(V)$ and
the openness and density of the generic set. The resolution gives these facts a direct proof, and
two stronger properties which we use below: the behaviour of the index along lines, which extends
Proposition~\ref{prop:BGM-complements}(i) beyond hypersurfaces, and finite determinacy. Let $\Theta$ be the space of germs of holomorphic fields on $V$ with an isolated
zero, with the topology of Section~\ref{sec:generic}.

\begin{prop}\label{prop:radial-generic}
Let $(V,\0)$ be a normal surface singularity.
\begin{enumerate}
\item[(i)] The radial index is bounded below on $\Theta$; let $K_{\rm rad}(V)$ be its minimum.
\item[(ii)] The set $\mathcal G\subset\Theta$ of fields of radial index $K_{\rm rad}(V)$ is open.
\item[(iii)] If $v\in\Theta$ and $w\in\mathcal G$, then $v+tw\in\mathcal G$ for all $t$ outside a
closed discrete subset of $\CC$. In particular $\mathcal G$ is dense.
\item[(iv)] For $v\in\Theta$ there is $N$ such that every field $v'$ with
$v'-v\in\mathfrak m^N\Der(\OO_V)$ has an isolated zero and the same radial index as $v$.
\end{enumerate}
\end{prop}

\begin{proof}
(i) By Theorem~\ref{thm:gsv-resolution} and Proposition~\ref{prop:index-on-resolution},
$\Ind_{\rm rad}(v)=1-\chi(E)+N(w)-Z\cdot(Z+\kappa_X)$ with $N(w)\ge0$. The intersection form on
the cycles supported on $E$ is negative definite, so $-Z^2$ grows quadratically in the coefficients
of $Z$, while $Z\cdot\kappa_X$ grows at most linearly. Hence $-Z\cdot(Z+\kappa_X)$ is bounded below
on effective cycles supported on $E$.

(ii) The radial index of a field at $\0$ plus the Poincar\'e--Hopf indices of its other zeros in
$V\cap\mathbb B_s$ equals one plus the degree of the field on $V\cap\mathbb S_s$ relative to the
radial field, and this degree is locally constant under small perturbations that do not vanish on
$V\cap\mathbb S_s$. If $v_j\to v$ uniformly, then for $j$ large $v_j$ vanishes at $\0$, has no zeros
on $V\cap\mathbb S_s$, and its other zeros in $V\cap\mathbb B_s$ lie on the smooth part, where their
indices are positive. Hence $\Ind_{\rm rad}(v_j)\le\Ind_{\rm rad}(v)$, with equality if
$v\in\mathcal G$.

(iii) If $v$ and $w$ are linearly dependent at every point of $V\setminus\{\0\}$, then $v=hw$ with $h$
holomorphic on $V\setminus\{\0\}$, hence on $V$ because $V$ is normal, and $v+tw=(h+t)w$ lies in
$\mathcal G$ for $t\ne-h(\0)$. Otherwise the set $A\subset V\setminus\{\0\}$ where $v$ and $w$ are
dependent is a curve, and on $A$ the zeros of $v+tw$ are the points where the meromorphic function
$-v/w$ takes the value $t$. Near $\0$, $A$ consists of finitely many branches through $\0$, and on
each branch this function has a limit $t_b\in\PP^1$. For $t$ different from the finitely many $t_b$,
$v+tw$ has an isolated zero at $\0$. Fix such a $t_0$. For $t$ close to $t_0$, the zeros of $v+tw$
near $\0$ other than $\0$ lie on the branches with $t_b=t_0$, their number is constant for
$0<|t-t_0|$ small, and so are their Poincar\'e--Hopf indices, because these are the multiplicities
of holomorphic families parametrized by the normalizations of the branches. By (ii), the radial
index is therefore constant on a punctured neighbourhood of every $t_0$, and it can only jump up at
$t_0$. It follows that the set of $t$ where the index is not equal to its generic value is closed and
discrete, and its complement is connected. For $|t|$ large, $v+tw=t(w+t^{-1}v)$ lies in the open set
$\mathcal G$, and multiplying by a nonzero constant does not change the index. Hence the generic
value is $K_{\rm rad}(V)$. Letting $t\to0$ along good values shows that $v$ is a limit of fields in
$\mathcal G$.

(iv) By the \L{}ojasiewicz inequality there are $c,a>0$ with $|v(p)|\ge c|p|^a$ on $V$ near $\0$,
while $|v'(p)-v(p)|\le C|p|^N$. For $N>a$ the fields $v+s(v'-v)$, $0\le s\le1$, do not vanish on
$V\cap\mathbb S_\varepsilon$ for small $\varepsilon$. Hence $v'$ has an isolated zero and the same
degree on the link relative to the radial field, that is, the same radial index.
\end{proof}

For an isolated hypersurface singularity, $\Ind_{\rm GSV}$ and $\Ind_{\rm rad}$ differ by $\mu$, so
$\mathcal G$ is also the set of GSV-generic fields of Theorem~\ref{thm:BGM}.

\medskip
For an irreducible component $E_i$ of $E$, let $m_i(V)$ be the minimum vanishing order along $E_i$
of the lifts of holomorphic vector fields on $V$, and put $Z_V=\sum_i m_i(V)E_i$.

\begin{lema}\label{lem:surface-generic-minimum}
The generic fields whose lift has divisorial cycle exactly $Z_V$ are dense in $\Theta$.
\end{lema}

\begin{proof}
Let $\xi_1,\ldots,\xi_s$ generate $\Der_{\CC}(\OO_{V,\0})$. A field $\sum_jg_j\xi_j$ has order
along $E_i$ at least $\min_j\operatorname{ord}_{E_i}\widetilde\xi_j$, because multiplication by
$g_j\in\OO_{V,\0}$ cannot lower an exceptional valuation. So some generator has order $m_i(V)$ along
$E_i$, and the constant combinations $u_a=\sum_ja_j\xi_j$ with order exactly $m_i(V)$ along every
$E_i$ form the complement of finitely many proper linear subspaces of $\CC^s$. For generic $a$,
$u_a$ also has an isolated zero: on $V\setminus\{\0\}$ the generators span the tangent spaces, so the
incidence set $\{(a,p):u_a(p)=0\}$ is a vector bundle of rank $s-2$ over $V\setminus\{\0\}$, of
dimension $s$, and the parameters $a$ for which $u_a$ vanishes on a curve form a set of dimension
less than $s$. Fix such a field $u$, with lifted cycle $Z_V$.

Let $w\in\mathcal G$. Along each $E_i$, the leading coefficient of $\widetilde u+t\widetilde w$ in
order $m_i(V)$ vanishes for at most one value of $t$, so $\widetilde u+t\widetilde w$ has cycle $Z_V$
for all but finitely many $t$. By Proposition~\ref{prop:radial-generic}(iii), $u+tw\in\mathcal G$
for all $t$ outside a closed discrete set. Choose $t_k\to\infty$ avoiding both exceptional sets.
Then $w+t_k^{-1}u=t_k^{-1}(u+t_kw)$ is generic with cycle $Z_V$ and converges to $w$. Since
$\mathcal G$ is dense, the claim follows.
\end{proof}

\begin{remark}
We do not know whether a generic \emph{constant} combination of generators is always generic; the
lemma only shows that generic fields can be chosen with the least possible divisorial cycle.
\end{remark}

Let $N_V$ denote the minimum residual length $N(w)$ among fields with isolated zero whose lifted
divisorial cycle is $Z_V$.

\begin{teo}\label{thm:optimal-surface}
For a normal surface singularity $(V,\0)$,
\begin{equation}\label{eq:K-surface-rad}
 K_{\rm rad}(V)=1-\chi(E)-Z_V\cdot(Z_V+\kappa_X)+N_V .
\end{equation}
The generic fields with cycle $Z_V$ are exactly the fields with cycle $Z_V$ and residual length
$N_V$. If $(V,\0)$ is an isolated hypersurface singularity, then
\begin{equation}\label{eq:K-surface}
 K=1+\mu-\chi(E)-Z_V\cdot(Z_V+\kappa_X)+N_V .
\end{equation}
\end{teo}

\begin{proof}
For a field with cycle $Z_V$, Theorem~\ref{thm:gsv-resolution} and
Proposition~\ref{prop:index-on-resolution} give
\[
 \Ind_{\rm rad}(v;V,\0)=1-\chi(E)-Z_V\cdot(Z_V+\kappa_X)+N(w).
\]
By Lemma~\ref{lem:surface-generic-minimum} some field with cycle $Z_V$ has index $K_{\rm rad}$,
so $K_{\rm rad}$ is at least the right side of \eqref{eq:K-surface-rad}. A field with cycle $Z_V$
and residual length $N_V$ has index equal to the right side, which is therefore at least
$K_{\rm rad}$. This proves \eqref{eq:K-surface-rad} and the characterization. For a hypersurface,
$K=K_{\rm rad}+\mu$.
\end{proof}

\subsection{An algebraic formula for the residual length}\label{ss:residual}
In this subsection $(V,\0)=\{f=0\}\subset(\CC^3,\0)$ is an isolated hypersurface singularity.
Let $R=\OO_{\CC^3,\0}$, $J(f)=(f_x,f_y,f_z)$, and let
\[
 \mathfrak c_f=(J(f):f)
\]
be the ideal of cofactors of ambient logarithmic vector fields. As recalled in
Section~\ref{sec:prelim},
\[
 \ell(R/\mathfrak c_f)=\mu-\tau.
\]
For a field $v$ on $V$, choose an ambient logarithmic extension
$\widehat v=(v_1,v_2,v_3)$ with isolated zero, write
$I_v=(v_1,v_2,v_3)$ and $\widehat v(f)=k_vf$, and set
\[
 \rho(V)=\min_{\widehat v}\ell\bigl(R/(I_v:k_v)\bigr),
\]
where the minimum ranges over such extensions of fields with isolated zero on $V$.
By Gómez-Mont's formula~\eqref{eq888}, $K=\tau+\rho(V)$. Define the arithmetic
genus of an exceptional cycle $Z$ by adjunction,
\[
 p_a(Z)=1+\frac12 Z\cdot(Z+\kappa_X).
\]

\begin{prop}\label{prop:residual-colength}
For an isolated hypersurface surface singularity,
\begin{equation}\label{eq:NV-colength}
 N_V=\rho(V)-\ell(R/\mathfrak c_f)+\chi(E)+2p_a(Z_V)-3.
\end{equation}
Equivalently,
\[
 N_V=\rho(V)-(\mu-\tau)+\chi(E)+2p_a(Z_V)-3.
\]
\end{prop}

\begin{proof}
Combine $K=\tau+\rho(V)$ with~\eqref{eq:K-surface}. Adjunction gives
$-Z_V\cdot(Z_V+\kappa_X)=2-2p_a(Z_V)$, and
$\ell(R/\mathfrak c_f)=\mu-\tau$. Rearranging yields~\eqref{eq:NV-colength}.
\end{proof}

Thus the residual length is controlled by two algebraic features of the
logarithmic derivation module: the minimum colon colength $\rho(V)$ and the
colength of its cofactor ideal. The formula does not identify $\rho(V)$ with a
standard module multiplicity; that question remains open.

By Theorem~\ref{thm:optimal-surface}, the generic fields with cycle $Z_V$, which are dense by
Lemma~\ref{lem:surface-generic-minimum}, all have residual length $N_V$. These data describe the
generic field along the exceptional divisor: $Z_V$ records its common divisorial
vanishing, while $N_V$ records the least possible residual zero length. The result fixes this
length, not the residual zero scheme itself.

\begin{remark}\label{rem:defect}
The numbers $Z_V$ and $N_V$ are determined by the image of the module of vector fields on $V$ in
$\Theta_X$. By Theorem~\ref{thm:gsv-resolution}, $K_{\rm rad}(V)-1$ is the minimum of
$\operatorname{Ind}(\widetilde v;X)-\chi(E)$; for a hypersurface this is the defect
$D=K-(1+\mu)$. In every dimension and for every resolution to which $v$ lifts, the same boundary
comparison gives $\Ind_{\rm rad}(v)-1=\operatorname{Ind}(\widetilde v;X)-\chi(E)$, and for isolated
hypersurfaces and complete intersections
$\Ind_{\rm GSV}(v)=\operatorname{Ind}(\widetilde v;X)-\chi(E)+\chi(F)$.
\end{remark}

\subsection{The logarithmic index}\label{ss:log}

\hyphenation{nor-ma-li-za-tion}

Let now $(V,\0)$ be an isolated singularity of any dimension $n$, and $\pi:X\to V$ a resolution
whose exceptional set $E$ is a simple normal crossing divisor and to which vector fields lift: the
normalization for $n=1$, the minimal good resolution for a normal surface, and in general a
functorial resolution, which commutes with smooth morphisms and therefore lifts derivations
\cite{Kollar}. Lifted fields preserve $E$, so they are sections of the
logarithmic tangent sheaf $\Theta_{X}(-\log E)$, which is locally free.

For a lifted field $\widetilde v$, the \emph{logarithmic index}
$\Ind_{\log}(\widetilde v;X,E)$ is the degree of the top Chern class
of $\Theta_X(-\log E)$ localized at the zero locus of the section
$\widetilde v\in H^0(X,\Theta_X(-\log E))$. All these zeros lie over
the singular point, because the logarithmic and ordinary tangent
bundles agree away from $E$. The definition is valid even if the
logarithmic zero locus is not zero-dimensional; in that case the
localized Chern class, rather than a sum of local lengths, is intended.

For a surface this definition has a concrete local description. At a
smooth point of $E$, choose coordinates with $E=\{x=0\}$ and write
\[
 \widetilde v=x a(x,y)\,\partial_x+b(x,y)\,\partial_y.
\]
Its logarithmic coefficient pair is $(a,b)$. At a crossing, where
$E=\{xy=0\}$, write
\[
 \widetilde v=x a(x,y)\,\partial_x+y b(x,y)\,\partial_y;
\]
again the logarithmic coefficient pair is $(a,b)$. If its common zero
is isolated, its local contribution is
$\dim_{\CC}\mathcal O_{X,p}/(a,b)$. In particular, a zero of
$\widetilde v$ as an ordinary vector field need not be a logarithmic
zero: the normal factor $x$ (or $xy$) has been removed from the local
logarithmic frame. This is why the logarithmic index detects the
singularities of the lifted field more finely than its ordinary zero
scheme.

\begin{lettered}{C}\label{thm:log}\label{thmC}
For every holomorphic vector field $v$ on $V$ with an isolated zero,
\[
 \Ind_{\rm rad}(v;V,\0)=1+\Ind_{\log}(\widetilde v),
 \qquad\text{hence}\qquad
 K_{\rm rad}(V)=1+\min_v\Ind_{\log}(\widetilde v)
\]
whenever the minimum exists, in particular for curves and normal surfaces. If $(V,\0)$ is an
isolated hypersurface or complete intersection singularity, then
\[
 \Ind_{\rm GSV}(v;V,\0)=\chi(F)+\Ind_{\log}(\widetilde v),
 \qquad\text{hence}\qquad
 K=\chi(F)+\min_v\Ind_{\log}(\widetilde v).
\]
\end{lettered}

\begin{proof}
By Remark~\ref{rem:defect},
$\Ind_{\rm rad}(v)-1=\Ind(\widetilde v;X)-\chi(E)$, and for hypersurfaces and complete
intersections $\Ind_{\rm GSV}(v)-\chi(F)$ is the same number. The residue sequence
$0\to\Theta(-\log E)\to\Theta\to\bigoplus_i\OO_{E_i}(E_i)\to0$ and
$c(\OO_{E_i}(E_i))=1+E_i$ give $c(\Theta)=c(\Theta(-\log E))\prod_i(1+E_i)$. Both bundles agree
near the boundary, with the same section $\widetilde v$, so the difference of the two localized top
Chern classes is the compactly supported number
\[
 \int\bigl(c_n(\Theta)-c_n(\Theta(-\log E))\bigr)=\chi(E).
\]
For $n=1$ both sides equal the number of points of $E$. For $n=2$, since
$c_1(\Theta(-\log E))=-\kappa_X-E$, the left side is
$c_1(\Theta(-\log E))\cdot E+\sum_{i<j}E_iE_j=\sum_i(-\kappa_X\cdot E_i-E_i^2)-\#\{\text{nodes}\}=\chi(E)$.
In general this is the logarithmic Gauss--Bonnet formula: the class
$c(\Theta(-\log E))\cap[X]$ is the Chern--Schwartz--MacPherson class of the complement of
$E$ \cite{Aluffi}, so the difference is $\chi(E)$. Hence
$\Ind(\widetilde v;X)=\chi(E)+\Ind_{\log}(\widetilde v)$.
\end{proof}

\begin{itemize}
    \item \textbf{Curves ($n=1$).} On the normalization $X$, the divisor $E$ consists of $r$ points $p_1,\dots,p_r$, so $\Theta_X(-\log E) = \Theta_X(-\sum_{i=1}^r p_i)$ is locally generated by $t_i\partial_{t_i}$. Writing $a_i\partial_{t_i} = (a_i/t_i)\,t_i\partial_{t_i}$ shows that its local logarithmic index at $p_i$ is $\operatorname{ord}_{t_i}(a_i) - 1$. Hence $\min \operatorname{Ind}_{\log} = \sum_{i=1}^r (\alpha_i - 1)$, which recovers Theorem B via Theorem C.

    \item \textbf{Weighted homogeneous germs.} For a weighted homogeneous germ, the Euler vector field is outward-pointing and transverse to the link, so its radial index is $1$. By Theorem C, its logarithmic index vanishes on every good resolution to which it lifts logarithmically.
\end{itemize}
Thus, the defect is not an \emph{ad hoc} correction: it is the logarithmic Poincar\'e--Hopf index of a lifted generic holomorphic vector field, and it vanishes for the Euler field.

\begin{prop}[Surface characterization of the logarithmic index]
\label{prop:log-surface-characterization}
For a good resolution of a normal surface singularity,
let $Z$ be the divisorial zero cycle of a lifted field and let $N(w)$
be the residual zero length as in Proposition~\ref{prop:index-on-resolution}.
Then
\[
 \Ind_{\log}(\widetilde v;X,E)
 =N(w)-Z\cdot(Z+\kappa_X)-\chi(E).
\]
If the logarithmic zero scheme is finite, the left-hand side is also
the sum of the local lengths of the logarithmic coefficient pairs
$(a,b)$ described above. Thus the first expression is a global
intersection-theoretic calculation, and the second is a local
calculation whenever the logarithmic zeros are isolated.
\end{prop}

\begin{proof}
By Theorem~\ref{thm:log} and the resolution comparison,
$\Ind_{\log}(\widetilde v)=\Ind(\widetilde v;X)-\chi(E)$.
Proposition~\ref{prop:index-on-resolution} gives the asserted formula.
For isolated logarithmic zeros, the localized top Chern class of a
rank-two bundle is the sum of the lengths of the local zero schemes
of its section.
\end{proof}

The characterization distinguishes the invariant from a mere change
of notation: it can be computed locally from the logarithmic
coefficients when their common zeros are isolated, or globally from
the exceptional divisor and the residual intersection number. The
minimum $\Ind_{\log}(V)$ is taken over fields on $V$, not over arbitrary
sections of $\Theta_X(-\log E)$.

\medskip
The logarithmic index has a description on the link $L=V\cap\mathbb S_\varepsilon$. Let
$\Ind_{\rm rad}(v;V,\0)$ be the radial index of a field with an isolated zero
\cite{ASV,EG1}: choose a nonvanishing field on the cylinder $V\cap(\mathbb B_\varepsilon\setminus
\mathbb B_\delta)$ that agrees with $v$ near the outer boundary and with the outward radial field near the inner boundary, take
a generic extension with isolated zeros, and add $1$ to the sum of their Poincar\'e--Hopf indices.
Call $d(v)=\Ind_{\rm rad}(v)-1$ the signed radial defect. It is the relative degree of the boundary field $v|_L$ with respect
to the radial field; it depends only on the homotopy class of $v|_L$ among nonvanishing sections of
$TV|_L$. Recall that $K_{\rm rad}(V)$ is the minimum of $\Ind_{\rm rad}$ over holomorphic fields on $V$
with an isolated zero, and that $\Ind_{\rm GSV}=\Ind_{\rm rad}+\chi(F)-1$ for isolated hypersurfaces
and complete intersections \cite{BarSeaSuw}.

\begin{prop}\label{prop:radial-link}
For every holomorphic field $v$ on $V$ with an isolated zero and every resolution as above,
\[
 d(v)=\Ind_{\rm rad}(v;V,\0)-1=\Ind_{\log}(\widetilde v),
 \qquad\text{hence}\qquad
 K_{\rm rad}(V)-1=\Ind_{\log}(V),
\]
and for hypersurfaces both numbers equal the defect $D(V)$. Thus the generic fields are exactly the
holomorphic fields of least relative degree. If $n\ge2$, the relative degree
is a complete invariant: all minimizing fields restrict to sections in one homotopy class on $L$,
and among holomorphic fields this is precisely the class of least relative degree.
\end{prop}

\begin{proof}
The identities are Theorem~\ref{thm:log}, and for hypersurfaces $D=K-\chi(F)=K_{\rm rad}-1$. For $n\ge2$ the link is connected
\cite{mimi} and of real dimension $2n-1$, while the sphere bundle of $TV|_L$ has fibre $S^{2n-1}$.
By obstruction theory, two sections are homotopic over the $(2n-2)$-skeleton, and the only
obstruction to a homotopy lies in $H^{2n-1}(L;\pi_{2n-1}(S^{2n-1}))\cong\ZZ$. This class is the
relative degree.
\end{proof}

\medskip
Proposition~\ref{prop:bendixson} and Corollary~\ref{cor:petals} are the one-dimensional case of a
statement that holds in every dimension, and which follows from Pugh's formula \cite{Pugh}. Let $u$
be a field on a Milnor fibre $F$ that agrees near $\partial F$ with the field used to define the GSV
index of $v$, let $L_-(u)\subset\partial F$ be the region where $u$ points inward, and suppose that
the tangential component of $u$ has isolated zeros on $L_-(u)$. If $b(u)$ is the sum of their
indices, Pugh's formula gives 
$$\Ind_{\rm GSV}(v)+b(u)=\chi(F),$$ 

\noindent then $ b(u)=-d(v)$, and for a generic field $b(u)=\chi(F)-K(V)=-D(V)$. Thus generic fields are those whose boundary tangencies have the least possible signed count, and
for a weighted homogeneous germ this count is zero and is realized without tangencies by the Euler
field. In dimension one every tangency is internal and $b$ counts petals. In higher dimensions $b$
is a signed count, the configuration of $L_-(u)$ and of the tangential zeros is not canonical, and
a vanishing count does not mean that tangencies can be avoided: the surface $V_3$ of
Remark~\ref{rem:homotopic-transversal} has $D=0$, yet no holomorphic field on it is transversal to a
small link. On surfaces the petals reappear transversally to the dicritical exceptional curves; see
Remark~\ref{rem:surface-petals}. A simpler and canonical description, by a winding number along a
curve of the link, is given in Theorem~\ref{thmF}.

\begin{remark}[Homotopic to radial versus transversal]\label{rem:homotopic-transversal}
A field transversal to the link is homotopic to the radial field there, so it has $d(v)=0$ and
$\Ind_{\rm GSV}(v)=\chi(F)$. For a weighted homogeneous germ the Euler field is such a field, and
$K_{\rm rad}=1$. In odd dimension the converse holds: $K_{\rm rad}=1$ means $K=1-\mu$, and
Theorem~\ref{thmA} gives $1-\mu\ge1-\tau$, hence $\mu=\tau$ and the germ is weighted homogeneous
\cite{saito}. In even dimension the converse fails. By Theorem~\ref{thm:cyclic-e12}, the surface
$V_3=\{u^3+x^3+y^7+xy^5=0\}$ has $D=0$, so $K_{\rm rad}=1$: a generic field is homotopic to the
radial field on the link. However $\mu-\tau=2$, so $V_3$ is not weighted homogeneous, and by
\cite{MS2} no holomorphic field on $V_3$ is transversal to a small link, in any coordinates. In
even dimension, $K_{\rm rad}=1$ is equivalent to $e(V)=\mu-\tau$.
\end{remark}

The geometric problem left open is to describe, in dimension at least three, the zero scheme on $E$
of the lift of a generic field: which divisorial part and which residual logarithmic zeros realize
$\Ind_{\log}(V)$. For curves the answer is the orders $\alpha_i$, and for surfaces it is the pair
$(Z_V,N_V)$ of Theorem~\ref{thm:optimal-surface}.

\subsection{The defect and the foliation of the lifted field}\label{ss:structure}

Let $n=2$ and let $\pi:X\to V$ be the minimal good resolution, with exceptional divisor
$E=\bigcup_iE_i$: the $E_i$ are smooth compact curves meeting transversally, at points we call
\emph{nodes}, with no triple points. Then $\chi(E)=\sum_i\chi(E_i)-\#\{\text{nodes}\}$ and, by
adjunction, $\chi(E_i)=-E_i^2-\kappa_X\cdot E_i$. Intersection numbers of divisors supported on $E$ are
well defined, and so are their intersection numbers with any line bundle on $X$.

Let $v$ have an isolated zero, with lift $\widetilde v$. As in
Proposition~\ref{prop:index-on-resolution} write $\widetilde v=\sigma_Z\,w$ with $Z=\sum_im_iE_i$,
$m_i\ge0$, and $w$ a section of $\Theta_X(-Z)$ with isolated zeros, all on $E$. The section $w$
defines a holomorphic foliation $\mathcal F$ on $X$ with isolated singularities, with tangent bundle
$T_{\mathcal F}=\OO_X(Z)$ and normal bundle $N_{\mathcal F}=-\kappa_X-Z$. At a point $p$, if
$w=a\,\partial_x+b\,\partial_y$ in local coordinates, the Milnor number of $\mathcal F$ is
$\mu_p(\mathcal F)=\ell(\OO_p/(a,b))$, and Proposition~\ref{prop:index-on-resolution} reads
\begin{equation}\label{eq:index-foliation}
 \operatorname{Ind}(\widetilde v;X)=\sum_{p\in E}\mu_p(\mathcal F)+T_{\mathcal F}\cdot N_{\mathcal F}.
\end{equation}
The lift $\widetilde v$ is tangent to every $E_i$, since its flow preserves $E$ and hence each of its
components. A component $E_i$ is either $\mathcal F$-invariant or not; in the second case we call it
\emph{dicritical}. On a dicritical component $w$ is generically transverse to $E_j$ while
$\widetilde v$ is tangent to it, so $m_j\ge1$.

For a compact smooth curve $C\subset X$ we use the following two numbers. If $C$ is invariant, the
map $T_{\mathcal F}|_C\to T_X|_C$ factors through $T_C$ and gives a nonzero section of
$T_C\otimes T_{\mathcal F}^*|_C$; let $Z(\mathcal F,C,p)$ be its order at $p$ and $Z(\mathcal F,C)$
their sum. If $C$ is not invariant, the composite $T_{\mathcal F}|_C\to T_X|_C\to N_C$ is a nonzero
section of $N_C\otimes T_{\mathcal F}^*|_C$; let $\operatorname{tang}(\mathcal F,C,p)$ be its order
at $p$ and $\operatorname{tang}(\mathcal F,C)$ their sum. These are the indices of Brunella
\cite[Chapter~2]{Bru}. Taking degrees,
\begin{equation}\label{eq:brunella}
 T_{\mathcal F}\cdot C=
 \begin{cases}
 \chi(C)-Z(\mathcal F,C) & \text{if } C\text{ is invariant},\\
 C^2-\operatorname{tang}(\mathcal F,C) & \text{if } C\text{ is dicritical}.
 \end{cases}
\end{equation}

\subsubsection*{Local data}
Let $p\in E$ and let $E'$ be the union of the invariant components. In coordinates $(x,y)$ centred
at $p$ in which the components of $E$ through $p$ are coordinate axes, write
\[
 w=a\,x^{\varepsilon_1}\partial_x+b\,y^{\varepsilon_2}\partial_y,
\]
where $\varepsilon_1=1$ if $\{x=0\}$ is an invariant component through $p$ and $\varepsilon_1=0$
otherwise, and similarly for $\varepsilon_2$ and $\{y=0\}$; this is possible because $w$ is tangent
to the invariant components. Put $\mu_{\log}(w,p)=\ell(\OO_p/(a,b))$. For an invariant component
$E_i$ through $p$, the \emph{tangential coefficient} $a_i$ is the coefficient of the vector field
along $E_i$ in this frame ($a$ if $E_i=\{y=0\}$, $b$ if $E_i=\{x=0\}$), and $\ord_p(a_i|_{E_i})$
is the order of its restriction to $E_i$. Set
\[
 \lambda_p=\mu_{\log}(w,p)+\sum_{E_i\ni p\ \text{invariant}}m_i\,\ord_p\bigl(a_i|_{E_i}\bigr)\ \ge0 .
\]
It vanishes for all but finitely many $p$. These numbers are finite: if $(a,b)$ had a common factor,
or $a_i|_{E_i}\equiv0$, then $w$ would vanish along a curve.

\begin{teo}[Structure of the defect]\label{thm:structure}
For every field $v$ with an isolated zero on $(V,\0)$,
\begin{multline*}
 \Ind_{\rm rad}(v)-1
 =\sum_{p\in E}\lambda_p+\#\{\text{nodes between dicritical components}\}\\
 +\sum_{E_j\ \mathrm{dicritical}}\Bigl[(m_j-1)\bigl(\chi(E_j)-s_j\bigr)
 +m_j\operatorname{tang}(\mathcal F,E_j)\Bigr],
\end{multline*}
where $s_j$ is the number of nodes of $E_j$ with invariant components. If $(V,\0)$ is an isolated
hypersurface singularity, the left side equals $\Ind_{\rm GSV}(v)-\chi(F)$.
\end{teo}

\begin{proof}
By Remark~\ref{rem:defect} and \eqref{eq:index-foliation}, the left side is
$\sum_p\mu_p(\mathcal F)+T_{\mathcal F}\cdot N_{\mathcal F}-\chi(E)$. Since
$T_{\mathcal F}=\OO(Z)$,
$T_{\mathcal F}\cdot N_{\mathcal F}=\sum_im_i\,N_{\mathcal F}\cdot E_i$, and by adjunction
$N_{\mathcal F}\cdot E_i=(-\kappa_X-Z)\cdot E_i=\chi(E_i)+E_i^2-T_{\mathcal F}\cdot E_i$.
By \eqref{eq:brunella},
\[
 N_{\mathcal F}\cdot E_i=
 \begin{cases}
 E_i^2+Z(\mathcal F,E_i) & \text{if } E_i\text{ is invariant},\\
 \chi(E_i)+\operatorname{tang}(\mathcal F,E_i) & \text{if } E_i\text{ is dicritical}.
 \end{cases}
\]
For an invariant $E_i$, \eqref{eq:brunella} and $Z\cdot E_i=m_iE_i^2+\sum_{q}m_{j(q)}$, the sum
over the nodes $q$ of $E_i$ with $E_{j(q)}$ the other component, give
$m_iE_i^2=\chi(E_i)-Z(\mathcal F,E_i)-\sum_qm_{j(q)}$. Therefore
\[
\begin{split}
 \Ind_{\rm rad}(v)-1={}&\sum_p\mu_p(\mathcal F)+\#\{\text{nodes}\}\\
 &+\sum_{E_i\ \mathrm{inv}}\Bigl[(m_i-1)Z(\mathcal F,E_i)-\sum_{q\in E_i}m_{j(q)}\Bigr]\\
 &+\sum_{E_j\ \mathrm{dic}}\Bigl[(m_j-1)\chi(E_j)+m_j\operatorname{tang}(\mathcal F,E_j)\Bigr].
\end{split}
\]
We distribute the terms over the points of $E$ and compare with $\lambda_p$:\\

\noindent  (i) \emph{  A point $p$ of a single component, invariant.} With $E_i=\{y=0\}$, $w=a\partial_x+by\partial_y$,
$\mu_p=\ell(a,yb)=\ell(a,y)+\ell(a,b)$ and $Z(\mathcal F,E_i,p)=\ell(a,y)=\ord_p(a|_{E_i})$. The
local contribution $\mu_p+(m_i-1)Z(\mathcal F,E_i,p)$ equals $\lambda_p$.\\

\noindent  (ii) \emph{A point of a single component, dicritical.} The contribution is $\mu_p=\lambda_p$. \\

\noindent  (iii) \emph{A node $q$ between invariant components $E_i=\{y=0\}$, $E_j=\{x=0\}$.} Here
$w=ax\partial_x+by\partial_y$, $\mu_q=\ell(xa,yb)=1+\ord_q(a|_{E_i})+\ord_q(b|_{E_j})+\ell(a,b)$, and
$Z(\mathcal F,E_i,q)=1+\ord_q(a|_{E_i})$, $Z(\mathcal F,E_j,q)=1+\ord_q(b|_{E_j})$. The node receives
$-m_j$ from $E_i$, $-m_i$ from $E_j$ and $+1$. The total contribution is
$m_i\ord_q(a|_{E_i})+m_j\ord_q(b|_{E_j})+\ell(a,b)=\lambda_q$. \\

\noindent  (iv) \emph{A node $q$ between an invariant $E_i=\{y=0\}$ and a dicritical $E_j$.} Here
$w=a\partial_x+by\partial_y$, $\mu_q=\ell(a,y)+\ell(a,b)$ and $Z(\mathcal F,E_i,q)=\ell(a,y)$. The node
receives $-m_j$ from $E_i$ and $+1$. The total contribution is $\lambda_q+1-m_j$. Summing $1-m_j$
over the $s_j$ such nodes of $E_j$ and adding $(m_j-1)\chi(E_j)$ gives
$(m_j-1)(\chi(E_j)-s_j)$. \\

\noindent  (v) \emph{A node between dicritical components.} The contribution is $\mu_q+1=\lambda_q+1$.\\

\noindent Adding all contributions proves the formula.
\end{proof}

\begin{cor}\label{cor:positivity}
If every dicritical component $E_j$ of the lift of some generic field satisfies $m_j=1$ or
$\chi(E_j)\ge s_j$ (equivalently, $E_j$ is rational with at most two invariant neighbours,
or elliptic with no invariant neighbours), then $K_{\rm rad}(V)\ge1$; for a hypersurface this
means $D\ge0$, that is, $K\ge1+\mu$. In particular this holds if the lift of a generic field has no dicritical
components. In general $K_{\rm rad}(V)<1$, or $D<0$, can only come from a dicritical component of positive genus or with at
least three invariant neighbours, along which the generic field vanishes to order at least two.
\end{cor}

\begin{proof}
Apply Theorem~\ref{thm:structure} to a generic field; all terms except
$(m_j-1)(\chi(E_j)-s_j)$ are nonnegative.
\end{proof}

\medskip
For surfaces, Theorems~\ref{thmA}, \ref{thm:log}, \ref{thm:optimal-surface} and
\ref{thm:structure} combine into the following statement. Here $Z_V$ and $N_V$ are as in
Theorem~\ref{thm:optimal-surface}. For a generic field $v$ put
\[
 Q(v)=\sum_{E_j\ \text{dicritical}}(m_j-1)\max\bigl(s_j-\chi(E_j),0\bigr),
 \qquad P(v)=\Ind_{\log}(\widetilde v)+Q(v),
\]
so that $P(v)\ge0$ collects all the other terms of Theorem~\ref{thm:structure}.

\begin{thmCprime}
Let $(V,\0)\subset(\CC^3,\0)$ be an isolated hypersurface surface singularity.
\begin{enumerate}
\item The minimum GSV index, excess, and defect satisfy
\[
 K=1+\mu+\Ind_{\log}(V),\qquad
 e(V)=\mu-\tau+\Ind_{\log}(V),\qquad
 D(V)=\Ind_{\log}(V).
\]
On the minimal good resolution,
\[
 \Ind_{\log}(V)=N_V-Z_V\cdot(Z_V+\kappa_X)-\chi(E).
\]
\item One has $\Ind_{\log}(V)\ge\tau-\mu$, with equality if and only if $(V,\0)$ is weighted
homogeneous; in that case $\Ind_{\log}(V)=0$. Otherwise,
\[
 \Ind_{\log}(V)\ge1-(\mu-\tau).
\]
The even-dimensional gap conjecture for surfaces is equivalent to the stronger bound
$\Ind_{\log}(V)\ge2-(\mu-\tau)$.
\item For a generic field, $Q(v)-P(v)=-D(V)\le\mu-\tau$; if $(V,\0)$ is not weighted homogeneous,
then $Q(v)-P(v)\le\mu-\tau-1$. In particular, $D(V)<0$ forces $Q(v)\ge-D(V)>0$. Thus the lift
of a generic field vanishes to order at least two along a dicritical exceptional curve $E_j$ with
$\chi(E_j)<s_j$. Equivalently, $E_j$ has genus at least two, or has genus one and at least one
invariant neighbour, or has genus zero and at least three invariant neighbours.
\item For $V_s=\{u^s+x^3+y^7+xy^5=0\}$ with $s\ge4$, the generic lift satisfies
$Q(v)\ge s-3$. Hence dicritical curves of the kind in (3) occur, and their total weight $Q(v)$ is
unbounded as $s$ varies.
\end{enumerate}
\end{thmCprime}

\begin{proof}[Proof of Theorem~C$'$]
(1) For $n=2$ we have $\chi(F)=1+\mu$, so $K=K_{\rm rad}+\mu$ and $D=K-1-\mu=K_{\rm rad}-1$.
Theorem~\ref{thm:log} gives $K_{\rm rad}-1=\Ind_{\log}(V)$, which is the third identity. The first
follows by adding $1+\mu$, and the second by subtracting $1+\tau$. The formula for $\Ind_{\log}(V)$
on the resolution is \eqref{eq:K-surface-rad}.

\medskip
(2) By Theorem~\ref{thmA}, $K\ge1+\tau$, with equality exactly for weighted homogeneous germs. By
(1) this inequality reads $\Ind_{\log}(V)\ge\tau-\mu$. In the weighted homogeneous case $\mu=\tau$,
so $\Ind_{\log}(V)=0$, and the minimum is attained by the Euler field. Otherwise $K\ge2+\tau$, that
is, $\Ind_{\log}(V)\ge1-(\mu-\tau)$. In the same way the gap conjecture $K\ge3+\tau$ reads
$\Ind_{\log}(V)\ge2-(\mu-\tau)$.

\medskip
(3) Let $v$ be generic. Splitting each term $(m_j-1)(\chi(E_j)-s_j)$ of Theorem~\ref{thm:structure}
into its positive and negative parts, that theorem becomes $D(V)=P(v)-Q(v)$, where
\begin{multline*}
 P(v)=\sum_{p\in E}\lambda_p+\#\{\text{nodes between dicritical components}\}\\
 +\sum_{E_j\ \mathrm{dic}}\Bigl[m_j\operatorname{tang}(\mathcal F,E_j)
 +(m_j-1)\max\bigl(\chi(E_j)-s_j,0\bigr)\Bigr].
\end{multline*}
Every term of $P(v)$ is nonnegative. Since $D=e-(\mu-\tau)$ and $e\ge0$, with $e\ge1$ when $(V,\0)$
is not weighted homogeneous, we get the two bounds for $Q(v)-P(v)=-D(V)$. If $D(V)<0$, then
$Q(v)\ge-D(V)>0$, so some dicritical $E_j$ has $m_j\ge2$ and $\chi(E_j)<s_j$. As
$\chi(E_j)=2-2g(E_j)$, this means $g(E_j)\ge2$, or $g(E_j)=1$ and $s_j\ge1$, or $g(E_j)=0$ and
$s_j\ge3$.

\medskip
(4) By Theorem~\ref{thm:cyclic-e12}, $D(V_s)=3-s$, so (3) gives $Q(v)\ge s-3$.
\end{proof}

\begin{remark}
For weighted homogeneous germs the Euler field gives $D=0$, so the total contribution in
Theorem~\ref{thm:structure} is zero. For the surface of type $E_{12}$ of Section~\ref{ss:E12surface}, the
central curve is dicritical with $m_0=7$, $\chi=2$ and $s_0=3$, the three arms are invariant with
$m=3,2,1$, and the formula reads $3+2+2+6\,(2-3)=1=D$.
\end{remark}

\begin{remark}\label{rem:surface-petals}
The flowers of Proposition~\ref{prop:bendixson} appear on surfaces along dicritical curves. If
$E_j=\{t=0\}$ is dicritical and the lift vanishes to order $m_j$ along it, then
$\widetilde v=t^{m_j}w$ with $w$ transverse to $E_j$ at a general point, so the normal component of
$\widetilde v$ is $t^{m_j}(c+O(t))\partial_t$ with $c\ne0$. Transversally to $E_j$ the field looks
like $t^{m_j}\partial_t$, with $2(m_j-1)$ petals, and the petals sweep along $E_j$. For a weighted
homogeneous germ the Euler field lifts to the generator of the $\CC^*$-action and has no petals.
For the $E_{12}$ surface above the central curve carries $12$ petals, and by
Theorem~C$'$(3) a negative defect forces petals along a dicritical curve with $\chi(E_j)<s_j$.
\end{remark}

\begin{remark}\label{rem:sqh}
For a semi-quasihomogeneous surface $f=f_0+f_+$, the weighted resolution can make the
negative contributions in Theorem~\ref{thm:structure} explicit. Suppose that the leading
weighted part of a logarithmic field is $hE_w$, where $E_w$ is the weighted Euler field
and $h$ has positive weighted degree $d$. Suppose also that the higher weighted terms
have strictly larger vanishing order along the central curve $C$. In a uniformizing
chart with $x_i=t^{w_i}u_i$, the leading normal component is then a nonzero multiple of
$t^{d+1}\partial_t$ at a general point of $C$. When $t$ is a local parameter on the
resolved surface, this gives a dicritical component with $m_C=d+1$ and contribution
$d(\chi(C)-s)$, where $s$ counts its nodes with invariant components. If the chart
has a generic finite stabilizer, the order must first be computed in the quotient
parameter. These hypotheses require verification for the field in question:
semi-quasihomogeneity alone does not force its leading part to be a multiple of $E_w$.
The explicit resolution calculation in Section~\ref{rem:E12-resolution} illustrates
this mechanism.
\end{remark}

\subsection{Examples of generic vector fields on surfaces}\label{ss:E12surface}

\subsubsection*{Weighted homogeneous surfaces}
Let $f$ be weighted homogeneous with weights $w_i$, and $\xi_w=\sum_iw_iz_i\partial_i$ the Euler field. By Theorem~\ref{thmA}, $\xi_w$ is generic, of index $1+\mu$. Its lift to the minimal good resolution generates the $\CC^*$-action; the total index of the lift is the Euler characteristic of its zero set, which is $\chi(E)$, in agreement with \eqref{eq:gsv-resolution}.

\subsubsection*{A surface of type $E_{12}$}
Let $V=\{f=0\}\subset(\CC^3,\0)$ with
$$
f=z^2+g=z^2+x^3+y^7+xy^5 ,
$$
so that $\mu=12$ and $\tau=11$. With $\delta_1,\delta_2$ as in Section~\ref{ss:E12curve} and $k_1=15x+21y^2$, $k_2=441y+75y^2$, the vector fields on $V$ are generated by
$$
\widehat\delta_i=\delta_i+\tfrac12k_i\,z\,\partial_z\quad (i=1,2),\qquad
\eta_x=2z\,\partial_x-g_x\,\partial_z,
$$
$$
\eta_y=2z\,\partial_y-g_y\,\partial_z,\qquad g_y\,\partial_x-g_x\,\partial_y ,
$$
with cofactors $k_1,k_2,0,0,0$. To justify generation, the cofactor
ideal has colength $\mu-\tau=1$, hence equals $(x,y,z)$. The cofactors
$k_1,k_2$ generate $(x,y)$, and the remaining generator $z$ is the
cofactor of $\frac12 f\partial_z$, which restricts to zero on $V$.
The kernel of the cofactor map is generated by the Hamiltonian fields.
Consider $v=\widehat\delta_2+\eta_x$, that is,
$$
v=(147xy+25xy^2+5y^4+2z)\,\partial_x+(-3x+63y^2+10y^3)\,\partial_y+\phi\,\partial_z ,
$$
with $\phi=\tfrac12(441y+75y^2)z-3x^2-y^5$.
Its first two components eliminate $x,z$ as
\begin{align*}
x&=21y^2+\tfrac{10}{3}y^3,\\
z&=-\tfrac{3087}{2}y^3-510y^4-\tfrac{125}{3}y^5.
\end{align*}
After substitution the last component is
$$-\frac{y^4}{12}(25y+147)(750y^2+9196y+27891).$$
It is $y^4$ times a unit, so $\ell(\OO/(v))=4$. Adding
$k_2=441y+75y^2$ gives the maximal ideal, of colength one. Formula~\eqref{eq888} therefore gives index $4-1+11=14$.

To prove genericity, write a general field on $V$ as
$$a_1\widehat\delta_1+a_2\widehat\delta_2+a_3\eta_x+a_4\eta_y+a_5\eta_{xy},$$
where $\eta_{xy}=g_y\partial_x-g_x\partial_y$, and $a_1,\ldots,a_5$ are
holomorphic germs. Let $b=a_2(0)$ and $c=a_3(0)$. If $bc\ne0$, the first
two components eliminate $x,z$ and give
$$x=21y^2+O(y^3),\qquad
z=-\frac{3087b}{2c}y^3+O(y^4).$$
The remaining component has leading term
$$-\frac{1323}{4c}(1029b^2+4c^2)y^4.$$
Thus on the nonempty open dense condition
$$bc(1029b^2+4c^2)\ne0$$
the two colengths are $4$ and $1$. Hence the minimum-index open set meets this open set, proving $K=14$.

Thus
$$K=14,\qquad 1+\tau=12,\qquad 1+\mu=13.$$
In particular $e=2$, the bound of Theorem~\ref{thmA} is not attained, and the generic holomorphic vector field is not homotopic on the link to a vector field transversal to it. For comparison, the Hamiltonian field $\eta_y$, which is the vector field considered in \cite{MS}, has index $18$: its extension $f\partial_x+\eta_y$ has cofactor $g_x$, its components generate $\langle f,z,g_y\rangle$, of colength $18$, and together with $g_x$ they generate an ideal of colength $\tau=11$, so \eqref{eq888} gives $18-11+11$.

\subsubsection*{Resolution data for the example}\label{rem:E12-resolution}
Make the weighted blowup with weights $(14,6,21)$.
The exceptional curve has equation $x^3+y^7+z^2=0$ in the weighted
projective plane. Its coarse curve is rational, and its three quotient
points have orders $2,3,7$. On the strict transform they have types
$\frac1r(1,1)$, so their resolutions contribute single curves of
self-intersection $-r$. The rational self-intersection before resolving
these points is $-42/(14\cdot6\cdot21)=-1/42$. The proper transform has
self-intersection
$$-1/42-1/2-1/3-1/7=-1.$$
Label the central curve $E_0$ and the arms of self-intersection $-2,-3,-7$ by $E_1,E_2,E_3$. The graph is a tree, and
$\chi(E)=4\cdot2-3=5$.

Here are local calculations for the vanishing orders. On the chart
$y=t^6$, $x=t^{14}u$, $z=t^{21}w$, the strict transform is
$$u^3+1+w^2+t^2u=0.$$
The lift of $\widehat\delta_2$ has components
\begin{align*}
\dot t&=\tfrac{21}{2}t^7-\tfrac12t^9u+\tfrac53t^{13},\\
\dot u&=7t^8u^2+5t^{10}+\tfrac53t^{12}u,\\
\dot w&=\tfrac{21}{2}t^8uw+\tfrac52t^{12}w.
\end{align*}
The order along the central curve is $7$. At the quotient point of
order three, $(t,u)$ are uniformizing coordinates, with scalar
$\mu_3$ action; at the quotient point of order two, use $(t,w)$ with
scalar $\mu_2$ action. Resolving a scalar quotient by the chart
$t=su$, $U=u^r$ shows that a leading term $t^7\partial_t$ gives
orders $2$ and $3$ on these arms, respectively.

At the order-seven point use $x=t^{14}$, $y=t^6u$, $z=t^{21}w$.
Then $1+u^7+w^2+t^2u^5=0$ and the scalar $\mu_7$ action is on $(t,u)$.
The leading terms are
$$\dot t=\tfrac{21}{2}t^7u+\text{higher terms},\qquad
\dot u=-3t^8+\text{higher terms}.$$
In the resolution chart $t=su$, $U=u^7$, one obtains
$$\dot s=U\bigl(\tfrac{21}{2}s^7+3s^9\bigr)+\text{higher powers of }U,$$
so this arm has order one.

All the listed generators have central order at least seven.
Indeed $\widehat\delta_1$ starts in weighted degree twelve,
$\eta_x$ has degree seven, $\eta_y$ degree fifteen, and $\eta_{xy}$
degree twenty-two; the degree-six part of $\widehat\delta_2$ is a
multiple of the Euler field and gives central order seven, while its
remaining terms have higher degree. Scalar invariance in a quotient
chart forces a vector monomial of total degree $d$ to satisfy
$d\equiv1\pmod r$. Its order on the arm is at least
$\lceil(7-1)/r\rceil$. The preceding lift attains these bounds.
Therefore
$$Z_V=7E_0+3E_1+2E_2+E_3.$$
Adjunction and the graph give $Z_V^2=-2$, $\kappa_XZ_V=0$.
The established value $K=14$ then gives $N_V=4$. In~\eqref{eq:NV-colength}, this is
$\rho(V)=3$, $\ell(R/\mathfrak c_f)=\mu-\tau=1$, $\chi(E)=5$, and $p_a(Z_V)=0$.
For the weighted homogeneous member, the Euler field has divisor $E_0$ and three residual zeros,
so $N=3$; here $\rho(V)=1$ and $\ell(R/\mathfrak c_f)=0$, while $\chi(E)=5$ and
$p_a(E_0)=0$. Thus the formula isolates the change in the module colengths that accounts for
the jump $N_V:3\to4$.

\subsection{The excess on surfaces}\label{ss:excess}
For $n=2$, by \eqref{eq888}, a generic field satisfies
\[
 e=\ell\bigl(R/(I:k)\bigr)-1,\qquad D=e-(\mu-\tau).
\]
Since $v$ has an isolated ambient zero, $A=R/I$ is an Artinian complete intersection, hence
Gorenstein, and $R/(I:k)\cong kA$ with $f\in\operatorname{ann}_A(k)=(I:k)/I$, because
$kf=\sum v_if_{z_i}\in I$.

\subsubsection*{A first-jet criterion}

\begin{prop}\label{prop:gap-first-jet}
Let $v$ have an isolated ambient zero and let its cofactor $k$
have a nonzero linear part.
If the linear part of $v$ vanishes, then
$\ell(R/(I:k))\ge3$.
\end{prop}

\begin{proof}
Set $L=(I:k)$ and $A=R/I$. Lemma~\ref{lem001}
ensures $\ell(R/L)>0$. Suppose it is one or two.
In the first case $L=\mathfrak m$; in the second, analytic coordinates
give $L=(x,y,z^2)$. Indeed, a local algebra of length two has
one-dimensional square-zero maximal ideal, so one can choose
coordinates whose images are $0,0,\bar z$, with $\bar z\ne0$.
In either case write $v=M\lambda$ for the indicated regular
sequence of generators $\lambda$ of $L$. In $A$ one has
\[
 \operatorname{ann}_A(L/I)=(\det M)A=kA.
\]
For the determinant identity, the Koszul map lifting $R/I\to R/L$
is induced by $M$. Its dual in degree three sends the generator of
$\operatorname{Ext}^3_R(R/L,R)$ to $\det M$ in
$\operatorname{Ext}^3_R(R/I,R)=A$. Change of rings identifies this
map with $\operatorname{Hom}_A(R/L,A)\hookrightarrow A$.
The second identity follows from the double-annihilator property
of the Artinian Gorenstein algebra $A$. Since the generators of
the same nonzero principal ideal in a local ring differ by a unit,
$k\in I+(\det M)$.

When $v$ has zero linear part, $I\subset\mathfrak m^2$. If
$L=\mathfrak m$, every entry of $M$ vanishes at $0$. If
$L=(x,y,z^2)$, the first two columns of $M$ vanish at $0$.
In either case $\det M\in\mathfrak m^2$, contradicting the order
of $k$.
\end{proof}

\subsubsection*{An exact family with excess two}

\begin{teo}\label{thm:cyclic-e12}
For every integer $s\ge3$, put
\[
 f_s=u^s+g(x,y),\qquad g=x^3+y^7+xy^5.
\]
Then
\[
 \mu=12(s-1),\qquad \tau=11(s-1),\qquad
 K=\tau+3,\qquad e=2,\qquad D=3-s.
\]
Thus, within the class of semi-quasihomogeneous surfaces, the defect $D=K-\chi(F)$ takes every
nonpositive integer value, although $K\ge1+\tau$ for each germ.
\end{teo}

\begin{proof}
In the Milnor algebra of $g$ the relations are
$x^2=-y^5/3$, $xy^4=-7y^6/5$ and $[g]=[y^7]/15$.
They give $y^8=xy^7=0$. The classes
$1,y,\ldots,y^7,x,xy,xy^2,xy^3$ span and are a basis because
the weighted principal part $x^3+y^7$ has Milnor number $12$.
Thus $[g]$ is a nonzero socle element. The Milnor algebra of
$f_s$ is the tensor product with $\C[u]/(u^{s-1})$, where
$[f_s]=[g]\otimes1$. This gives the stated $\mu,\tau$ and
$(J(f_s):f_s)=(x,y,u^{s-1})$.

Use the plane fields
\begin{align*}
 P_1&=(5x^2+7xy^2)\partial_x+(2xy+3y^3)\partial_y,\\
 P_2&=(147xy+25xy^2+5y^4)\partial_x+
        (-3x+63y^2+10y^3)\partial_y .
\end{align*}
Direct substitution gives $P_i(g)=k_i g$ with
$k_1=15x+21y^2$ and $k_2=441y+75y^2$. These cofactors generate
$(x,y)$. Lift $P_i$ to $\widehat P_i=P_i+(k_i u/s)\partial_u$.
Together with the three Hamiltonian fields
\[
 H_x=su^{s-1}\partial_x-g_x\partial_u,\qquad
 H_y=su^{s-1}\partial_y-g_y\partial_u,\qquad
 H_{xy}=g_y\partial_x-g_x\partial_y,
\]
the two lifts generate the fields on $V(f_s)$. Indeed,
$u^{s-1}$ is the cofactor of $(f_s/s)\partial_u$, which
restricts to zero, and the kernel of the ambient cofactor map is
the Koszul module of the three partial derivatives.

Write the extension as
$a_1\widehat P_1+a_2\widehat P_2+a_3H_x+a_4H_y+a_5H_{xy}$.
Let $b=a_2(0)$ and $c=a_3(0)$, and assume $bc\ne0$. The $y$ component
eliminates $x$ by the implicit function theorem. Give $y$ weight
$s-1$ and $u$ weight $3$. After elimination the initial term of
$x$ is $21y^2$: the $su^{s-1}$ term has weight $3(s-1)$, and
all the other terms of the solved equation have weight greater
than $2(s-1)$. The initial terms of the remaining two components
and the cofactor are respectively
\[
 3087b\,y^3+sc\,u^{s-1},\qquad (441b/s)\,yu,\qquad 441b\,y.
\]
For the middle term, the gradient terms in the $u$ component
have weight at least $4(s-1)>s+2$.

The first two initial forms are a regular sequence, so the
colength of the component ideal is
\[
 \ell(R/I)=\frac{3(s-1)(s+2)}{3(s-1)}=s+2.
\]
After adding $k$, the pair consisting of the $x$ component and
$k$ likewise has regular initial forms and colength $s-1$.
In the smooth germ $k=0$ after eliminating $x$, the parameter
is $u$, and the eliminated $y$ has weight greater than $s-1$.
Modulo $k$ the $u$ component consists of gradient terms of
weight greater than $3(s-1)$. Hence it vanishes in this quotient
of length $s-1$ and $\ell(R/(I+(k)))=s-1$.

The condition $bc\ne0$ holds on an open set of coefficient jets
and its value gives the minimum. More explicitly, take an extension
of a minimum-index field in the five generators and add a small
multiple of $\widehat P_2+H_x$. For arbitrarily small parameters,
both constants are nonzero, and the field still has the minimum
index by openness of the minimum-index locus. Thus
$K=\tau+(s+2)-(s-1)=\tau+3$. Finally the weights $(7s,3s,21)$
make $x^3+y^7+u^s$ homogeneous of degree $21s$ and give $xy^5$
weight $22s$.
\end{proof}

For $s=2$ the germ is the surface of type $E_{12}$ above, with $K=14=\tau+3$. For $s=3$ the
surface is not weighted homogeneous and has $D=0$; for $s=6$ it is $x^3+y^6+z^7+xz^5$ after
renaming, with $K=58=\mu-2$.

\begin{remark}\label{rem:subtle}
The family shows how subtle $K$ is as an analytic invariant. For fixed $s\ge4$, the germs
\[
 U_s=\{u^s+x^3+y^7=0\},\qquad V_s=\{u^s+x^3+y^7+xy^5=0\}
\]
lie in the semi-quasihomogeneous family $u^s+x^3+y^7+t\,xy^5$, which has constant Milnor number
$12(s-1)$ and is topologically trivial, as is every deformation of a weighted homogeneous germ by
terms of higher weight \cite{DG}. So $U_s$ and $V_s$ have homeomorphic links, hence the same
minimal good resolution graph \cite{Neumann}. The arithmetic genus of Wagreich, the maximum of
$p_a(D)=1+\frac12D\cdot(D+K_X)$ over positive cycles supported on $E$ \cite{Wagreich}, is computed
from the graph alone, so it is the same for both. By Laufer's formula
$\mu=12p_g+K_X^2+\chi(E)-1$ \cite{Laufer}, in which $K_X^2$ and $\chi(E)$ are read off from the
resolution graph, they also have the same geometric genus. Yet
\[
 K_{\rm rad}(U_s)=1,\qquad K_{\rm rad}(V_s)=4-s .
\]
Among the classical numerical invariants only the Tjurina number separates them:
$\tau(U_s)-\tau(V_s)=s-1$. Thus, within a single topological type, the minimal radial index can
drop by an arbitrarily large amount. For $s=3$ the two values coincide, although $V_3$ is not
weighted homogeneous (Remark~\ref{rem:homotopic-transversal}).

Here the Tjurina numbers differ, and part of the difference in $K$ reflects that $U_s$ is weighted
homogeneous while $V_s$ is not: the excess is $0$ for $U_s$ and $2$ for $V_s$. Proposition~\ref{prop:same-invariants}
below gives two surfaces, neither of them weighted homogeneous, with the same topology, $\mu$, $\tau$
and genera, and different $K$. Numerical computations as in Remark~\ref{rem:surface-data}, repeated
with two primes and with logarithmic fields of increasing degree, suggest that this is common: the
germs
\[
 x^3+y^4+z^5+yz^4,\qquad x^3+y^4+z^5+xyz^3,\qquad x^3+y^4+z^5+xy^2z^2
\]
all have $\mu=24$ and $\tau=22$, they are semi-quasihomogeneous deformations of $x^3+y^4+z^5$, and
their computed minimal indices are $K=26,27,28$.
\end{remark}

\begin{prop}\label{prop:same-invariants}
Let
\[
 g_1=x^5+y^6+x^2y^4,\qquad g_2=x^5+y^6+x^3y^3,\qquad S_i=\{z^2+g_i=0\}\subset(\CC^3,\0).
\]
The surfaces $S_1$ and $S_2$ have homeomorphic links, the same resolution graph, the same
geometric and arithmetic genera, and $\mu=20$, $\tau=18$. Nevertheless
\[
 K(S_1)=22,\qquad K(S_2)=23 .
\]
\end{prop}

\begin{proof}
Give $x,y,z$ the weights $6,5,15$, for which $x^5+y^6+z^2$ is homogeneous of degree $30$; the
monomials $x^2y^4$ and $x^3y^3$ have weights $32$ and $33$. Thus both germs are
semi-quasihomogeneous deformations of $z^2+x^5+y^6$ by terms of higher weight, so $\mu=20$, and,
by \cite{DG}, as in Remark~\ref{rem:subtle}, they have homeomorphic links, the same resolution graph and the same
genera.

\emph{Fields on the curves.} For $g_1$ put
\begin{align*}
 \delta_1&=\bigl(18xy+\tfrac45y^3+\tfrac{16}5x^2y\bigr)\partial_x+\bigl(15y^2-x^2+\tfrac{12}5xy^2\bigr)\partial_y,\\
 \delta_2&=(6xy^2+4x^3)\partial_x+(5y^3+3x^2y)\partial_y,
\end{align*}
with $k_1=90y+16xy$ and $k_2=30y^2+20x^2$, and for $g_2$
\begin{align*}
 \delta_1&=(60x^2-6y^3+9xy^3)\partial_x+(50xy+3x^2y+9y^4)\partial_y,\\
 \delta_2&=(3x^2y^2-12xy^2)\partial_x+(x^3-10y^3+3xy^3)\partial_y,
\end{align*}
with $k_1=300x+54y^3$ and $k_2=18xy^2-60y^2$. Direct substitution gives $\delta_i(g)=k_ig$, and $\det(\delta_1,\delta_2)=4g_1$, respectively
$60g_2$. By Saito's criterion $\delta_1,\delta_2$ is a basis of $\Der(-\log g)$ in both cases. The
cofactor ideal is therefore $(k_1,k_2)$, which is $(y,x^2)$ for $g_1$ and $(x,y^2)$ for $g_2$; it
has colength $2$, so $\tau(g_i)=\mu-2=18$, and $\tau(S_i)=\tau(g_i)$.

\emph{Fields on the surfaces.} As for the surface of type $E_{12}$ in Section~\ref{ss:E12surface},
the fields on $S_i$ are restrictions of
\[
 v=a_1\widehat\delta_1+a_2\widehat\delta_2+a_3\eta_x+a_4\eta_y+a_5\eta_{xy},\qquad
 \widehat\delta_j=\delta_j+\tfrac12k_jz\,\partial_z,
\]
with $\eta_x=2z\partial_x-g_x\partial_z$, $\eta_y=2z\partial_y-g_y\partial_z$,
$\eta_{xy}=g_y\partial_x-g_x\partial_y$ and arbitrary germs $a_j$: the cofactor ideal of
$z^2+g_i$ is $(k_1,k_2,z)$, since it contains this ideal and both have colength $\mu-\tau=2$; the
generator $z$ is the cofactor of $\frac12(z^2+g_i)\partial_z$, which vanishes on $S_i$; and the fields
with cofactor zero are generated by the Hamiltonian fields. The cofactor of $v$ is
$k=a_1k_1+a_2k_2$. Put $c_j=a_j(0)$.

Suppose $c_3\ne0$. The $x$ component of $v$ is $2a_3z+a_1A_1+a_2A_2+a_5g_y$, where
$\delta_j=A_j\partial_x+B_j\partial_y$, so it can be solved as $z=\zeta(x,y)$. Writing $\bar h$ for
$h(x,y,\zeta(x,y))$, we obtain $\ell(R/I)=\ell\bigl(\CC\{x,y\}/(\bar Y,\bar Z)\bigr)$ and
$\ell(R/(I,k))=\ell\bigl(\CC\{x,y\}/(\bar Y,\bar Z,\bar k)\bigr)$, where
\[
 \bar Y=\overline{a_1B_1+a_2B_2+2a_4z-a_5g_x},\qquad
 \bar Z=\tfrac12\bar k\,\zeta-\overline{a_3g_x+a_4g_y}.
\]
Here $A_2,B_2\in\mathfrak m^3$, $k_2\in\mathfrak m^2$ and $g_x,g_y\in\mathfrak m^4$.

\emph{The surface $S_1$.} The quadratic part of $\zeta$ is $-9(c_1/c_3)xy$. Hence $\bar Y$ has
quadratic part $Q=c_1\bigl(15y^2-x^2-18(c_4/c_3)xy\bigr)$, and $\bar Z$ has cubic part
$\frac12(90c_1y)\bigl(-9(c_1/c_3)xy\bigr)=-405(c_1^2/c_3)xy^2$, all its other terms being of order at
least four. If $c_1\ne0$, then $Q$ is divisible neither by $x$ nor by $y$, so the two tangent cones
have no common component and $\ell(R/I)=2\cdot3=6$. The linear part of $\bar k$ is $90c_1y$, so
$\bar k=0$ is a smooth curve $y=\varphi(x)$ with $\varphi\in\mathfrak m^2$, along which
$\bar Y=-c_1x^2+O(x^3)$ and $\bar Z=O(x^3)$. Thus $\ell(R/(I,k))=2$, and by
\eqref{eq888} the index is $6-2+18=22$ whenever $c_1c_3\ne0$.

\emph{The surface $S_2$.} Now the quadratic part of $\zeta$ is $-30(c_1/c_3)x^2$, and $\bar Y$ has
quadratic part $10c_1x\bigl(5y-6(c_4/c_3)x\bigr)$. Its Newton polygon has the vertices $x^2$, $xy$
and $y^3$, with coefficients $-60c_1c_4/c_3$, $50c_1$ and $\beta=6c_1c_4/c_3-10c_2$. If these are
nonzero, $\bar Y=0$ consists of two smooth branches: one tangent to $5y=6(c_4/c_3)x$, and one of the
form $x=\gamma y^2+O(y^3)$ with $\gamma=-\beta/(50c_1)$. The cubic part of $\bar Z$ is
$\frac12(300c_1x)\bigl(-30(c_1/c_3)x^2\bigr)=-4500(c_1^2/c_3)x^3$, so $\bar Z$ has order $3$ along
the first branch. For the second, use the weights $(2,1)$. The part of $\bar Z$ of least weight is
$\frac12(300c_1x-60c_2y^2)\cdot3(c_1/c_3)y^3-6c_4y^5$, of weight $5$; at $x=\gamma y^2$ it equals
\[
 -6c_4\Bigl(1+\frac{9c_1^2}{c_3^2}\Bigr)y^5 ,
\]
the terms in $c_2$ cancelling. Hence $\bar Z$ has order $5$ along the second branch whenever
$c_4(c_3^2+9c_1^2)\ne0$, and $\ell(R/I)=3+5=8$. The linear part of $\bar k$ is $300c_1x$, so
$\bar k=0$ is a smooth curve $x=(c_2/5c_1)y^2+O(y^3)$, along which
$\bar Y=(10c_2+\beta)y^3+O(y^4)=6(c_1c_4/c_3)y^3+O(y^4)$ and $\bar Z\in\mathfrak m^3$. Thus
$\ell(R/(I,k))=3$, and the index is $8-3+18=23$ on the open set where
$c_1c_3c_4\,\beta\,(c_3^2+9c_1^2)\ne0$.

\emph{Conclusion.} Let $v$ be a generic field on $S_i$, written as above. For $\varepsilon\in\CC$
small, $v+\varepsilon(\widehat\delta_1+\eta_x+\eta_y)$ is still generic, since the generic set is open,
and its constants are $c_1+\varepsilon$, $c_2$, $c_3+\varepsilon$, $c_4+\varepsilon$. The conditions above
are polynomials in $\varepsilon$ that are not identically zero, so they hold for some small
$\varepsilon$. Hence $K(S_1)=22$ and $K(S_2)=23$.
\end{proof}

The branches $g_1$ and $g_2$ themselves have minimal indices $-14$ and $-13$: the lift of
$\delta_1$ has order $6$, respectively $7$, on the normalization, and $K=\alpha_1-\mu$. Here the
single suspension keeps the difference; for the branches $C_1,C_2$ of Section~\ref{ss:E12curve} it
does not (Remark~\ref{rem:single-forgets}).

\begin{remark}\label{rem:surface-data}
The following values were obtained by computing the formula \eqref{eq888} for random logarithmic
fields of bounded degree over finite fields (several primes); they are numerical evidence, not proofs.
\begin{center}\footnotesize\setlength{\tabcolsep}{3.5pt}
\begin{tabular}{lcccccc}
\toprule
$f$ & $\mu$ & $\tau$ & $K$ & HF of $R/(I:k)$ & $e$ & $D$\\
\midrule
$x^3+y^4+z^8+x^2z^3$ & 42 & 39 & 42 & $(1,1,1)$ & $2$ & $-1$\\
$x^4+y^5+z^5+x^2z^3$ & 48 & 44 & 48 & $(1,2,1)$ & $3$ & $-1$\\
$x^3+y^5+z^8+y^2z^6$ & 56 & 52 & 56 & $(1,2,1)$ & $3$ & $-1$\\
$x^4+y^4+z^6+x^2z^4$ & 45 & 42 & 46 & $(1,2,1)$ & $3$ & $0$\\
$x^3+y^4+z^5+yz^4$ & 24 & 22 & 26 & $(1,2,1)$ & $3$ & $1$\\
$x^2+y^7+z^8+y^6z^3$ & 42 & 41 & 45 & $(1,2,1)$ & $3$ & $2$\\
$x^4+y^5+z^6+y^4z+x^2y^4z^2$ & 57 & 51 & 55 & $(1,2,1)$ & $3$ & $-3$\\
$x^3+y^6+z^6+y^4z^3$ & 50 & 46 & 51 & $(1,2,1,1)$ & $4$ & $0$\\
$x^4+y^4+z^7+x^3z^2$ & 54 & 48 & 54 & $(1,2,1,1,1)$ & $5$ & $-1$\\
\bottomrule
\end{tabular}
\end{center}
In a random search over about $50$ germs that are not weighted homogeneous, the value $e=1$ never
occurred. In all semi-quasihomogeneous examples the generic cofactor has a nonzero linear part.
\end{remark}

\section{Higher dimensions and the double suspension}\label{sec:periodic}

\subsection{The index on a resolution in any dimension}\label{ss:log-alldim}
Let $(V,\0)$ be an isolated singularity of dimension $n\ge1$, and let $\pi:X\to V$ be a
resolution as in Section~\ref{ss:log}: its exceptional set $E$ is a simple normal crossing divisor
and holomorphic vector fields on $V$ lift to $X$, as sections of $\Theta_X(-\log E)$. By
Theorem~\ref{thm:log}, for every holomorphic vector field $v$ on $V$ with an isolated zero,
\begin{equation}\label{eq:index-log-alldim}
 \Ind_{\rm rad}(v;V,\0)=1+\Ind_{\log}(\widetilde v;X,E),\qquad
 \Ind(\widetilde v;X)=\chi(E)+\Ind_{\log}(\widetilde v;X,E),
\end{equation}
and, if $(V,\0)$ is an isolated hypersurface or complete intersection singularity,
\begin{equation}\label{eq:gsv-log-alldim}
 \Ind_{\rm GSV}(v;V,\0)=\chi(F)+\Ind_{\log}(\widetilde v;X,E)
 =1+(-1)^n\mu+\Ind_{\log}(\widetilde v;X,E).
\end{equation}
The left-hand sides do not depend on the resolution, so neither does the logarithmic index of the
lift. For an isolated hypersurface singularity, comparing \eqref{eq:gsv-log-alldim} with
G\'omez-Mont's formulas \eqref{eq88} and \eqref{eq888} expresses this number, which is computed on
$X$ from the zeros of $\widetilde v$ along $E$, through the ambient local algebra:
\begin{equation}\label{eq:log-gomez-mont}
 \Ind_{\log}(\widetilde v;X,E)=
 \begin{cases}
 \ell\bigl(R/(f,I)\bigr)-1+(\mu-\tau), & n\text{ odd},\\[2pt]
 \ell\bigl(R/(I:k)\bigr)-1-(\mu-\tau), & n\text{ even},
 \end{cases}
\end{equation}
that is, $\Ind_{\log}(\widetilde v;X,E)=e(v)-(-1)^n(\mu-\tau)$.
Taking minima,
\[
 D(V)=K(V)-\chi(F)=\min_v\Ind_{\log}(\widetilde v;X,E)=e(V)-(-1)^n(\mu-\tau).
\]
Thus in every dimension the defect is the least logarithmic Poincar\'e--Hopf index of a lifted
holomorphic field on a resolution, and the results of the previous sections become statements
about it. It vanishes for weighted homogeneous germs, where the Euler field has logarithmic index
zero. For $n$ odd and $(V,\0)$ not weighted homogeneous, Theorem~\ref{thm:odd-gap} gives
$\min_v\Ind_{\log}(\widetilde v)\ge\mu-\tau+2\ge3$, so no lifted field is free of logarithmic zeros. For $n$
even the minimum can be negative (Corollary~\ref{cor:unbounded}), and
Conjecture~\ref{conjE} reads $\min_v\Ind_{\log}(\widetilde v)\ge2-(\mu-\tau)$. For $n\ge2$ this
minimum is also the least relative degree on the link of Proposition~\ref{prop:radial-link}.

\medskip
The relative degree $d(v)=\Ind_{\rm rad}(v;V,\0)-1$ has a simple geometric reading, valid in every
dimension. Let $L=V\cap\mathbb S_\varepsilon$, oriented as the boundary of $V\cap\mathbb
B_\varepsilon$, and let $R$ be the orthogonal projection of the ambient position vector $z$ onto $T_zV$
(on the smooth part), which points outward along $L$. Thus $2R$ is the
Riemannian gradient of $|z|^2$ for the induced Euclidean metric. On $L$ the tangent bundle of $V$ splits as
\[
 TV|_L=\CC R\oplus\mathcal H,
\]
where $\mathcal H=TL\cap J\,TL$ is the field of complex tangent hyperplanes of the link; $\mathcal H=0$ if $n=1$. A holomorphic field $v$ with an isolated zero decomposes
accordingly as $v=c\,R+w$, with $w$ in $\mathcal H$ and
\[
 c=\frac{\langle v,z\rangle}{|R|^2},\qquad \langle v,z\rangle=\sum_jv_j(z)\,\bar z_j ,
\]
since $\langle v,R\rangle=\langle v,z\rangle$ for $v$ tangent to $V$. Thus $v$ points out of the
ball where $\operatorname{Re}c>0$, is tangent to $L$ where $\operatorname{Re}c=0$, and points
straight into the singular point, $v=-\lambda R$ with $\lambda>0$, exactly where $w=0$ and $c<0$.
Let $\Sigma(v)=\{w=0\}\subset L$ be the set of points where $v$ is a complex multiple of the radial
field. After an arbitrarily small smooth perturbation of $v$ along $L$, used only to compute
the relative degree and not assumed holomorphic, the section $w$ is transverse to zero, so $\Sigma(v)$ is a closed curve, and $c$ does not vanish on it
because $v$ does not vanish on $L$. Orient $\Sigma(v)$ so that a positive tangent vector followed by
a complement mapped by $dw$, with the complex orientation of $\mathcal H$, gives the orientation of
$L$.

\begin{lettered}[Generic fields seen on the link]{F}\label{thmF}\label{rem:winding}
Let $(V,\0)$ be an isolated singularity of dimension $n\ge1$ and $v$ a holomorphic vector field on
$V$ with an isolated zero. Then
\[
 \Ind_{\rm rad}(v;V,\0)-1=\sum_{\gamma\subset\Sigma(v)}\operatorname{wind}\bigl(c|_\gamma,0\bigr),
\]
the algebraic number of turns of $c$ around the origin along the oriented components of $\Sigma(v)$,
with counterclockwise turns counted positively and clockwise turns negatively. Hence a field
is generic if and only if this signed winding is minimal, and its minimum is the defect:
\[
 D(V)=K_{\rm rad}(V)-1=\min_v\sum_{\gamma\subset\Sigma(v)}\operatorname{wind}\bigl(c|_\gamma,0\bigr),
\]
with $D(V)=K(V)-\chi(F)$ for isolated hypersurfaces and complete intersections.
\end{lettered}

\begin{proof}
On the collar $V\cap(\mathbb B_\varepsilon\setminus\mathbb B_\delta)\cong L\times[0,1]$, with $s$
increasing outward, the field $\bigl((1-s)+s\,c\bigr)R+s\,w$ equals $R$ on the inner boundary and
$v$ on the outer one, so the sum of the indices of its zeros is $d(v)$. Its zeros are the pairs
$(p,s)$ with $p\in\Sigma(v)$, $c(p)=-\lambda<0$ and $s=1/(1+\lambda)$. We verify the local sign using the complex orientation of $TV$ and the boundary
orientation of $L$. Since $R$ points outward, the latter is represented by
$(JR,\mathcal H)$, with $\mathcal H$ carrying its complex orientation. The
orientation of $\Sigma(v)$ is chosen so that a positive tangent $t$, followed
by normal vectors mapped by $dw$ to a positive basis of $\mathcal H$, gives
the orientation of $L$. In the oriented coordinates $(\partial_s,t)$ and $(R,JR)$, the remaining
block of the derivative at a zero is
\[
 \begin{pmatrix}
 -(1+\lambda)&s\operatorname{Re}(dc(t))\\
 0&s\operatorname{Im}(dc(t))
 \end{pmatrix}.
\]
The normal block has positive determinant by the chosen orientation, and
the displayed block has determinant
$-(1+\lambda)s\operatorname{Im}(dc(t))$. At a counterclockwise crossing
of the negative real axis, $\operatorname{Im}(dc(t))<0$, giving index $+1$;
a clockwise crossing gives index $-1$. The domain is oriented by
$(\partial_s,t,\text{normal})$, and the target by $(R,JR,\mathcal H)$,
so no additional permutation sign occurs. Summing the local indices
therefore gives the algebraic winding of $c$ along $\Sigma(v)$.
The smooth perturbation used for
transversality does not change the relative degree of the original holomorphic field. The second statement follows from
Proposition~\ref{prop:radial-link} and Theorem~\ref{thm:BGM}.
\end{proof}

\begin{remark}\label{rem:winding-words}
The winding in Theorem~\ref{thmF} is algebraic, not an unsigned count of turns.
After choosing a regular ray from the origin, let $N_+(v)$ and $N_-(v)$ denote
its positive and negative transverse crossings by $c|_{\Sigma(v)}$, respectively.
Then
\[
 \Ind_{\rm rad}(v;V,\0)-1=N_+(v)-N_-(v).
\]
A generic field minimizes this difference, not necessarily the sum $N_+(v)+N_-(v)$; the
individual counts depend on the chosen ray and perturbation, only their difference does not. Taking
the negative real axis as the ray, $N_+(v)\ge D(V)$ for every field: $v$ points straight into the
singular point, crossing counterclockwise, at least $D(V)$ times along $\Sigma(v)$.

For a weighted homogeneous germ the Euler field has
$\langle\xi_w,z\rangle=\sum_jw_j|z_j|^2>0$, so its algebraic winding is zero.
For $n=1$, $\Sigma(v)=L$, and on the $i$th branch $c$ has winding $k-1$,
where $k=\operatorname{ord}a_i$; this gives the $2(k-1)$ petals of
Proposition~\ref{prop:bendixson}. In higher dimensions the positivity
of the curve case need not persist: crossings of both signs may occur,
and in even dimension the total algebraic winding may be negative
(Corollary~\ref{cor:unbounded}). In particular, in odd dimension a
non-weighted-homogeneous hypersurface germ forces every holomorphic field to point straight into
the singular point at least $\mu-\tau+2$ times along $\Sigma(v)$, since $D(V)\ge\mu-\tau+2$.
\end{remark}

\hyphenation{lo-ga-rith-mic}

By Theorem~\ref{thm:periodic} below, the double suspension $P=f+z^2+w^2$ has the same $\mu$, $\tau$
and $K$ as $f$, and $\chi(F)$ is unchanged. Hence the minimal logarithmic index of a resolution of
$\{P=0\}$ equals that of a resolution of $\{f=0\}$, although the two resolutions are not related in
any evident way. For surfaces the minimum is computed on the minimal good resolution by
Theorem~C$'$; in dimension at least three, describing the divisorial part and the residual
logarithmic zeros of a minimizing lift remains open (see the end of Section~\ref{ss:log}).

\subsection{Double suspension}

\hyphenation{cons-tan-t}

We use two elementary facts about representatives of logarithmic fields. First, if a
field on $P=0$ has an isolated zero there, its ambient representative $\xi$ can be
replaced by $\xi+P u$, with $u$ a generic constant vector, to obtain an isolated zero
in the ambient space. Off $P=0$, this follows by choosing a regular value of
$-\xi/P$. Any positive-dimensional component of the ambient zero germ would then
have to lie in $P=0$, where the zero is isolated. Second, a sufficiently small
holomorphic perturbation of a minimizer that remains tangent to $P=0$ and vanishes
at the origin is again a minimizer. Indeed, the zero at the origin stays isolated
for small parameters, and conservation of number subtracts only the nonnegative
indices of zeros on the smooth locus (Theorem~\ref{thm:BGM}).

\begin{lettered}[Double suspension]{D}\label{thm:periodic}\label{thmD}
Let $P=f(x)+z^2+w^2$. Then $\mu(P)=\mu(f)$, $\tau(P)=\tau(f)$, and
\[
 K(P)=K(f).
\]
In particular, $D$ and $e$ are invariant under double suspension.
\end{lettered}

\begin{proof}
The Jacobian ideal of $P$ is $J(f)R'+(z,w)$, where $R'=R\{z,w\}$, and
$P\equiv f$ modulo $(z,w)$. This gives the equalities for $\mu$ and $\tau$.
Since the dimension changes by two, the Euler characteristic of the Milnor
fibre is unchanged.

For the upper bound, take a field $v$ of $f$ with an isolated ambient zero and
$v(f)=kf$. The field
\[
 v'=v+\bigl(\tfrac{k}{2}z-\lambda w\bigr)\partial_z
      +\bigl(\tfrac{k}{2}w+\lambda z\bigr)\partial_w
\]
satisfies $v'(P)=kP$. Choose $\lambda\in\C$ with
$\lambda^2+k(0)^2/4\ne0$. The two new components then generate $(z,w)$.
If $I$ is the component ideal of $v$, the component ideal of $v'$ is
$I'=IR'+(z,w)$. Quotienting by $(z,w)$ identifies $R'/(I',P)$ with
$R/(I,f)$ and $R'/(I':k)$ with $R/(I:k)$. The indices are equal by
G\'omez-Mont's formulas, so $K(P)\le K(f)$.

For the lower bound, put $y=(z,w)^T$ and write a minimizing field for $P$
as $\xi=(a,b)$, with $a$ the $x$-components and $b$ the $y$-components.
We may take the cofactor to depend only on $x$. Indeed, write its
$y$-dependent part as $2y^Tc(x,y)$ and replace $b$ by $b-Pc$.
Differentiating $\xi(P)=k(x)P$ in $y$ at $y=0$ gives
\[
 b(x,0)=C(x)\nabla f(x),\qquad C(x)=-\tfrac12 D_y a(x,0)^T .
\]
Set $S=\begin{pmatrix}0&-1\\1&0\end{pmatrix}$ and replace $b$ by $b+tSy$.
This perturbation preserves the cofactor and, for small $t$, preserves
minimality. For all but finitely many $t$, $D_y b(0,0)+tS$ is invertible,
so the equations $b+tSy=0$ define a graph $y=h_t(x)$. Writing
\[
 b(x,y)+tSy=b(x,0)+M_t(x,y)y
\]
gives $h_t=H_t\nabla f$, where $H_t=-M_t(x,h_t(x))^{-1}C(x)$. Thus
\[
 g_t=f+h_t^Th_t,\qquad
 \nabla g_t=B_t\nabla f,\qquad B_t=1+2(Dh_t)^T H_t .
\]
Here and below $1$ denotes the identity matrix of the indicated size.
At the origin, $H_t(0)=-(D_y b(0,0)+tS)^{-1}C(0)$ and
$Dh_t(0)=H_t(0)\operatorname{Hess} f(0)$. Consequently $B_t(0)$ is
rational in $t$ and tends to the identity as $|t|\to\infty$. We can choose
$t$ arbitrarily small with both matrices invertible.

Fix such a $t$ and suppress its subscript. The germ $g=f+h^Th$ is
right equivalent to $f$. To see this, set $g_s=f+s h^Th$. Then
\[
 \nabla g_s=\bigl(1+2s(Dh)^TH\bigr)\nabla f,\qquad
 \partial_sg_s=(H^Th)^T\nabla f .
\]
Choose a piecewise smooth path from $0$ to $1$ in the complex $s$-plane
avoiding the finitely many zeros of the determinant at the origin.
Along the path the two identities express $\partial_sg_s$ as $V_s(g_s)$
for a holomorphic field $V_s$ vanishing at the origin. Integrating
the field with the opposite sign along the path gives a local change
of coordinates taking $g$ to $f$.

On the graph, the remaining components $A(x)=a(x,h(x))$ satisfy
$A^T\nabla f=kg$. We may arrange that $A$ has an isolated ambient zero
by replacing it by $A+gu$ for a generic constant vector $u$, as above.
This amounts to adding $Pu$ to the $x$-components upstairs, with
the corresponding change in the cofactor and with the normal equations
unchanged. Continue to denote the resulting components and cofactor
on the graph by $A$ and $k$. The field $W=B^{-T}A$ satisfies $W(g)=kg$
and has the same component ideal as $A$. The normal equations generate
the graph ideal, so their quotient identifies $P$, the remaining components,
and the cofactor upstairs with $g$, $A$, and $k$, respectively.
The quotients in G\'omez-Mont's formulas therefore have the same lengths
for the two fields. Since $\tau(P)=\tau(g)$ and the two dimensions have
the same parity, $W$ has index $K(P)$. Right equivalence now gives
$K(f)=K(g)\le K(P)$.
\end{proof}

\begin{cor}[Transport of minimizing fields]
Let $f$ be non-weighted-homogeneous, and let $v$ be a minimizing field on $f=0$. Choose an
ambient representative with an isolated zero and write $v(f)=kf$. For any $\lambda\ne0$, the
field
\[
 \Sigma_\lambda(v)=v+\bigl(\tfrac{k}{2}z-\lambda w\bigr)\partial_z
      +\bigl(\tfrac{k}{2}w+\lambda z\bigr)\partial_w
\]
is tangent to $P=f+z^2+w^2=0$, has an isolated zero, and minimizes the GSV index there. In
particular, the suspension operator transports every chosen minimizing field, and any family of
such fields, into the minimizing stratum of the double suspension.
\end{cor}

\begin{proof}
The identity $\Sigma_\lambda(v)(P)=kP$ follows by cancellation of the two mixed terms. Since
$f$ is not weighted homogeneous, $k(0)=0$; the new $z,w$ components then have invertible linear
part $\left(\begin{smallmatrix}0&-\lambda\\ \lambda&0\end{smallmatrix}\right)$. Together with
the isolated zero of the chosen ambient representative, this gives an isolated zero upstairs.
The quotient identifications in the proof of the double-suspension theorem show that the two
GSV indices agree. Since $K(P)=K(f)$, the suspended field is minimizing.
\end{proof}

\begin{remark}
The cofactor ideal under double suspension is
\[
 (J(P):P)=(J(f):f)R'+(z,w).
\]
The last two generators are realized as cofactors by the fields
$\tfrac12P\partial_z$ and $\tfrac12P\partial_w$, whose restrictions
to $P=0$ vanish.
\end{remark}

\begin{cor}\label{cor:unbounded}
For $f_s=u^s+x^3+y^7+xy^5$, $s\ge3$, and $m\ge0$, the germ
$P_{s,m}=f_s+\sum_{j=1}^{2m}z_j^2$ satisfies
\[
 K(P_{s,m})=\tau(f_s)+3,\qquad e(P_{s,m})=2,\qquad D(P_{s,m})=3-s.
\]
In particular, in every even dimension $n\ge2$ the defect takes arbitrarily large negative values
as the germ varies, while $K\ge1+\tau$ for each germ.
\end{cor}

\begin{proof}
Theorem~\ref{thm:cyclic-e12} gives
$K(f_s)=\tau(f_s)+3$ and $\mu(f_s)-\tau(f_s)=s-1$.
Apply Theorem~\ref{thm:periodic} repeatedly.
\end{proof}

\subsection{The single suspension}

Put $P=f+z^2$ and write $N=n+1$, so that $P=0$ has dimension $N$
and $R=\C\{x_1,\ldots,x_N\}$. A minimizing field on $P=0$ admits
a representative of the form
\begin{equation}\label{eq:normal}
 \xi=a+zu+\tfrac12\bigl(kz-u(f)\bigr)\partial_z,\qquad
 a(f)=kf,\qquad a(0)=0,\qquad u(0)\ne0,
\end{equation}
where $a$, $u$, and $k$ depend only on $x$.
Indeed, normalize the cofactor of an ambient representative to $k_0(x)$
by subtracting a multiple of $P$ from its normal component. If the
tangential components are $a_0+zA(x,z)$, the tangency equation gives
$a_0(f)=k_0 f$ and normal component $\tfrac12(k_0z-A(f))$.
Weierstrass division of $A$ by $P$ gives a remainder $u+zv$.
On $P=0$, this gives \eqref{eq:normal} with $a=a_0-fv$ and
$k=k_0-v(f)$. A small perturbation by
$z u_0-\tfrac12u_0(f)\partial_z$, for a constant vector field $u_0$,
ensures $u(0)\ne0$ while preserving minimality.

Rectify $u$ to $\partial_{x_1}$ on the base and write
$a=\sum_{i=1}^N a_i\partial_{x_i}$ and $f_i=\partial_{x_i}f$ in those
coordinates. The first component is $a_1+z$, so the component equations
can be reduced by substituting $z=-a_1$. Put
\[
 g=f+a_1^2,\qquad B=(a_2,\ldots,a_N,f_1+ka_1).
\]
Call a pair $(a,u)$ admissible if $R/(B,g)$ has finite length. This is
the isolated-zero condition for the restriction of \eqref{eq:normal}
to $P=0$.

\begin{teo}[Single suspension]\label{thm:single}
If $N$ is odd, then
\[
 K(P)=\min_{(a,u)}\ell\bigl(R/(B,g)\bigr)-\tau(f).
\]
If $N$ is even, choose generic constants $c_2,\ldots,c_N,c_*$ for each pair
and put
\[
 \begin{split}
 B_c&=(a_2+c_2g,\ldots,a_N+c_Ng,f_1+ka_1+c_*g),\\
 k_c&=k+\sum_{j=2}^N c_j f_j+c_*a_1 .
 \end{split}
\]
Then
\[
 K(P)=\tau(f)+\min_{(a,u)}\ell\bigl(R/(B_c:k_c)\bigr).
\]
Both minima range over the admissible pairs.
\end{teo}

\begin{proof}
In either parity one can use an ambient representative with an isolated
zero while keeping the first component $a_1+z$. Add $c_jP$ to the
$j$th tangential component for $j\ge2$ and subtract $\tfrac12c_*P$
from the normal component. On the graph $z=-a_1$, the resulting
component ideal is $B_c$ and the cofactor is $k_c$. For a generic
choice of the constants the zero is isolated in the ambient space.
Indeed, on $g=0$ the ideal $B_c$ has the same zeros as $B$,
and off $g=0$ the zeros are discrete for a regular value of
minus the vector of generators of $B$ divided by $g$.
In odd dimension $(B_c,g)=(B,g)$, and in even dimension the
colon quotient is $R/(B_c:k_c)$. Apply G\'omez-Mont's formulas and
use $\tau(P)=\tau(f)$. The normal form \eqref{eq:normal} occurs
for a minimizer, so the minima give $K(P)$.
\end{proof}

\begin{remark}\label{rem:single-forgets}
Unlike the double suspension, the single suspension does not preserve the information carried by
$K$. The branches $C_1,C_2$ of Section~\ref{ss:E12curve} have the same topology, $\mu=24$ and
$\tau=21$, and $K(C_1)=-15\ne K(C_2)=-14$; by Theorem~\ref{thm:periodic} their double suspensions
still have different minima. Computations as in Remark~\ref{rem:surface-data}, with two primes and
logarithmic fields of increasing degree, give $K(z^2+C_1)=K(z^2+C_2)=27$. For the branches of
Proposition~\ref{prop:same-invariants}, on the other hand, the difference survives a single
suspension. Thus $K(f+z^2)$ is not a function of $K(f)$, $\mu$ and $\tau$.
\end{remark}

\begin{prop}[Comparison of the two parities]\label{prop:comparison}
Let $a(f)=kf$ have an isolated ambient zero. Write $I=(a_1,\ldots,a_N)$, with $N=n+1$, and
\[
 L=\ell(R/I),\qquad Q=\ell(R/(I,f)),\qquad M=\ell(R/(I:k)).
\]
Then
\[
 \begin{array}{ll}
 K(f+z^2)\le\tau(f)+2L-Q & \text{if } \dim\{f=0\}\text{ is odd},\\[2pt]
 K(f+z^2)\le 2L-M-\tau(f) & \text{if } \dim\{f=0\}\text{ is even}.
 \end{array}
\]
In particular, if an odd-dimensional minimizing field admits an
ambient representative with $f\in I$, then $e(f+z^2)\le e(f)$.
\end{prop}

\begin{proof}
For the even-dimensional target, first suppose that $k(0)=0$ and use
$\xi=a+(\tfrac{k}{2}z+P)\partial_z$, whose cofactor is $k+2z$.
The algebra of ambient zeros is free of rank two over $R/I$ by
Weierstrass division. Quotienting further by $k+2z$ sets $z=-k/2$
and gives $R/(I,f)$. The length of the cofactor image is therefore
$2L-Q$.
If $k$ is a unit, $f\in J(f)$, so the germ and its suspension are weighted
homogeneous. Thus $K(f+z^2)=1+\tau(f)\le\tau(f)+2L-Q$, since $1\le Q\le L$.

For the odd-dimensional target use $a+\tfrac{k}{2}z\partial_z$.
Its algebra of zeros on $P=0$ is obtained from the free rank-two
$R/I$-algebra with relation $z^2=-f$ by quotienting by $kz$.
Since $kf=a(f)\in I$, the ideal generated by $kz$ consists exactly
of $(k(R/I))z$ and has length $M$. An isolated ambient representative
can be chosen without changing this algebra on $P=0$. This gives
the second inequality. For an odd-dimensional minimizer with $f\in I$,
we have $L=Q$ and $e(f)=Q-1$, which gives the comparison.
\end{proof}

For an odd-dimensional minimizer, the existence of an ambient representative whose
component ideal contains $f$ can be expressed intrinsically in terms of its zero
scheme in the given embedding. It is equivalent to $(I,f)$ being a complete intersection
in the ambient space. Indeed, if $(I,f)$ is generated by $N$ elements,
the images of $a_i+c_i f$ span $(I,f)/\mathfrak m(I,f)$ for generic
constants $c_i$. This follows because the images of $a_1,\ldots,a_N,f$
already span this $N$-dimensional space. Nakayama's lemma then gives
$(a_i+c_i f)=(I,f)$. The converse follows because a finite-colength
ideal generated by $N$ elements in $R$ is a complete intersection.

\begin{remark}
Theorem~\ref{thm:periodic} is reminiscent of Kn\"orrer periodicity \cite{Kn}: the GSV index is a
homological index computed on $\OO_V$, and $V$ and its double suspension have equivalent stable
categories of maximal Cohen--Macaulay modules. We do not use this.
\end{remark}

\begin{cor}\label{cor:all-dim} One has
\begin{enumerate}
\item For a reduced plane curve $g$ and $m\ge0$,
$K\bigl(g+z_1^2+\cdots+z_{2m}^2\bigr)=\sum_i\alpha_i(g)-2\delta(g)$. In particular, for a branch of
multiplicity three, $K=1+2\mu-3\tau$ in every odd dimension.
\item In every dimension $n\ge1$ there are germs that are not weighted homogeneous with $e=2$:
$x^3+y^7+xy^5+\sum_{j\le n-1}z_j^2$.
\item For $n$ odd, $D\ge\mu-\tau+2$ for every germ that is not weighted homogeneous; for $n$ even,
$D$ takes arbitrarily large negative values as the germ varies (Corollary~\ref{cor:unbounded}),
while $K\ge1+\tau$ for every germ.
\end{enumerate}
\end{cor}

\begin{proof}
(1) follows from Theorems~\ref{thm:periodic}, \ref{thm:optimal-reduced-curves} and
\ref{thm:curve-exact-minima}. For (2), the curve has $e=2$ (Section~\ref{ss:E12curve}) and the
surface $z^2+x^3+y^7+xy^5$ has $e=2$ (Section~\ref{ss:E12surface}); apply
Theorem~\ref{thm:periodic}. (3) is Theorem~\ref{thm:odd-gap} and Corollary~\ref{cor:unbounded}.
\end{proof}

\section{The gap conjecture and further questions}\label{sec:questions}

We recall Conjecture~\ref{conjE}: for $n$ even and $(V,0)$ not weighted homogeneous, $e(V)\ge2$.
In odd dimension the gap is settled by Theorem~\ref{thm:odd-gap}. By Theorem~\ref{thm:periodic},
Conjecture~\ref{conjE} only depends on the class of $f$ under double suspension. The evidence is the
following.
\begin{itemize}
\item Theorem~\ref{thmE}(2) proves it whenever the linear part of a generic field has rank at most
$n-2$, or, in multiplicity at least three, satisfies $B^{n-1}=0$; for surfaces this includes the
case of zero linear part (Corollary~\ref{cor:zero-linear}).
\item Theorem~\ref{thm:cyclic-e12} shows that $e=2$ is attained in an infinite family of surfaces,
and all computed examples have $e\ge2$ (Remark~\ref{rem:surface-data}).
\item Numerically, in dimensions one to four:
\end{itemize}

\begin{center}\footnotesize\setlength{\tabcolsep}{4pt}
\begin{tabular}{llcccccc}
\toprule
germ $g$ & suspension & $n$ & $\mu$ & $\tau$ & $K$ & $e$ & $D$\\
\midrule
$x^3+y^7+xy^5$ & --- & 1 & 12 & 11 & $-8$ & $2$ & $3$\\
 & $+z^2$ & 2 & & & $14$ & $2$ & $1$\\
 & $+z^2+w^2$ & 3 & & & $-8$ & $2$ & $3$\\
 & $+z^2+w^2+u^2$ & 4 & & & $14$ & $2$ & $1$\\
$x^3+y^{10}+xy^7$ & --- & 1 & 18 & 16 & $-11$ & $4$ & $6$\\
 & $+z^2$ & 2 & & & $20$ & $3$ & $1$\\
 & $+z^2+w^2$ & 3 & & & $-11$ & $4$ & $6$\\
$x^4+y^5+z^5+x^2z^3$ & --- & 2 & 48 & 44 & $48$ & $3$ & $-1$\\
 & $+w^2$ & 3 & & & $-30$ & $13$ & $17$\\
 & $+w^2+u^2$ & 4 & & & $48$ & $3$ & $-1$\\
$x^3+y^4+z^8+x^2z^3$ & $+w^2+u^2$ & 4 & 42 & 39 & $42$ & $2$ & $-1$\\
$x^3+y^4+z^5+yz^4$ & --- & 2 & 24 & 22 & $26$ & $3$ & $1$\\
 & $+w^2$ & 3 & & & $-14$ & $7$ & $9$\\
 & $+w^2+u^2$ & 4 & & & $26$ & $3$ & $1$\\
\midrule
$x^3+y^7+xy^5+z^3+w^2$ & & 3 & 24 & 22 & $-16$ & $5$ & $7$\\
$x^3+y^7+xy^5+z^4+w^2$ & & 3 & 36 & 33 & $-26$ & $6$ & $9$\\
$x^3+y^7+xy^5+z^3+w^3$ & & 3 & 48 & 44 & $-35$ & $8$ & $12$\\
$x^3+y^3+z^3+w^4+xyzw$ & & 3 & 24 & 22 & $-4$ & $17$ & $19$\\
$x^3+y^3+z^4+w^4+xyzw$ & & 3 & 36 & 32 & $-8$ & $23$ & $27$\\
$x^3+y^4+z^4+w^4+xyzw^2$ & & 3 & 54 & 50 & $-13$ & $36$ & $40$\\
\bottomrule
\end{tabular}
\end{center}

The table combines proved values with numerical computations as in
Remark~\ref{rem:surface-data}. The first four rows and the odd-dimensional entries for
$x^3+y^{10}+xy^7$ follow from the results above. The other base values remain numerical
evidence; double suspension preserves the exact value of $K$ by Theorem~\ref{thm:periodic},
but does not by itself certify a numerically computed base value.

\hyphenation{ge-ne-ra-li-zing}

\begin{enumerate}
\item In odd dimension, $D=\mu-\tau$ for weighted homogeneous germs and $D\ge\mu-\tau+2$
otherwise, sharply (Theorem~\ref{thm:odd-gap}). A formula generalizing $D=\sum_i(\alpha_i-1)$ for curves is open; Theorem~\ref{thm:log} suggests
looking for it on a logarithmic resolution.
\item The computations suggest that a single suspension can either decrease or increase the
excess. In the table, $x^3+y^{10}+xy^7$ has proved excess $4$ and computed suspension
excess $3$, whereas $x^4+y^5+z^5+x^2z^3$ has computed excesses $3$ and $13$ before
and after suspension. These numerical comparisons remain to be proved.
Theorem~\ref{thm:single} reduces the problem to local algebra on the base.
\item For semi-quasihomogeneous surfaces, it remains to determine which resolution data and
leading terms of logarithmic fields suffice to compute $K$. Remark~\ref{rem:sqh} describes
one mechanism for negative contributions, but does not establish that the Seifert data
and the weighted order of the cofactor ideal alone determine the answer.
\item For a normal surface singularity, Theorem~\ref{thm:optimal-surface} expresses $K_{\rm rad}$
through the cycle $Z_V$ and the residual length $N_V$. These are analytic invariants: $K_{\rm rad}$
is not determined by the topology of the link. The germs $z^2+x^3+y^7$ and $z^2+x^3+y^7+xy^5$ lie in
a semi-quasihomogeneous family with constant Milnor number, so their links are homeomorphic, but
$K_{\rm rad}=1$ for the first, which is weighted homogeneous, and $K_{\rm rad}=14-12=2$ for the
second (Section~\ref{ss:E12surface}). It remains to compute
$Z_V$ and $N_V$ from the analytic structure, beyond the weighted homogeneous case, and through
Theorem~\ref{thm:normal-global} to turn them into global obstructions.
\end{enumerate}

\section{Global consequences of the local theory}\label{global}

Compact complex manifolds generally carry no nonzero holomorphic vector fields; for K\"ahler
manifolds the existence of a field with zeros has strong consequences \cite{rell}, and smooth
compact surfaces with such fields were studied in \cite{rell2}. On singular varieties these
methods do not apply directly, but the local minima of the index give obstructions of a different
kind. By Theorem~\ref{thm:BGM} these minima exist at every isolated singularity, so the obstructions
apply to any compact variety with isolated singularities.

\medskip
Let $V$ be a compact reduced complex analytic space of pure dimension $n$ with isolated
singularities. For $p\in\Sing(V)$ let $K_p=K_{\rm rad}(V,p)$ be the minimum local radial index.
Every derivation vanishes at the isolated singular points of $V$. For a global holomorphic vector
field $v$ on $V$ with isolated zeros, set
$$M(v)=\sum_{q\in V_{\rm reg},\,v(q)=0}\Ind_{\rm PH}(v,q),\qquad
 \Delta_p(v)=\Ind_{\rm rad}(v;V,p)-K_p\geq0.$$
Thus $M(v)$ counts the regular zeros with their positive multiplicities, and $\Delta_p(v)=0$
exactly when $v$ is generic at $p$.

\begin{prop}\label{prop:global-budget}
The following identity holds:
\begin{equation}\label{eq:global-budget}
\chi(V)-\sum_pK_p=M(v)+\sum_p\Delta_p(v).
\end{equation}
In particular $\chi(V)\ge\sum_pK_p$, and the difference bounds the number of regular zeros. It is
zero if and only if $v$ has no regular zeros and is generic at every singular point.
\end{prop}

\begin{proof}
The Poincar\'e--Hopf theorem for the radial index \cite{ASV,BarSeaSuw} gives
$\chi(V)=\sum_{p\in\Sing(V)}\Ind_{\rm rad}(v;V,p)+M(v)$. Subtract $\sum_pK_p$; every remaining
summand is nonnegative.
\end{proof}

When $V$ is a hypersurface with isolated singularities in a complex manifold $M$, the identity can be
written on $V$ with the GSV index. Put
$$c_n^{\rm vir}(V)=\int_Vc_n\bigl(TM|_V-\OO_M(V)|_V\bigr).$$
By the global GSV index formula \cite{BarSeaSuw,Suw2},
$c_n^{\rm vir}(V)=\chi(V)+(-1)^n\sum_p\mu_p$, and $K_{{\rm GSV},p}=K_p+(-1)^n\mu_p$. Hence
\begin{equation}\label{eq:global-budget-gsv}
c_n^{\rm vir}(V)-\sum_pK_{{\rm GSV},p}=M(v)+\sum_p\Delta_p(v).
\end{equation}

\begin{teo}\label{thm:projective-budget}
Let $V\subset\PP^{n+1}$ be a reduced hypersurface of degree $d$ with isolated singularities. If $V$
admits a global holomorphic vector field $v$ with isolated zeros, then
$$\sum_{j=1}^{n+1}\bigl[1-(1-d)^j\bigr]-\sum_pK_{{\rm GSV},p}
=M(v)+\sum_p\Delta_p(v)\geq0.$$
Equality holds precisely when $v$ has no regular zeros and is generic at every singular point.
\end{teo}

\begin{proof}
Let $h=c_1(\OO_{\PP^{n+1}}(1))$. The Euler sequence and the normal bundle $\OO_V(d)$ give
\begin{align*}
c_n^{\rm vir}(V)
&=\sum_{l=0}^n\binom{n+2}{n-l}(-1)^ld^{l+1}
=\sum_{j=1}^{n+1}\bigl[1-(1-d)^j\bigr],
\end{align*}
since $\int_{\PP^{n+1}}h^{n+1}=1$. Apply \eqref{eq:global-budget-gsv}.
\end{proof}

The defect $D_p=K_{{\rm GSV},p}-\chi(F_p)=K_p-1$ turns this into a
single statement for both parities.

\begin{cor}\label{cor:odd-projective}
Let $V$ and $v$ be as in Theorem~\ref{thm:projective-budget}. Put
$$s=\#\Sing(V),\qquad A_n(d)=n+1-\sum_{j=1}^n(1-d)^j,$$
and let $\pol(V)$ be the polar degree. Then
\begin{equation}\label{eq:polar-budget}
s+\sum_pD_p+M(v)+\sum_p\Delta_p(v)=A_n(d)+(-1)^n\pol(V).
\end{equation}
\begin{enumerate}
\item If $n$ is odd, then $D_p=\mu_p-\tau_p+e_p$ with $e_p\ge0$, and $D_p\ge3$ at every
non-weighted-homogeneous point. Hence, if $s_{\rm nw}$ counts these points,
$$s+\pol(V)+\sum_p(\mu_p-\tau_p)+\sum_pe_p+M(v)\le A_n(d),$$
and in particular $s+\pol(V)+3s_{\rm nw}+M(v)\le A_n(d)$.
\item If $n$ is even, then $D_p=0$ at weighted homogeneous points, and
$D_p\ge1-(\mu_p-\tau_p)$ at the others. If all singular points are weighted homogeneous,
$$s+M(v)\le A_n(d)+\pol(V).$$
\end{enumerate}
\end{cor}

\begin{proof}
The polar-degree formula \cite{dp} is $\pol(V)=(d-1)^{n+1}-\sum_p\mu_p$. Since
$(1-d)^{n+1}=(-1)^{n+1}(d-1)^{n+1}$, the left side of Theorem~\ref{thm:projective-budget} is
$A_n(d)+(-1)^n\bigl(\pol(V)+\sum_p\mu_p\bigr)-\sum_pK_{{\rm GSV},p}$, and
$K_{{\rm GSV},p}=1+(-1)^n\mu_p+D_p$. This gives \eqref{eq:polar-budget}. For $n$ odd,
$K_{{\rm GSV},p}=1-\tau_p+e_p$, so $D_p=\mu_p-\tau_p+e_p$; at a non-weighted-homogeneous point
$\mu_p-\tau_p\ge1$ by \cite{saito} and $e_p\ge2$ by Theorem~\ref{thm:odd-gap}. For $n$ even,
$K_{{\rm GSV},p}=1+\tau_p+e_p$, so $D_p=e_p-(\mu_p-\tau_p)$, with $e_p=0$ exactly in the weighted
homogeneous case, where $\mu_p=\tau_p$, and $e_p\ge1$ otherwise by Theorem~\ref{thmA}.
\end{proof}

In even dimension the defect can be negative (see the table in Section~\ref{sec:questions}), and
such singularities relax the budget instead of tightening it.

\begin{exe}
The budget is sharp. For the triangle $XYZ=0$ in $\PP^2$ and the field
$aX\partial_X+bY\partial_Y+cZ\partial_Z$ with $a,b,c$ distinct, the zeros are the three vertices.
At each vertex the field has invertible linear part, so it is generic by
Lemma~\ref{lem001}, and $D_p=0$. Here $s=3$, $\pol(V)=4-3=1$ and $A_1(3)=4$, so
\eqref{eq:polar-budget} reads $3+0+0+0=4-1$.

In the other direction, let $V\subset\PP^4$ be a cubic threefold whose singular points are all
ordinary double points. Then $A_3(3)=10$, $\pol(V)=16-s$ and $D_p=0$, so \eqref{eq:polar-budget}
would give $M(v)+\sum_p\Delta_p(v)=-6$. Hence such a threefold, smooth or not, admits no
holomorphic vector field with isolated zeros.
\end{exe}

\section{Compact curves and surfaces}\label{s. low}
\subsection{Compact singular curves}
A smooth compact curve carries a nonzero holomorphic vector field only if its genus is at most
one. The presence of singularities imposes additional local restrictions, and since
Section~\ref{sec:optimal-curves} treats curves in any embedding dimension, so can we.

\medskip
Let $V$ be a reduced irreducible compact singular curve. Let $T_V=\Der_{\CC}(\OO_V)$ denote its
tangent sheaf, and for each singular point use the branchwise minima $\alpha_i(V,p)$ of
\eqref{eq:branchwise-alpha}. We call $V$ locally planar when every germ embeds in a smooth complex
surface.

\begin{lema}\label{order2}
If a derivation of a singular branch $(C,\underline{0})$ lifts to a vector field of order two on its
normalization, then $\OO_{C,\underline{0}}=\CC+t^m\CC\{t\}$, where $m\ge2$ is the multiplicity. In particular
$(C,\underline{0})$ is weighted homogeneous, and it is planar only for $m=2$, the ordinary cusp $y^2=x^3$.
\end{lema}

\begin{proof}
Write the lifted field as $a(t)\partial_t$, with $\ord a=2$. If $h\in\OO_{C,\underline{0}}$ has positive
order $k$, then $a(t)h'(t)$ has order $k+1$ and belongs to $\OO_{C,\underline{0}}$. Starting with an element
of smallest positive order $m$, iteration shows that the value semigroup contains every integer
$k\ge m$ and none between $1$ and $m-1$. Its conductor is therefore $m$, so
$t^m\CC\{t\}\subset\OO_{C,\underline{0}}$, and since there are no values between $1$ and $m-1$,
$\OO_{C,\underline{0}}=\CC+t^m\CC\{t\}$. This is the monomial curve $(t^m,t^{m+1},\ldots,t^{2m-1})$, which is
weighted homogeneous and has embedding dimension $m$.
\end{proof}

\begin{teo}\label{curveglobal}
Let $V$ be a reduced irreducible compact singular curve. Suppose it admits a nonzero global
holomorphic vector field $v$. Then:
\begin{enumerate}
\item The normalization is $\PP^1$, and there are at most two points above the singular locus.
\item Every unibranch singular point of $V$ is weighted homogeneous. If $V$ is locally planar,
every singular point is weighted homogeneous.
\item If the lifted field has a double zero above a singular point, that point is the unique
singular point, it is unibranch with $\OO_{V,p}=\CC+t^m\CC\{t\}$, and there are no other zeros of
the field.
\item If there are two singular points, both are unibranch, the lift has simple zeros above them,
and the field is generic at both points and has no smooth zeros.
\end{enumerate}
Moreover, with $A(V)=\sum_{p,i}\alpha_i(V,p)$,
$$A(V)\leq2,\qquad h^0(V,T_V)\leq3-A(V)\leq2.$$
The value $h^0(V,T_V)=2$ occurs precisely when $V$ is rational with exactly one singular point,
which is unibranch with local ring $\CC+t^m\CC\{t\}$ for some $m\ge2$. If $V$ is locally planar,
this point is an ordinary cusp.
\end{teo}

\begin{proof}
Every derivation lifts to the normalization and vanishes at each point above a singularity.
Poincar\'e--Hopf on the normalization gives
$$\sum_{p,i}\ord a_{p,i}+M(v)=2-2g.$$
Since the first sum is positive, $g=0$ and there are at most two such points. All their orders are
one or two, and an order two leaves no room for other zeros. This proves (1) and (3), using
Lemma~\ref{order2}.

At a unibranch singular point the order is one or two. Order one means $\alpha_1=1$, and the point
is weighted homogeneous by Corollary~\ref{cor:curve-wh}(2); order two is Lemma~\ref{order2}. If a
singular point has two branches, both orders are one, so by Theorem~\ref{thmB} its radial index is
$1$ and, if it is planar, its GSV index is $1-\mu$. Theorem~\ref{thmA} then gives
$1-\mu\ge1-\tau$, so $\mu=\tau$ and the point is weighted homogeneous by \cite{saito}. This
proves (2). At each simple zero of the lift the order is minimal, so by Theorem~\ref{thmB} the
field is generic at a point all of whose branches have simple zeros; this proves (4).

Every global field lifts to a section of $T_{\PP^1}(-\sum\alpha_i\widetilde p_i)$, and lifting is
injective. Since $T_{\PP^1}\simeq\OO_{\PP^1}(2)$, its space of sections has dimension
$\max(0,3-A(V))$, which gives the dimension bound. If $h^0(T_V)=2$, then $A(V)=1$, so there is a
single singular branch, and the image is all of $H^0(T_{\PP^1}(-\widetilde p))$. This space
contains a field with a double zero at $\widetilde p$, so the germ is as in Lemma~\ref{order2}.
Conversely, every field $a(t)\partial_t$ with $a(\underline{0})=0$ preserves $\CC+t^m\CC\{t\}$, since it does
not lower orders. Thus on a rational curve whose only singularity is of this type, all sections of
$T_{\PP^1}(-\widetilde p)$ descend, and $h^0(T_V)=2$.
\end{proof}

An immediate obstruction is that one non-weighted-homogeneous unibranch singularity, or, on a
locally planar curve, any non-weighted-homogeneous singularity, rules out every nonzero global
holomorphic vector field on an irreducible compact curve. Another consequence is that a
non-weighted-homogeneous branch has $\alpha_1\geq3$: order one forces weighted homogeneity by
Corollary~\ref{cor:curve-wh}, and order two forces the monomial curve of Lemma~\ref{order2}.

\medskip
The planarity assumption in (2) is used only at a point with two branches. There the global field
has eigenvalues $\lambda$ and $-\lambda$ at the two preimages, as does every field on $\PP^1$ with
two simple zeros, so the global field does not satisfy the half-plane criterion of
Corollary~\ref{cor:curve-wh}(1).
We do not know whether such a point must be weighted homogeneous when it is not planar.

For instance, the curves $x^pz^{q-p}=y^q$ in $\PP^2$, with $p<q$ coprime and $p,q-p\geq2$, are
invariant under a $\CC^*$-action and have two singular points.

\begin{remark}
When the field is the restriction of a global field on an ambient smooth surface, the rationality
conclusion is also related to Miyaoka's theorem \cite{miy}; see \cite[p.~82]{Bru}.
\end{remark}

\subsection{Compact singular surfaces}
Let $V$ be a compact normal complex surface, and let $X\to V$ be the resolution of
Section~\ref{ss:surfaces}, with exceptional divisor $E_p$ over $p\in\Sing(V)$ and canonical class
$K_X$. Let $Z_p=Z_{V,p}$ and $N_p=N_{V,p}$ be the cycle and residual length of
Theorem~\ref{thm:optimal-surface} at $p$.

\begin{teo}\label{thm:normal-global}
If $V$ admits a global holomorphic vector field $v$ with isolated zeros, then
\begin{equation}\label{eq:normal-budget}
12\chi(\OO_X)-K_X^2+\sum_p\bigl[Z_p\cdot(Z_p+K_X)-N_p\bigr]=M(v)+\sum_p\Delta_p(v).
\end{equation}
In particular
$$12\chi(\OO_X)\ge K_X^2-\sum_pZ_p\cdot(Z_p+K_X)+\sum_pN_p+M(v),$$
with equality if and only if $v$ has no regular zeros and is generic at every singular point.
\end{teo}

\begin{proof}
Since $E_p$ is connected, $\chi(V)=\chi(X)-\sum_p\bigl(\chi(E_p)-1\bigr)$, and by
\eqref{eq:K-surface-rad}, $K_p=1-\chi(E_p)-Z_p\cdot(Z_p+K_X)+N_p$, where the intersection numbers
may be computed on $X$ because $Z_p$ is supported on $E_p$. Hence
$$\chi(V)-\sum_pK_p=\chi(X)+\sum_p\bigl[Z_p\cdot(Z_p+K_X)-N_p\bigr].$$
Noether's formula $12\chi(\OO_X)=K_X^2+\chi(X)$ holds on every compact complex surface. Substitute in
\eqref{eq:global-budget}.
\end{proof}

When $v$ is generic at every singular point, \eqref{eq:normal-budget} is the count of zeros of the
lifted field after its divisorial part is removed. The lift is $\widetilde v=s_Dw$ with
$D=\sum_pZ_p$ and $w$ a section of $T_X(-D)$, and
$c_2\bigl(T_X(-D)\bigr)=\chi(X)+D\cdot(D+K_X)$ counts the zeros of $w$: the $N_p$ points on the
exceptional divisors and the $M(v)$ regular zeros. The local theory says that no global field can do
better than $Z_p$ and $N_p$ at each point.

\medskip
For hypersurfaces one can stay on $V$. Write $p_g^{\omega}(V)=h^2(V,\OO_V)$ and
$q(V)=h^1(V,\OO_V)$. The superscript emphasizes that this is the dualizing-sheaf genus
$h^0(V,\omega_V)$, not necessarily the geometric genus of a resolution when the singularities are
nonrational. Set $\kappa=K_V^2$ for the square of the Cartier canonical class.

\begin{teo}\label{surfaceglobal}
Let $V$ be a connected compact reduced hypersurface with isolated
singularities in a smooth compact complex threefold. If $v$ is a
global holomorphic vector field with isolated zeros, then
\begin{equation}\label{surfacebudget}
12(1-q(V)+p_g^{\omega}(V))
=\kappa+\sum_pK_{{\rm GSV},p}+M(v)+\sum_p\Delta_p(v).
\end{equation}
In particular,
\begin{align*}
p_g^{\omega}(V)-q(V)
&\geq\left\lceil\frac{\kappa+\sum_pK_{{\rm GSV},p}+M(v)}{12}\right\rceil-1\\
&\geq\left\lceil\frac{\kappa+\sum_p(1+\tau_p)+M(v)}{12}\right\rceil-1.
\end{align*}
If $\kappa\geq0$ and $v$ has at least one zero, then
$p_g^{\omega}(V)\geq q(V)$.
\end{teo}
\begin{proof}
The hypersurface is normal and Gorenstein. Adjunction identifies
$\omega_V$ with $\omega_M(V)|_V$.
The exact sequence
$0\to\OO_M(-V)\to\OO_M\to\OO_V\to0$ and analytic
Hirzebruch--Riemann--Roch on the smooth compact manifold $M$ give
$$\chi(\OO_V)=\frac{\kappa+c_2^{\rm vir}(V)}{12}.$$
For projective varieties this also follows from Baum--Fulton--MacPherson Riemann--Roch \cite{BFM}. Substitute \eqref{eq:global-budget-gsv}.
Since $p_g^{\omega}-q$ is integral, the ceiling inequalities follow.
If there is a singular point, $K_{{\rm GSV},p}\geq1+\tau_p\geq2$; if all zeros
are smooth, $M(v)\geq1$. Thus the right side of
\eqref{surfacebudget} is positive when $\kappa\geq0$ and there is
at least one zero.
\end{proof}

Only the numerical condition $K_V^2\geq0$ is needed for the last
conclusion. When the ambient manifold is K\"ahler, nefness of $\omega_V$ is sufficient; nefness of $\OO_M(K_M+V)$ on the ambient space is a stronger sufficient condition. Once $K_V^2\geq0$
is imposed directly, K\"ahlerness is not used in this proof.

\begin{cor}[The equality range]
Under $\kappa\geq0$, if $p_g^{\omega}(V)=q(V)$, then
$$\kappa+\sum_pK_{{\rm GSV},p}+M(v)\leq12,$$
and hence
$$\kappa+\#\Sing(V)+\sum_p\tau_p+M(v)\leq12.$$
In particular there are at most six singular points. A singularity
analytically equivalent to $z^2+x^3+y^7+xy^5=0$ forces
$p_g^{\omega}(V)\geq q(V)+1$, because its local minimum is $14$.
\end{cor}
\begin{proof}
Set $p_g^{\omega}=q$ in \eqref{surfacebudget} and discard the
nonnegative $\Delta_p(v)$. Each singular point contributes at least two.
The $E_{12}$ assertion follows from $\sum_pK_{{\rm GSV},p}\geq14$ and the ceiling
bound.
\end{proof}

\begin{remark}\label{rem:nonempty-zero}
The nonemptiness condition cannot be dropped if smooth surfaces are
included. Let $A$ be an abelian surface, $B$ an elliptic curve,
$M=A\times B$, and $V=A\times\{b\}$. The adjoint bundle
$\OO_M(K_M+V)=\operatorname{pr}_B^*\OO_B(b)$ is nef. A nonzero
translation field on $A$ has no zeros, but
$p_g^\omega(A)=1<2=q(A)$. If $V$ is singular, the zero set of every
holomorphic vector field is automatically nonempty.
\end{remark}

\subsection*{Acknowledgments}
Both authors are grateful to M. Corr\^ea for pointing out some corrections needed in the original
version of the article.

\subsection*{AI declaration}
Claude Opus 5.5 (Anthropic) was used to draft and revise parts of the manuscript, to check
arguments, and to design and run the computer experiments (computations of G\'omez-Mont's formulas
for random logarithmic fields over finite fields) behind the numerical evidence. OpenAI Codex
(GPT-6) was used to review the manuscript and its arguments and to revise the English exposition
and the \LaTeX\ source. The authors checked all the mathematics and take full responsibility for
the content.

\end{document}